\documentclass[reqno]{myart}
\usepackage{mathtools}
\usepackage{stmaryrd}% only for trianglelefteqslant
\usepackage[
  warnings-off={mathtools-colon,mathtools-overbracket}]{unicode-math}
\let\dfn\emph% distinguish definition emph from other emph (like `NOT')
\setoperatorfont{\textup}

\AfterEndPreamble{\let\phi\varphi\let\epsilon\varepsilon\let\geq\geqslant\let\leq\leqslant}

\AfterEndPreamble{\renewcommand{\setminus}{\char"2216}}
\def\comp{\vysmwhtcircle}% composite
\def\defeq{\coloneq}% TODO: defeq :=
\newcommand{\blank}{\mathord{-}}
\AfterEndPreamble{%
  \let\div\relax
  \def\div{/}
}
\newcommand{\uL}{\symup{L}}
\newcommand{\uR}{\symup{R}}

\newcommand{\NN}{\symbb{N}}
\newcommand{\ZZ}{\symbb{Z}}
\newcommand{\CC}{\symbb{C}}

\newcommand{\cA}{\symcal{A}}
\newcommand{\cB}{\symcal{B}}
\newcommand{\cC}{\symcal{C}}
\newcommand{\cD}{\symcal{D}}

\newcommand{\cF}{\symcal{F}}
\newcommand{\cG}{\symcal{G}}
\newcommand{\cM}{\symcal{M}}
\newcommand{\cN}{\symcal{N}}

\newcommand{\cS}{\symcal{S}}
\newcommand{\cT}{\symcal{T}}

\newcommand{\cX}{\symcal{X}}
\newcommand{\cY}{\symcal{Y}}

\newcommand{\id}{\textup{id}}

\DeclareMathOperator{\coker}{coker}

\DeclareMathOperator{\im}{im}
\let\lim\relax
\DeclareMathOperator*{\lim}{lim}

\DeclareMathOperator{\Hom}{Hom}
\DeclareMathOperator{\End}{End}
\DeclareMathOperator{\Ext}{Ext}
\DeclareMathOperator{\Tor}{Tor}
\DeclareMathOperator{\gHom}{\underline{Hom}}

\DeclareMathOperator{\GKdim}{GK dim}
\DeclareMathOperator{\add}{add}

\DeclareMathOperator{\proj}{proj}
\DeclareMathOperator{\Proj}{Proj}
\DeclareMathOperator{\qgr}{qgr}
\DeclareMathOperator{\charOp}{char}
\DeclareMathOperator{\GL}{GL}

\DeclareMathOperator{\Ann}{Ann}

\DeclareMathOperator{\ev}{ev}
\DeclareMathOperator{\Pic}{Pic}
\DeclareMathOperator{\Ind}{Ind}
\DeclareMathOperator*{\dlim}{\underrightarrow{\lim}}
\NewDocumentCommand\pair{O{\blank, \blank}}{\langle#1\rangle}

\let\mod\relax
\newcommand{\mod}[1]{\textup{mod-}#1}

\DeclareMathOperator{\qmodop}{qmod}
\newcommand{\qmod}[2]{%
\mathchoice
  {\mod{#1}/#2}
  {\mod{#1}/#2}
  {#1/#2}
  {#1/#2}}

\newcommand{\qmodG}[2][]{\qmod{#2}{\cG_{#1}}}
\newcommand{\qmodF}[2][]{\qmod{#2}{\cF_{#1}}}

\newcommand{\Mod}[1]{\textup{Mod-}#1}

\newcommand{\QQMod}[2]{%
  \mathchoice
    {\Mod{#1}/#2}
    {\Mod{#1}/#2}
    {#1/#2}
    {#1/#2}}

\newcommand{\QModG}[2][]{\QQMod{#2}{\cG_{#1}}}
\newcommand{\QModGG}[1]{\QQMod{#1}{\cG'}}
\newcommand{\QModF}[2][]{\QQMod{#2}{\cF_{#1}}}
\newcommand{\QModFF}[1]{\QQMod{#1}{\cF'}}
\newcommand{\QModFT}[1]{\QQMod{#1}{\cF(T_{#1})}}
\DeclareMathOperator{\depth}{depth}
\newcommand{\dep}[2][]{\depth^{#1}(#2)}

\newcommand{\depF}[2][]{\dep[\cF_{#1}\!]{#2}}
\newcommand{\taut}{\tau^t}
\NewDocumentCommand{\tautR}{s}{%
  \IfBooleanTF#1%
    {\taut_R}% If a star is seen
    {\taut}}% If no star is seen
\NewDocumentCommand{\tautA}{s}{%
  \IfBooleanTF#1%
    {\taut_A}% If a star is seen
    {\taut}}% If no star is seen

\NewDocumentCommand{\etatA}{s}{%
  \IfBooleanTF#1%
    {\eta^t_*}% If a star is seen
    {\eta^t}}% If no star is seen
\makeatletter
\newcommand{\raisemath}[1]{\mathpalette{\raisem@th{#1}}}
\newcommand{\raisem@th}[3]{\raisebox{#1}{$#2#3$}}
\makeatother
\newcommand{\submodule}{\leq}
\newcommand{\supmodule}{\geq}

\makeatletter
\unless\ifbibitem@option
\usepackage[%
  backend=biber,
  style=alphabetic,sorting=nyt,doi=false,url=false,
  giveninits=true,
  minbibnames=3,minalphanames=3,maxalphanames=4,maxbibnames=4,
  citetracker=true
  ]{biblatex}
\defbibheading{bibliography}[References]{%
  \section*{#1}}

\fi
\newcommand\myprintbibliography{%
  \singlespacing%
  \sloppy\hbadness 10000\relax%
  \unless\ifbibitem@option%
    \printbibliography%[heading=bibintoc]%
  \else%
    \input{ref/refs.tex}%
  \fi
  \fussy}
\makeatother

\usepackage{tikz-cd}
\usetikzlibrary{nfold}
\tikzcdset{
  arrow style=Latin Modern,% TODO: it was tikz
  diagrams={>={Straight Barb[scale=0.8]}}% I've forgot what this means :(
}
\tikzset{
  no line/.style={draw=none,
    commutative diagrams/every label/.append style={/tikz/auto=false}},
  from/.style args={#1 to #2}{to path={(#1)--(#2)\tikztonodes}}}
\tikzcdset{
  left adj/.style={rightarrow,shift left=2,},
  right adj/.style={leftarrow,shift right=2,},
  left adj rev/.style={leftarrow,shift left=2,},
  right adj rev/.style={rightarrow,shift right=2,},
}
\usepackage{enumitem}
\setlist{nosep}
\setlist{leftmargin=0pt,labelwidth=0pt,
labelsep=\fontdimen2\font,itemindent=!,align=left}
\setlist[1]{labelindent=\parindent}
\setlist[enumerate,1]{label=(\arabic*)}
\setlist[2]{labelindent=2\parindent}
\setlist[enumerate,2]{label=(\roman*)}
\usepackage{varioref}
\usepackage[draft=false]{hyperref}% maybe draft=false
\usepackage[capitalise,nameinlink,noabbrev]{cleveref}
\usepackage{crossreftools}
\crefname{enumi}{}{}
\crefname{diagram}{diagram}{diagrams}
\crefname{equation}{}{}
\AtBeginDocument{\let\ref\relax
\let\ref\Cref}
\makeatletter
\newcommand{\enumilabel}[2]{%
  \phantomsection
  #1\protected@edef\cref@currentlabel{[enumi][][]#1}\label[enumi]{#2}%
}
\makeatother
\newtheorem{theorem}{Theorem}[section]
\newtheorem{lemma}[theorem]{Lemma}
\AddToHook{env/lemma/begin}{\crefalias{theorem}{lemma}}
\newtheorem{proposition}[theorem]{Proposition}
\AddToHook{env/proposition/begin}{\crefalias{theorem}{proposition}}
\newtheorem{corollary}[theorem]{Corollary}% added by me
\AddToHook{env/corollary/begin}{\crefalias{theorem}{corollary}}
\theoremstyle{definition}
\newtheorem{definition}[theorem]{Definition}
\AddToHook{env/definition/begin}{\crefalias{theorem}{definition}}
\newtheorem{example}[theorem]{Example}
\AddToHook{env/example/begin}{\crefalias{theorem}{example}}

\theoremstyle{remark}
\newtheorem{remark}[theorem]{Remark}
\AddToHook{env/remark/begin}{\crefalias{theorem}{remark}}

\AddToHook{env/hypothesis/begin}{\crefalias{theorem}{hypothesis}}
\numberwithin{equation}{section}
\title[Quotient Morita Theory with Applications]
  {Quotient Morita Theory with Applications}

\author[Chen X.-Y.]{Xuyang Chen}
\address{School of Mathematical Sciences,
  Fudan University, Shanghai 200433, China}
\email{xuyangchen21@m.fudan.edu.cn}
\author[Li H.-N.]{Haonan Li}
\address{School of Mathematics,
  Shanghai University of Finance and Economics,
  Shanghai 200433, China}
\email{lihaonan@mail.shufe.edu.cn}
\thanks{Haonan Li is supported by the Fundamental Research Funds for the Central Universities. Quanshui Wu is supported by the NSFC (Grant No. 12471032).}

\author[Wu Q.-S.]{Quanshui Wu}
\address{School of Mathematical Sciences,
  Fudan University, Shanghai 200433, China}
\email{qswu@fudan.edu.cn}
\date{September 30, 2026}


\begin{document}
  \begin{abstract}
    % !TeX root = ../main.tex

We study equivalences between quotient categories of module categories associated with Gabriel topologies.
We establish necessary and sufficient conditions for a Morita context
to induce an equivalence between quotient categories
and characterize such equivalences in terms of Morita contexts
after passing to rings of quotients.
Motivated by Artin--Zhang's noncommutative Serre theorem,
we introduce a notion of ampleness relative to a Gabriel topology
and use it to characterize quotient categories of finitely generated modules under suitable noetherian hypotheses.
As an application, we establish a sufficient condition under which the noncommutative Auslander theorem holds for actions of finite-dimensional
Hopf algebras on AS-regular algebras without assuming that the Hopf algebras are semisimple.

  \end{abstract}

  \maketitle

  \tableofcontents
  % !TeX root = ../main.tex

\section{Introduction}\label{sec:intro}

% 我们先说商范畴以及商范畴之间的等价很有用
% 然后我们再说AZ定理 ample条件
% 引入对一般商范畴的F-progenerator
% 然后我们叙述定理3.9的时候 强调这个拓扑可以取成比如说好的维数函数诱导的拓扑
% 然后我们叙述和PG定理的比较
% 之后我们引入Morita理论 说前面发展的内容派上本质用场了

If $A$ is a standard graded commutative algebra, then the quotient category $\qgr(A)$ of finitely generated graded $A$-modules modulo finite-dimensional modules is equivalent to the category of coherent sheaves on $\Proj(A)$ \cite{serreFaisceauxAlgebriquesCoherents1955}. This equivalence is known as Serre's theorem.
As a foundational step in developing a theory of noncommutative projective schemes,
Artin and Zhang \cite{artinNoncommutativeProjectiveSchemes1994}
extended Serre's theorem to noncommutative noetherian graded algebras.
The graded quotient category $\qgr(A)$ is therefore regarded as a fundamental object of study in noncommutative projective geometry.

% 引入商范畴的话可以从非交换解消开始提, 不过最开始要先提非交换McKay对应, 然后再提QWZ, 再之后才是别的
% 于是我们来研究一般的商范畴等价
% 这之后直接给出Morita理论的定理

% 这里要从MU的ample group action开始引, 然后引到Hopf版本
Recently, in the study of noncommutative (crepant) resolutions of graded singularities, particularly quotient singularities, equivalences between quotient categories have become an important tool for establishing noncommutative analogues of Auslander's theorem and the McKay correspondence
\cite{ueyamaNoncommutativeGradedAlgebras2015,moriAmpleGroupAction2016,chanMcKayCorrespondenceSemisimple2018,chanMcKayCorrespondenceSemisimple2019,baoPertinencyHopfActions2019,baoNoncommutativeAuslanderTheorem2018}.
In particular, \cite{baoPertinencyHopfActions2019} considers
the more general category $\qmodop_n (A)$, defined as the quotient category of finitely generated $A$-modules modulo those of Gelfand--Kirillov dimension $\leq n$.
Furthermore, \cite{qinNoncommutativeQuasiresolutions2019} uses equivalences between the quotient categories $\qmodop_n(A)$ and $\qmodop_n (B)$ to establish a version of Van den Bergh's noncommutative crepant resolution
\cite{vandenberghNoncommutativeCrepantResolutions2004,vandenberghThreedimensionalFlopsNoncommutative2004} for noncommutative algebras,
thereby motivating further research on noncommutative resolutions of noncommutative graded singularities
\cite{hePreresolutionsNoncommutativeIsolated2022,liNoncommutativeResolutionsASGorenstein2026,liNoncommutativeResolutionsNoncommutative2026}.

These considerations motivate a re-examination of quotient categories in a more general setting.
Let $A$ be a ring equipped with a Gabriel topology $\cF$. 
Let $\Mod A$ denote the category of (right) $A$-modules, and let $\mod A$ denote the category of finitely generated $A$-modules.
Recall that an $A$-module $M$ is $\cF$-torsion if every element of $M$ is annihilated by some ideal in $\cF$.
Denote by $\QModF A$ the quotient category of $\Mod A$ modulo $\cF$-torsion modules. 
If $A$ is right noetherian, there is likewise a well-defined quotient category $\qmodF A$ of $\mod A$ modulo finitely generated $\cF$-torsion modules. 
For further details, see \ref{Quotient Categories and Gabriel Topologies}.

A graded analogue is obtained by passing to a graded ring endowed with a graded Gabriel topology
\cite[Sects.\@ II.9 and II.10]{nastasescuGradedRingTheory1982}.

Returning to Artin--Zhang's generalization of Serre's theorem \cite[Thm.\@ 4.5]{artinNoncommutativeProjectiveSchemes1994}, a central ingredient is a noncommutative notion of ample line bundles.
Under an appropriate interpretation, this notion from (noncommutative) projective geometry extends naturally to general quotient categories.

Let $\cC$ be an abelian category.
Suppose that $\cM\in\cC$, $A=\End_{\cC}(\cM)$
and $\cF$ is a Gabriel topology on $A$.
The following two conditions interpret ampleness in the sense of Artin--Zhang.
See the beginning of \ref{subsec:ample-explanation} for precise explanations.
\begin{enumerate}
  \item
    For any $I\in\cF$ there exist $f_1, \dotsc, f_r\in I$
    such that the canonical morphism $(f_i)\colon \cM^{\oplus r}\to\cM$ induced by the $f_i$'s is an epimorphism.
  \item
    For any epimorphism
    $\cX\rightarrow\cY$ in $\cC$,
    $\coker\bigl(\Hom_{\cC}(\cM,\cX)\to\Hom_{\cC}(\cM,\cY)\bigr)$
    is an $\cF$-torsion $A$-module.
\end{enumerate}

If $\cM$ satisfies the preceding two conditions and is a finite generator in $\cC$ (i.e., every object $\cX\in\cC$ admits an epimorphism $\cM^{\oplus r}\to\cX$ for some $r\in\NN$), then $\cM$ is called $\cF$-ample.
When $\cF$ is the trivial topology, an $\cF$-ample object is projective and is a finite generator in $\cC$.
The corresponding version
in a cocomplete category is called a fully $\cF$-ample object (\ref{dfn:progen-large}).

The first result in this paper uses the notion of $\cF$-ample objects to characterize when an abelian category is equivalent to a quotient category.

Let $\pi\colon\Mod A\to \QModF A$ be the canonical quotient functor,
and $\omega$ be the right adjoint of $\pi\colon\Mod A\to \QModF A$.
Recall that an $A$-module $M$ is $\cF$-closed if the natural morphism $M \to \omega\pi (M)$ is an isomorphism.

\begin{theorem}[{{\ref{prop:pi-A-ample,thm:morita-i-finite}}}]
  \label{main-thm-serre}
  Let $\cC$ be an abelian category.
  Suppose that $\cM\in\cC$ and that $A=\End_{\cC}(\cM)$ is right noetherian.
  If $\cF$ is a Gabriel topology on $A$ such that $\cM$ is $\cF$-ample,
  then $A$ is $\cF$-closed, and there is an equivalence of categories $\cC \simeq \qmodF{A}$.

  Conversely, let $A$ be a right noetherian ring. Suppose that $\cF$ is a Gabriel topology on $A$ such that 
  $A$ is $\cF$-closed. Then $\pi(A)$ is $\cF$-ample in $\qmodF{A}$.
\end{theorem}

The Popescu--Gabriel theorem \cite{popescuCaracterisationCategoriesAbeliennes1964} states that a Grothendieck category is equivalent to a quotient category of modules for a suitable Gabriel topology.
In \ref{rem:pg}, we explain how this theorem is related to \ref{main-thm-serre}.
To apply \ref{main-thm-serre}, one may first choose a suitable Gabriel topology $\cF$ and then verify that $\cM$ is $\cF$-ample.
For example, if $A$ is a noetherian graded algebra
and $\cF$ is the graded Gabriel topology for which the graded $\cF$-torsion modules are precisely those graded $A$-modules in which every element is annihilated by some $A_{\geq n}$,
then the graded version of \ref{main-thm-serre} yields a characterization of $\qgr(A)$ analogous to that of the aforementioned Artin--Zhang theorem; see \ref{cor:AZ-thm}.

\iffalse
On the other hand, \ref{main-thm-serre} is closely related to
the classical theorem of Popescu and Gabriel \cite{popescuCaracterisationCategoriesAbeliennes1964},
which asserts that an abelian category $\cC$ is equivalent to $\QModF A$
for \emph{some} ring $A$ and \emph{some} Gabriel topology $\cF$ on $A$
if and only if $\cC$ is a Grothendieck category.
We present a new perspective on this result through the notion of $\cF$-ampleness. 

\begin{theorem}[{{\ref{thm:groth-cat-weakest-top}}}]\label{main-thm-PG}
  Let $\cC$ be a Grothendieck category, $\cM$ a generator of $\cC$, and 
  $A=\End_{\cC}(\cM)$.  Suppose that $\cF$ is a Gabriel topology on $A$.
  Then the following are equivalent.
  \begin{enumerate}
    \item $\pi\comp\Hom_{\cC}(\cM,\blank)\colon\cC\to\QModF A$
      gives an equivalence of categories.
    \item $\cF$ is the strongest Gabriel topology
      such that $\Hom_{\cC}(\cM,\cX)$ is $\cF$-closed for $\cX\in\cC$.
    \item $\cF$ is the strongest Gabriel topology
      such that $\cM$ is an $\cF$-admissible generator.
    \item $\cF$ is the weakest Gabriel topology such that
      $\cM$ is $\cF$-compact and $\cF$-projective.
    \item $\cM$ is fully $\cF$-ample.
  \end{enumerate}
\end{theorem}
\fi

We next study equivalences between quotient categories induced by Morita contexts.
Criteria for a Morita context to induce such an equivalence have been developed in studies of the noncommutative Auslander theorem and noncommutative resolutions of singularities
\cite{moriAmpleGroupAction2016,baoPertinencyHopfActions2019,liNoncommutativeResolutionsASGorenstein2026}.
For the Morita context to induce the desired equivalence,
the tensor functors defined by its bimodules must descend to the quotient categories;
equivalently, both the tensor functors and their first Tor functors must send torsion modules to torsion modules (\ref{prop:induced-functors,def:compatible}).
Such compatibility conditions serve as prerequisites for establishing the equivalences in the cited works.
Our criterion does not require the condition on the first Tor functors as a separate hypothesis:
under the hypotheses below,
these functors automatically send torsion modules to torsion modules
(see \ref{enum:tor-torsion}).

We use the standard notion of Morita contexts $(A,B,M,N,\xi,\zeta)$,
where $\xi \colon N \otimes_B M \to A$ and $\zeta \colon M \otimes_A N \to B$ are bimodule morphisms.

\begin{theorem}[{{\ref{thm:equiv-by-Morita-context}}}]\label{main-thm-morita-i}
  Let $(A,B,M,N,\xi,\zeta)$ be a Morita context, and let
  $\cF$ and $\cG$ be Gabriel topologies on $A$ and $B$, respectively.
  Then the Morita context $(A,B,M,N,\xi,\zeta)$ induces an equivalence
  between $\QModF A$ and $\QModG B$ if and only if the following conditions hold.
  \begin{enumerate}
    \item[(1a)]
      $\blank\otimes_B M$ maps $\cG$-torsion modules to
      $\cF$-torsion modules.
    \item[(1b)] $\blank\otimes_A N$ maps $\cF$-torsion modules to
      $\cG$-torsion modules.
    \item[(2a)]
      $\coker(\xi)$ is $\cF$-torsion.
    \item[(2b)]
      $\coker(\zeta)$ is $\cG$-torsion.
  \end{enumerate}
\end{theorem}

The preceding theorem gives a criterion for a Morita context to induce an equivalence of quotient categories.
We show that, when the underlying rings are closed with respect to their Gabriel topologies,
every such equivalence arises from a Morita context.
This extends the Morita theory for graded quotient categories $\qgr$
developed in \cite{liNoncommutativeResolutionsASGorenstein2026}
to general quotient categories.

\begin{theorem}[{{\ref{cor:quot-equiv-conseq-closed}}}]\label{main-thm-morita-ii}
  Let $A$ and $B$ be rings equipped with Gabriel topologies $\cF$ and $\cG$,
  respectively. Suppose that $A$ is $\cF$-closed and $B$ is $\cG$-closed.
  Assume that there is a pair of mutually quasi-inverse functors
  \begin{equation*}
    % \overline F\colon \QModF A \simeq \QModG B\reflectbox{$\colon$} \overline G.
    \overline F\colon \QModF A \simeq \QModG B\colon \overline G.
  \end{equation*}
  Let $M\defeq \omega_{\cF} \overline{G} \pi_{\cG}(B)$ and $N\defeq\omega_{\cG} \overline{F} \pi_{\cF}(A)$.
  Then the following hold.
  \begin{enumerate}
    \item As rings, $B\cong\End_A(M)$ and $A\cong\End_B(N)$.
      Via these isomorphisms, $M$ is a $B$-$A$-bimodule
      and $N$ is an $A$-$B$-bimodule.
    \item As bimodules, $N\cong\Hom_A(M, A)$
      and $M\cong\Hom_B(N, B)$.
    \item%\label{enum:quot-equiv-conseq-ample-closed}
      $\pi_{\cF}(M)$ is fully $\cG$-ample,
      and $\pi_{\cG}(N)$ is fully $\cF$-ample.
    \item $\overline F$ and $\overline G$ are induced by a Morita context
      $(A, B, M, N, \xi, \zeta)$.
  \end{enumerate}
\end{theorem}

Let $A\to A'$ be a ring homomorphism, and let $\cF$ be a Gabriel topology on $A$.
Then $\cF$ induces a Gabriel topology $\cF'$ on $A'$ such that an $A'$-module is $\cF'$-torsion
if and only if it is $\cF$-torsion when regarded as an $A$-module.
The assumption that $A$ and $B$ are closed in \ref{main-thm-morita-ii} can be dropped by replacing $(A,\cF)$ and $(B, \cG)$ with their respective rings of quotients
$A_{\cF}$ and $B_{\cG}$, equipped with the induced topologies $\cF'$ and $\cG'$ (see \ref{thm:morita-type,cor:quot-equiv-conseq-closed}).

The application considered here concerns the noncommutative Auslander theorem.
In the classical setting, if $S=\CC[x_1, \dotsc, x_n]$ and $G\subseteq\GL_n(\CC)$ is a finite subgroup
containing no pseudo-reflections, then the natural map
$S\#G\to\End_{S^G}(S)$ is an isomorphism
(\cite[Thm.\@ 5.15]{leuschkeCohenMacaulayRepresentations2012}).
For a finite-dimensional Hopf algebra $H$ acting on an algebra $S$,
the analogous Auslander map is
\begin{equation}\label{eq:auslander-map}
  \psi\colon S\# H\to\End_{S^H}(S),\quad s\#h\mapsto(x\mapsto s(h\rightharpoonup x)),
\end{equation}
where $S^H$ is the invariant subalgebra and $S\# H$ is the smash product.
The noncommutative Auslander theorem is said to hold for this action
if the Auslander map is bijective.

Under suitable hypotheses on the algebra and the action,
Bao, He, and Zhang provide a criterion for the bijectivity of the Auslander map for semisimple $H$
in terms of the GK dimension of a certain quotient of $S\# H$ \cite[Thm.\@ 0.3]{baoPertinencyHopfActions2019}.

Applying the preceding theorems,
we obtain the following sufficient condition without assuming that $H$ is semisimple.

% 我们这边相较BHZ的额外的条件是: 作用要faithful, S要是domain
\begin{theorem}[{{\ref{cor:nc-auslander-thm}}}]\label{main-thm-auslander}
  Let $S$ be a noetherian $k$-algebra, and let $H$ be a finite-dimensional Hopf algebra over $k$ acting on $S$
  such that $S$ is faithful as a left $H$-module.
  Suppose that $S$ is GK-Cohen--Macaulay of GK dimension $d\geq 2$
  and has no zero-divisors.
  Let $t$ be a nonzero left integral of $H$,
  and write $(1\#t)$ for the two-sided ideal of $S\#H$ generated by $1\# t$.
  Then the Auslander map $S\#H\to\End_{S^H}(S)$ is an isomorphism
  if $\GKdim(S\#H\div (1\#t))\leq d-2$.
\end{theorem}

In \ref{ex:auslander-1,ex:auslander-2}, we illustrate this result with examples involving modular group actions and a Hopf action of a Taft algebra on quantum polynomial algebras.

This paper is organized as follows. In \ref{Quotient Categories and Gabriel Topologies}, we review quotient categories and Gabriel topologies and present several preliminary results.
In \ref{progenerators in Quotient Categories}, we study $\cF$-ampleness in abelian categories and prove \ref{main-thm-serre}.
In \ref{Quotient Equivalences and Morita Contexts}, we develop a Morita theory for quotient categories and prove \ref{main-thm-morita-i,main-thm-morita-ii}.
In \ref{sec:application}, we apply this theory to the noncommutative Auslander theorem for non-semisimple Hopf actions on AS-regular algebras and prove \ref{main-thm-auslander}.

% !TeX root = ../main.tex

\section{Quotient categories, functors descent to quotient categories} \label{Quotient Categories and Gabriel Topologies}

% 关于Faith的书 \cite{faithAlgebraRingsModules1973}
% Def. 15.3 商范畴的Serre定义
% Thm. 15.7 商范畴是abelian cat
% p. 505 localizing子范畴通过section函子存在来定义
% Thm. 15.11 对于Grothendieck cat中的一个Serre subcat, localizing subcat与关于无限直和封闭等价
% Prop. 15.14 证明section函子如果存在则一定是fully faithful
% Prop. 16.8B localizing子范畴和radical的一一对应
% 这本书里没有提商范畴也可以通过关于乘性系作Verdier商得到

This section first recalls several definitions and properties concerning quotient categories and Gabriel topologies, with particular reference to
\cite{faithAlgebraRingsModules1973, popescuAbelianCategoriesApplications1973, stenstromRingsQuotients1975}.
It then discusses the descent of functors to the quotient level. Finally, it introduces and characterizes the notion of $\cF$-depth with respect to a Gabriel topology $\cF$.

\subsection{Quotient categories and modules of quotients}
Let $\cC$ be an abelian category, and $\cS$ a Serre subcategory of $\cC$.
The \dfn{quotient category} $\cC\div\cS$ is defined as follows.
The objects of $\cC\div\cS$ are those of $\cC$.
The corresponding object of an object $X \in \cC$ in $\cC\div\cS$ is denoted by $\pi(X)$.
The Hom-sets of $\cC\div\cS$ are defined by
\[
  \Hom_{\cC\div\cS}(\pi(X),\pi(Y))=
  \dlim_{\substack{X'\leq X, X\div X'\in\cS\\Y'\leq Y,Y'\in\cS}}\Hom_{\cC}(X',Y\div Y').
\]
For any $f\in\Hom_{\cC}(X, Y)$, $\pi(f) \in \Hom_{\cC\div\cS}(\pi(X),\pi(Y))$ is the canonical image of $f$ in the direct limit.
These assignments define an additive functor $\pi \colon \cC \to \cC\div\cS$,
called the \dfn{quotient functor}.
The quotient category $\cC\div\cS$ is an abelian category, and the quotient functor $\pi$ is an exact functor
(\cite[Chap.\@ 4, Thm.\@ 3.8]{popescuAbelianCategoriesApplications1973}). The category $\cC\div\cS$ and the exact functor $\pi \colon \cC \to \cC\div\cS$ have the following universal property: for any abelian category $\cA$ and any exact functor $T\colon \cC \to \cA$ annihilating $\cS$, there exists a unique functor $T'\colon \cC\div\cS \to \cA$ such that $T=T' \pi$.

A Serre subcategory $\cS$ is called a \dfn{localizing subcategory} of $\cC$ if $\pi$ has a right adjoint $\omega\colon\cC\div\cS \to \cC$, which is called a \dfn{section functor} (\cite[Sect.\@ 4.4]{popescuAbelianCategoriesApplications1973}).
If $\cC$ is a cocomplete abelian category with injective envelopes, then the Serre subcategory $\cS$ is a localizing subcategory if and only if $\cS$ is closed under taking coproducts (\cite[Chap.\@ 4, Prop.\@ 6.3]{popescuAbelianCategoriesApplications1973}).
The adjunction gives two natural transformations
called the \dfn{unit}
$\eta\colon\id_{\cC}\rightarrow \omega\pi$
and the \dfn{counit}
$\epsilon\colon\pi\omega\rightarrow\id_{\cC\div\cS}$.
By \cite[Chap.\@ 4, Prop.\@ 4.3]{popescuAbelianCategoriesApplications1973},
the counit $\epsilon$ of the adjunction is a natural isomorphism,
and thus $\omega$ is fully faithful.

Now let $A$ be a ring. In this paper, an $A$-module always means a right $A$-module.
The category of all $A$-modules is denoted by $\Mod A$.
For a Gabriel topology $\cF$ on $A$,
the corresponding \dfn{torsion functor} $\tau_{\cF}\colon\Mod A\to\Mod A$
is the subfunctor of the identity functor defined by
\begin{equation*}
  \tau_{\cF}(\blank)\defeq\dlim_{I\in\cF}\Hom_A(A\div I, \blank),
\end{equation*}
where $\cF$ is viewed as a direct system by reverse inclusion.
An $A$-module $M$ is \dfn{$\cF$-torsion} if and only if $\tau_{\cF}(M)=M$,
while it is \dfn{$\cF$-torsion-free} if and only if $\tau_{\cF}(M)=0$.
For an $A$-module $M$, $\tau_{\cF}(M)$ denotes the largest $\cF$-torsion submodule of $M$.
The torsion functor $\tau_{\cF}$ is left exact,
and its $i$-th right derived functor satisfies
%has the form
\begin{equation*}
  \uR^i\tau_{\cF}(\blank)\cong\dlim_{I\in\cF}\Ext_A^i(A\div I, \blank).
\end{equation*}

There is a one-to-one correspondence between the
Gabriel topologies on $A$ and the localizing subcategories (or hereditary torsion classes) of $\Mod A$ (\cite[Chap.\@ VI, Thm.\@ 5.1]{stenstromRingsQuotients1975}).
Denote by $\QModF A$ the quotient category of $\Mod A$ modulo $\cF$-torsion modules, 
and by $\Hom_{A/\cF}(\blank,\blank)$ the Hom-sets in $\QModF A$.
The corresponding quotient functor is denoted by $\pi_{\cF}\colon \Mod A\to\QModF A$. The full subcategory consisting of all $\cF$-torsion modules is a localizing subcategory of $\Mod A$; consequently, there exists a section functor
$\omega_{\cF}\colon \QModF A\to\Mod A$.
The subscripts are usually omitted if there is no ambiguity.

For any $A$-module $M$, let $M_{\cF} = \dlim_{I\in\cF}\Hom_{A}(I, M/\tau(M))$. Then $A_{\cF}$ carries a canonical ring structure, called the \dfn{ring of quotients} of $A$ with respect to $\cF$, together with a canonical ring morphism $\eta^A\colon A \to A_{\cF}$. 
On the one hand, since $\tau(A)$ is the largest $\cF$-torsion submodule of $A$, by the definition of quotient categories, $\End_{A/\cF}(\pi(A))=\dlim_{I\in\cF}\Hom_{A}(I, A/\tau(A))=A_{\cF}$.
In fact, by definition, the ring structure of the endomorphism ring $\End_{A/\cF}(\pi(A))$ is identical to that of $A_{\cF}$. 
On the other hand, $M_{\cF}$ is an $A_{\cF}$-module, and the canonical morphism $\eta^M\colon M \to M_{\cF}$ is an $A$-module morphism when $M_{\cF}$ is regarded as an $A$-module via $\eta^A$. 
An $A$-module $M$ is \dfn{$\cF$-closed}
if and only if $M$ is $\cF$-torsion-free and $\cF$-injective (i.e.\@ $\Ext_A^1(A/I, M)=0$ for any $I \in \cF$).
When regarded as an $A$-module, $M_{\cF}$ is $\cF$-closed for any $M\in \Mod A$.
Moreover, the full subcategory
$\{M_{\cF} \mid M \in \Mod A\}$ of $\Mod A_{\cF}$ is equivalent to the full subcategory of $\Mod A$ consisting of all $\cF$-closed $A$-modules. The latter is denoted by $\Mod (A, \cF)$ (see \cite[p.\@ 196--198]{stenstromRingsQuotients1975}).

There is a diagram of functors
% \begin{equation}\label{eq:localization-diagram}
%   \begin{tikzcd}[
%     row sep=large,
%     left diagonal/.style={
%       start anchor={[sloped,allow upside down,xshift=-0.5em]north west},
%       end anchor={[sloped,allow upside down,xshift=-1em]south}
%     },
%     right diagonal/.style={
%       start anchor={[sloped,allow upside down,xshift=1em]south},
%       end anchor={[sloped,allow upside down,xshift=0.5em]north east}
%     }
%   ]
%     \Mod A & & \Mod A/\cF\\
%     & {\Mod(A,\cF)} &
%     \arrow[from=1-1,left adj]{1-3}[]{\pi}
%     \arrow[from=1-1,right adj]{1-3}[swap]{\omega=i\bar a}
%     \arrow[from=1-1,no line]{1-3}[rotate=90,sloped]{\vdash}
%     \arrow[from=2-2,left diagonal,left adj rev]{1-1}[swap,sloped]{a}
%     \arrow[from=2-2,left diagonal,right adj rev]{1-1}[sloped]{i}
%     \arrow[from=2-2,left diagonal,no line]{1-1}
%       [rotate=90,sloped,allow upside down]{\vdash}
%     \arrow[from=1-3,right diagonal,left adj rev]{2-2}[swap,sloped]{\pi i}
%     \arrow[from=1-3,right diagonal,right adj rev]{2-2}[sloped]{\bar a}
%     \arrow[from=1-3,right diagonal,no line]{2-2}[sloped]{\simeq}
%   \end{tikzcd}
% \end{equation}
  \begin{equation}\label{eq:localization-diagram}
\begin{tikzcd}
\Mod A \arrow[rr, "\pi", shift left] \arrow[rd, "a"', shift right=2] &                                                   & \Mod A/\cF \arrow[ll, "\omega = i \bar{a}", shift left] \arrow[ld, "\bar{a}"' ] \\
 & {\Mod (A, \cF)} \arrow[lu, "i"'] \arrow[ru, "\pi i"', shift right=2] &
\end{tikzcd}
  \end{equation}
where $a\colon M\mapsto M_{\cF}$ is the localization functor (\cite[p.\@ 217]{stenstromRingsQuotients1975}), $\bar{a}$ is the functor induced by the universal property of the quotient categories, and $i$ is the inclusion functor. 
In fact, $(\bar{a}, \pi i)$ is a pair of functors of category equivalence, $i$ is a right adjoint to $a$ and $\omega = i \bar{a}$ is a right adjoint to $\pi$, thus a section functor.
More precisely, %up to natural isomorphisms, 
the section functor $\omega$ is chosen so that 
\[
  \omega\pi(M)=\dlim_{I\in\cF}\Hom_{A}(I, M/\tau(M))=M_{\cF}, \textrm{ 
for any } M \in \Mod A,\]
and the unit $\eta\colon\id_{\Mod A}\rightarrow \omega\pi$ is chosen so that $\eta^M\colon M\to M_{\cF}=\omega\pi(M)$ is the canonical morphism for every $M\in \Mod A$.
An $A$-module $M$ is $\cF$-closed
if and only if $\eta^M\colon M\to\omega\pi(M)$ is an isomorphism;
if and only if $M$ is isomorphic to $\omega(\cX)$ for some $\cX\in\QModF A$.
Therefore, $\Mod (A,\cF)$ coincides with the essential image of $\omega$, and is equivalent to $\QModF A$.

% 这个引理是想来解释 M 是 torsion-free 当且仅当 eta^M 是单射, 在后面第四节新加的内容里用到了一下
\begin{lemma}\label{lem:kernel-eta}
  Let $A$ be a ring and $\cF$ a Gabriel topology on $A$.
  Suppose $M$ is an $A$-module. Then $\tau(M)=\ker(\eta^M)$.
\end{lemma}

\begin{proof}
  Since the codomain of $\eta^M$ is $\cF$-closed and in particular $\cF$-torsion-free,
  it follows that $\tau(M) \subseteq \ker(\eta^M)$.
  By \cite[Chap.\@ IX, Lem.\@ 1.2]{stenstromRingsQuotients1975},
  $\ker(\eta^M)$ is $\cF$-torsion, and therefore $\tau(M)=\ker(\eta^M)$.
\end{proof}

The full subcategory consisting of all finitely generated 
$A$-modules is denoted by
$\mod A$. 
If $A$ is right noetherian, then $\mod A$ is an abelian subcategory of $\Mod A$, 
and the category of finitely generated $\cF$-torsion modules is a Serre subcategory of $\mod A$.
Let $\qmodF A$ denote the quotient category of $\mod A$ by the finitely generated $\cF$-torsion modules.
Note that, by definition, $\qmodF A$ is a full subcategory of $\QModF A$.

\begin{lemma}\label{lem-qmodG consists of noetherian objects}
  Let $A$ be a right noetherian ring and $\cF$ a Gabriel topology on $A$.
  Then $\qmodF A$ is equivalent to the full subcategory of $\QModF A$ consisting of all noetherian objects.
\end{lemma}
\begin{proof}
  By \cite[Prop.\@ XIII.2.1]{stenstromRingsQuotients1975}, every object of $\qmodF A$ is noetherian in $\QModF A$.
  
  On the other hand, suppose that $\cM = \pi(M) \in \QModF A$ is a noetherian object. By \cite[Prop.\@ XIII.2.1]{stenstromRingsQuotients1975} again, there exists some finitely generated submodule $M'$ of $M$ such that $M/M'$ is $\cF$-torsion. It follows that $\cM =\pi(M) \cong \pi(M') \in \qmodF A$.
\end{proof}

\subsection{Induced functors between quotient categories}

An \emph{exact} functor $F\colon\cC\to\cD$ between
abelian categories factors through the quotient category
$\cC\div\cS$ by a Serre subcategory $\cS$ if and only if
$F$ sends $\cS$ to $0$.
However, the situation becomes more involved when the functor $F$ is no longer exact. Here is an example showing that the exact property of $F$ is necessary
(see \cite[Cor.\@ 15.9]{faithAlgebraRingsModules1973}, \cite[Prop.\@ 2.1.4]{louInvariantTheoryRegular2016},
\cite[Lem.\@ 1.1]{smithCorrigendumMapsNonCommutative2016}). 

\begin{example} Let $A$ be the $2\times 2$ upper triangular matrix algebra over a field $k$, that is, the path algebra of the quiver $1 \xrightarrow{\alpha} 2$.
Let $\cC = \mod{A}$ and $\cS = \{M \in \mod{A} \mid M \otimes_A Ae_1=0\}$. Then $\cS$ is a Serre subcategory of $\cC$.
 Denote by $S_1$ and $S_2$ the simple modules corresponding to vertices $1$ and $2$, respectively.
In fact,
$\cS \cong \{S_2^{\oplus r}\mid r \geqslant 0\}.$ Consider the left exact
functor $T=\Hom_A(S_1, \blank)\colon \mod{A} \to \mod{k}$, which annihilates $\cS$ obviously, and the exact sequence
\[0 \to S_2 \xrightarrow{i} e_1A \xrightarrow{p} S_1 \to 0.\]
Clearly, $\pi(p)$ is an isomorphism in $\cC\div\cS$. Since $\Hom_A(S_1, e_1A)=0$, it follows that
$T(p)\colon \Hom_A(S_1, e_1A) \to \Hom_A(S_1, S_1)$ is the zero map.  Hence, there does not exist any functor $T'\colon \cC\div\cS \to \mod{k}$ such that $T \cong T' \pi$.

Note that $\mod{A}$ is an abelian category with enough injectives, and $\uR^1T(S_2)=\Ext^1_A(S_1, S_2)\cong k \neq 0$ (see (2) in the following proposition).
\end{example}

The following proposition characterizes
when the functor $F$ factors through quotient categories under certain natural conditions.
\begin{proposition}
  [{{\cite[Prop.\@ 2.11]{liNoncommutativeResolutionsASGorenstein2026}}}]
  \label{prop:induced-functors}
  Let $F\colon \cC\to\cD$ be a functor between two abelian categories $\cC$ and $\cD$.
  Let $\cS\subseteq\cC$ (resp.\@ $\cT\subseteq\cD$) be a Serre subcategory, 
  and $\pi\colon \cC\to \cC\div\cS$
  (resp.\@ $\pi\colon\cD\to\cD/\cT$) be the quotient functor.
  \begin{enumerate}
    \item Assume that $\cC$ has enough projective objects and $F$ is right exact.
      Then there exists a functor
      $\overline F\colon\cC\div\cS \to \cD/\cT$
      such that $ \overline F\pi\cong \pi F$ if and only if
      both $F$ and $\uL_1F$ send $\cS$ to $\cT$.
      Such a functor $\overline{F}$ is unique up to natural isomorphism.
    \item Dually, assume that $\cC$ has enough injective objects and
      $F$ is left exact.
      Then there exists a functor
      $\overline F\colon\cC\div\cS \to \cD/\cT$
      such that $\overline F\pi\cong \pi F$ if and only if
      both $F$ and $\uR^1F$ send $\cS$ to $\cT$. Such a functor $\overline{F}$ is unique up to a natural isomorphism.
  \end{enumerate}

  In either case, if furthermore $\cS$ is a localizing subcategory of $\cC$,
  then $\overline F\cong \pi F\omega$ where $\omega\colon\cC\div\cS \to \cC$ denotes the section functor.
\end{proposition}

Note that \ref{prop:induced-functors} also applies to contravariant functors
by dual arguments.
\begin{definition}\label{def:induced-functors}
  Let $A$ and $B$ be rings equipped with Gabriel topologies $\cF$ and $\cG$, respectively. Let $F\colon\Mod A\to\Mod B$ be a (covariant or contravariant) functor.
  A functor $\overline F\colon\QModF A\to\QModG B$ is said to be induced by the functor $F$
  if there is a natural isomorphism $\pi F\cong \overline F\pi$.
\end{definition}

If, moreover, there is another functor $G\colon\Mod B\to\Mod A$ inducing $\overline{G}\colon\QModG B\to\QModF A$,
then $\overline{G}\comp\overline{F}$ is induced by $GF$; that is, $\overline{G}\comp\overline{F}=\overline{GF}$.

In the following diagrams, the notation $G\dashv F$ means that $G$ is a left adjoint to $F$.
By a commutative diagram of functors, we always mean that the diagram commutes up to a natural isomorphism. 
The following lemma is used in several places, namely \ref{lem:equiv-quot-cat,thm:morita-type}.

\begin{lemma}\label{lem:adjunction-diagram}
  Let $A$ and $B$ be rings equipped with Gabriel topologies $\cF$ and $\cG$, respectively.
  Suppose there is a functor $F\colon\Mod A\to \Mod B$ admitting a left adjoint $G$.
  Then the following are equivalent.
  \begin{enumerate}
    \item $G$ induces a functor $\overline G\colon \QModG B\to \QModF A$, which has a right adjoint $\overline F$.
    \item $F$ sends $\cF$-closed $A$-modules to $\cG$-closed $B$-modules.
    \item There is a functor $\overline F\colon \QModF A \to \QModG B$ satisfying $\omega_{\cG}\overline F \cong F \omega_{\cF}$.
  \end{enumerate}
  When these conditions hold, 
  $\overline F
    \cong
    \pi_{\cG} F\omega_{\cF},
    \,
    \overline G
    \cong
    \pi_{\cF} G \omega_{\cG}.$
\end{lemma}
\begin{proof}
  \textit{(1)$\Rightarrow$(2).}
  Consider the following diagram of adjoint functors.
  \begin{equation}
    % \crefalias{equation}{diagram}
    \label{diagr:context-adjunction}
    % https://tex.stackexchange.com/a/52379
    \begin{tikzcd}[column sep=huge]
      \Mod A & \Mod B\\
      \QModF A & \QModG B
      \arrow[from=1-1,left adj rev,]{1-2}[]{G}
      \arrow[from=1-1,right adj rev,]{1-2}[swap,]{F}
      \arrow[from=1-1,no line]{1-2}[rotate=90,sloped,]{\vdash}
      \arrow[from=2-1,left adj rev,]{1-1}[]{\pi}
      \arrow[from=2-1,right adj rev,]{1-1}[swap,]{\omega}
      \arrow[from=2-1,no line,]{1-1}[rotate=90,sloped,]{\vdash}
      \arrow[from=2-1,left adj,]{2-2}[]{\overline F}
      \arrow[from=2-1,right adj,]{2-2}[swap,]{\overline G}
      \arrow[from=2-1,no line]{2-2}[rotate=-90,sloped,]{\vdash}
      \arrow[from=1-2,left adj,]{2-2}[]{\pi}
      \arrow[from=1-2,right adj,]{2-2}[swap,]{\omega}
      \arrow[from=1-2,no line,]{2-2}[rotate=90,sloped,]{\vdash}
    \end{tikzcd}
  \end{equation}
  By assumption, $\overline G$ is induced by $G$, that is,
  the outer square of \ref{diagr:context-adjunction} commutes.
  It follows from the uniqueness of right adjoints that the inner square of \ref{diagr:context-adjunction} commutes, that is, $\omega_{\cG} \overline F \cong F \omega_{\cF}$. Indeed, $\overline F
    \cong \pi_{\cG} \omega_{\cG} \overline F  \cong 
    \pi_{\cG} F\omega_{\cF}$.
  Recall that the $B$-module $\omega_{\cG} (\cY)$ is $\cG$-closed for any $\cY \in \QModG B$.
  Therefore, for any $\cF$-closed $A$-module $X$, $F(X)\cong F\omega_{\cF}\pi_{\cF}(X)\cong\omega_{\cG}\overline F\pi_{\cF}(X)$,
  which means that $F$ preserves closed modules.

Since   $\overline G \pi_{\cG} \cong \overline G \pi_{\cG} \omega_{\cG} \pi_{\cG} \cong \pi_{\cF} G \omega_{\cG}  \pi_{\cG}$, it follows that $\overline G \cong \pi_{\cF} G \omega_{\cG}$.

  \textit{(2)$\Rightarrow$(3).}
  Let $\overline F \defeq \pi_{\cG} F \omega_{\cF}$.
  For any $\cX\in\QModF A$, there is a natural isomorphism
  \begin{align*}
    \omega_{\cG}\overline F(\cX)&=\omega_{\cG}\big(\pi_{\cG} F \omega_{\cF}(\cX)\big)\\
    &=\omega_{\cG}\pi_{\cG} \big(F \omega_{\cF}(\cX)\big) \cong F \omega_{\cF}(\cX),
  \end{align*}
  as $F\omega_{\cF}(\cX)$ is $\cG$-closed by assumption.
  It follows that $\omega_{\cG}\overline F \cong F \omega_{\cF}$.

  \textit{(3)$\Rightarrow$(1).}
  Let $\overline G\defeq \pi_{\cF}G\omega_{\cG}$.
  For any $\cX\in\QModF A$ and $\cY\in\QModG B$, there are natural isomorphisms
  \begingroup\allowdisplaybreaks
  \begin{align*}
    \Hom_{\QModF A}\bigl(\overline{G}(\cY), \cX\bigr)
    &= \Hom_{\QModF A}\bigl(\pi_{\cF}G\omega_{\cG}(\cY), \cX\bigr)\\
    &\cong \Hom_A\bigl(G\omega_{\cG}(\cY), \omega_{\cF}(\cX)\bigr)\\
    &\cong \Hom_B\bigl(\omega_{\cG}(\cY), F\omega_{\cF}(\cX)\bigr)\\
    &\cong \Hom_B\bigl(\omega_{\cG}(\cY), \omega_{\cG}\overline F(\cX)\bigr)\\
    &\cong \Hom_{\QModG B}\bigl(\cY, \overline F(\cX)\bigr).
  \end{align*}
  \endgroup
  It follows that $\overline G$ is a left adjoint to $\overline F$.

  The inner square of \ref{diagr:context-adjunction} commutes by assumption.
  Hence, by the uniqueness of left adjoints, the outer square also commutes, that is, $ \overline{G} \pi_{\cG}\cong \pi_{\cF} G$, which means that the functor  $\overline G$ is induced by $G$.
\end{proof}

\subsection{Gabriel topologies under change of rings}

Let $B\to A$ be a ring homomorphism, and let $\cG$ be a Gabriel topology on $B$.
Then $\cG$ naturally induces a Gabriel topology $\cF$ on $A$
via the restriction of scalars functor,
namely, an $A$-module is $\cF$-torsion
if and only if, when viewed
as a $B$-module, it is $\cG$-torsion. The following lemma generalizes \cite[Prop.\@ X.2.2]{stenstromRingsQuotients1975}.
\begin{lemma}\label{lem:top-change-of-rings}
  Let $\phi\colon B\to A$ be a ring homomorphism, let $\cG$ be a Gabriel topology on $B$ and let $\cF$ denote the Gabriel topology on $A$ induced by $\cG$.
  Then
  \begin{equation}\label{eq:top-change-of-rings}
    \cF\subseteq \{I \submodule A_A \mid \phi^{-1}(I)\in\cG\}.
  \end{equation}
  The equality holds if $\coker(\phi)$ is $\cG$-torsion.
\end{lemma}

\begin{proof}
  The inclusion \ref{eq:top-change-of-rings} follows from the exact sequence $0 \to  B\div\phi^{-1}(I) \to  A\div I$.
  
  If $\coker(\phi)$ is $\cG$-torsion,  then the equality $\cF= \{I \submodule A_A \mid \phi^{-1}(I)\in\cG\}$ follows from the exact sequence $0 \to  B\div\phi^{-1}(I) \to  A\div I \to A\div (I+\im(\phi)) \to 0$.
\end{proof}

\begin{lemma}\label{lem:tau_A=tau_B}
  Keep the notation as in \ref{lem:top-change-of-rings}. If $\cF= \{I\submodule A_A\mid \phi^{-1}(I)\in\cG\}$, then, for any $M\in \Mod A$, $\tau_{\cF}(M)=\tau_{\cG}(M)$.
\end{lemma}

\begin{proof}
  By the definition of $\cF$, $\tau_{\cF}(M) \subseteq \tau_{\cG}(M)$. On the other hand,
  suppose $x\in \tau_{\cG}(M)$, that is, $x\cdot J=x \phi(J)=0$ for some $J\in\cG$. Let $I=\phi(J)A$.
  It follows from $J\submodule \phi^{-1}(I)$ that $\phi^{-1}(I)\in\cG$.
  By assumption, $I\in\cF$. Thus $x I = x \phi(J)A=0$. Therefore, $x\in \tau_{\cF}(M)$. Hence $\tau_{\cG}(M) \subseteq \tau_{\cF}(M)$.
\end{proof}

Let $\phi\colon B\to A$ be a ring homomorphism, let $\cG$ be a Gabriel topology on $B$, and let $\cF$ be the Gabriel topology on $A$ induced by $\cG$.
The restriction of scalars functor
\[\phi^* = \Hom_A({}_BA_A, \blank) \colon \Mod A\to\Mod B\] is an exact functor
sending $\cF$-torsion modules to $\cG$-torsion modules,
and therefore it induces a functor $\overline{\phi^*}\colon\QModF A\to \QModG B$.
Regarding the extension-of-scalars functor
$\phi_!\defeq\blank\otimes_B A\colon\Mod B\to\Mod A$, which is left adjoint to $\phi^*$,
we have the following lemma
generalizing \cite[Prop.\@ 2.5]{artinNoncommutativeProjectiveSchemes1994}.

\begin{lemma}
  \label{lem:equiv-quot-cat}
  Keep the notation as above.
  Suppose that the kernel and the cokernel of $\phi \colon B \to A$ are $\cG$-torsion.
  Then the following hold.
  \begin{enumerate}[ref=\ref{lem:equiv-quot-cat}(\arabic*)]
    \item%\label{enum:equiv-quot-cat-equiv}
    $\phi^*$ and $\phi_!$ induce an equivalence of categories
    \[
      \overline{\phi^*}\colon \QModF{A}\simeq\QModG{B}\colon \overline{\phi_!}.
    \]
    \item%\label{enum:equiv-quot-cat-comm}
      The inner and outer squares of the following diagram are commutative.
      \begin{equation*}
        % \label[diagram]{diagr:context-adjunction}
        % https://tex.stackexchange.com/a/52379
        \begin{tikzcd}[column sep=large]
          \Mod A & \Mod B\\
          \QModF A & \QModG B
          \arrow[from=1-1,left adj rev ,]{1-2}[]{\phi_!}
          \arrow[from=1-1,right adj rev,]{1-2}[swap,]{\phi^*}
          \arrow[from=1-1,no line]{1-2}[rotate=90,sloped,]{\vdash}
          \arrow[from=2-1,left adj rev,]{1-1}[]{\pi}
          \arrow[from=2-1,right adj rev,]{1-1}[swap,]{\omega}
          \arrow[from=2-1,no line,]{1-1}[rotate=90,sloped,]{\vdash}
          \arrow[from=2-1,left adj,]{2-2}[]{\overline{\phi^*}}
          \arrow[from=2-1,right adj,]{2-2}[swap,]{\overline{\phi_!}}
          % \arrow[from=2-1,no line]{2-2}[sloped,]{\simeq}
          \arrow[from=2-1,no line]{2-2}[sloped,]{\simeq}
          \arrow[from=1-2,left adj,]{2-2}[]{\pi}
          \arrow[from=1-2,right adj,]{2-2}[swap,]{\omega}
          \arrow[from=1-2,no line,]{2-2}[rotate=90,sloped,]{\vdash}
        \end{tikzcd}
      \end{equation*}
    \item%\label{enum:equiv-quot-cat-closed}
      $\phi^*$ sends $\cF$-closed $A$-modules to $\cG$-closed $B$-modules.
  \end{enumerate}
\end{lemma}

\begin{proof}
  (1) We claim that $\phi_!=\blank\otimes_B A\colon\Mod B\to\Mod A$ induces a functor $\overline{\phi_!}$ between the quotient categories. We first show that, for any $N\in\Mod B$, the kernel and cokernel of
  the $B$-module morphism $\id_N \otimes \phi \colon N \otimes_B B \to N\otimes_B A$ are $\cG$-torsion.

  Since both $\ker(\phi)$ and $\coker(\phi)$ are $B$-$B$-bimodules which are $\cG$-torsion as right $B$-modules, $N\otimes_B \ker (\phi)$, $N\otimes_B \coker(\phi)$ and $\Tor^B_1(N, \coker(\phi))$ are $\cG$-torsion modules.
  Therefore, $\coker(\id_N \otimes \phi)\cong N\otimes_B \coker(\phi)$ is $\cG$-torsion.
  It follows from the $B$-$B$-bimodule exact sequences
  \[0\to \ker(\phi)\to B\xrightarrow{\phi|} \im(\phi)\to 0\text{\quad and \quad}0\to \im(\phi)\xrightarrow{i} A\to \coker(\phi)\to 0\]
  that $\ker(\id_N \otimes \phi|)$ and $\ker(\id_N \otimes i)$ are $\cG$-torsion.
  Then, by the exact sequence  
  \[
    0\to\ker(\id_N \otimes \phi|)\to\ker(\id_N \otimes \phi)\to\ker(\id_N \otimes i),
  \]
  $\ker(\id_N \otimes \phi)$ is $\cG$-torsion.

  Now, suppose that $N$ is a $\cG$-torsion module. Applying $\blank\otimes_B B \xrightarrow{\id \otimes \phi} \blank\otimes_B A$
  to the short exact sequence $0\to\Omega\to Q\to N\to 0$,
  where $Q$ is $B$-projective, yields the following exact commutative diagram,
  \begin{equation*}
    \begin{tikzcd}
      & 0 & \Omega \otimes_B B  & Q \otimes_B B  & N \otimes_B B & 0\\
      0 & \Tor^B_1(N, A) & \Omega\otimes_B A & Q\otimes_B A & N\otimes_B A
      & 0
      \arrow[from=1-2,]{1-3}[]{}
      \arrow[from=1-3,]{1-4}[]{}
      \arrow[from=1-4,]{1-5}[]{}
      \arrow[from=1-5,]{1-6}[]{}
      \arrow[from=2-1,]{2-2}[]{}
      \arrow[from=2-2,]{2-3}[]{}
      \arrow[from=2-3,]{2-4}[]{}
      \arrow[from=2-4,]{2-5}[]{}
      \arrow[from=2-5,]{2-6}[]{}
      \arrow[from=1-3,]{2-3}[]{\id \otimes \phi}
      \arrow[from=1-4,]{2-4}[]{\id \otimes \phi}
      \arrow[from=1-5,]{2-5}[]{\id \otimes \phi}
    \end{tikzcd}
  \end{equation*}
  where the kernels and the cokernels of the vertical maps are $\cG$-torsion
  by the discussion in the last paragraph.
  Applying the exact quotient functor $\pi$ to the exact commutative diagram above, it follows that both $N\otimes_B A$ and $\Tor^B_1(N, A)$ are $\cG$-torsion, and, consequently, $\cF$-torsion. Therefore, $\phi_!$ induces a functor $\overline{\phi_!}  \colon  \QModG{B} \to \QModF{A}$ by \ref{prop:induced-functors}.

  It is left to show that 
  $\overline{\phi_!}\colon \QModG{B}\leftrightarrows\QModF{A}
  \colon\overline{\phi^*}$
  is an equivalence of categories.
  On one hand, the map $\id_N \otimes\phi$ induces a natural isomorphism between the identity functor
  and $\overline{\phi^*} \overline{\phi_!}$.
  On the other hand, note that for any $M\in\Mod A$, the natural ($A$-module) epimorphism
  $\psi\colon M\otimes_B A\to M$
  splits as a $B$-module morphism, with the section morphism
  $\id_M \otimes\phi\colon M \cong M\otimes_B B \to M\otimes_B A$.
  Since $\pi(\id_M \otimes\phi)$ is an isomorphism in $\QModG B$,
  the kernel of the surjective morphism $\psi$ is $\cG$-torsion,
  and thus $\cF$-torsion.
  Hence $\psi$ induces a natural isomorphism between
  $\overline{\phi_!} \overline{\phi^*}$
  and the identity functor.

  (2) Obviously, the outer square is commutative.
  The uniqueness of adjoint functors yields
  the commutativity of the inner square.

  (3) It follows from \ref{lem:adjunction-diagram}.
\end{proof}

\ref{lem:equiv-quot-cat} applies to the canonical ring morphism $\eta^B\colon B \to B_{\cG}$, where $B_{\cG}$ is endowed with the Gabriel topology $\cG'$ induced by $\cG$.

\begin{corollary}\label{cor:eta_A-induces-equi}
  Keep the above notation.
  \begin{enumerate}[ref=\ref{cor:eta_A-induces-equi}(\arabic*)]
    \item\label{enum:eta_A-induces-equi-i} The functors $(\eta^B)_!$ and $(\eta^B)^*$ associated to $\eta^B\colon B\to B_{\cG}$
    induce an equivalence
    \[
    \overline{(\eta^B)_!}\colon  \QModG{B}\simeq\QModGG{B_{\cG}}  \colon\overline{(\eta^B)^*}.
    \]
    \item\label{enum:eta_A-induces-equi-ii} There are natural isomorphisms
      \[\pi_{\cG'}\comp(\eta^B)_!\cong \overline{(\eta^B)_!}\comp\pi_{\cG}
      \text{\quad and\quad}
      (\eta^B)^*\comp\omega_{\cG'}\cong \omega_{\cG}\comp\overline{(\eta^B)^*}.\]
    \item $(\eta^B)^*$ sends $\cG'$-closed $B_{\cG}$-modules to $\cG$-closed $B$-modules.
  \end{enumerate}
\end{corollary}

\begin{proof}
  Note that both the kernel and the cokernel of $\eta^B$ are $\cG$-torsion (see \cite[Chap. IX, Lems.\@ 1.2 and 1.5]{stenstromRingsQuotients1975}).
  Then (1)--(3) are restatements of \ref{lem:equiv-quot-cat}.
\end{proof}

\subsection{Depth with respect to a Gabriel topology}

Let $\cF$ be a Gabriel topology on $A$.
For any $M\in\Mod A$,
the \dfn{$\cF$-depth} of $M$
is defined by
\begin{equation*}
  \depF M\defeq
  \inf\left\{i\;\middle\vert\;\Ext_A^i(A\div I, M)\neq 0
  \text{ for some }I\in\cF\right\}.
\end{equation*}
It is clear that $M$ is $\cF$-torsion-free if and only if $\depF M\geq 1$, and $\cF$-closed if and only if $\depF{M}\geq 2$.
The following lemmas relate depth with respect to a Gabriel topology to derived torsion functors.

\begin{lemma}\label{lem:tf-4-term-exact-seq}
  Let $\cF$ be a Gabriel topology on a ring $A$,
  and $M$ an $\cF$-torsion-free $A$-module.
  Then there is a short exact sequence
  \begin{equation*}
    \begin{tikzcd}
      0 \rar & M \rar["\eta^M"] & \omega\pi(M) \rar & \uR^1\tau(M) \rar & 0.
    \end{tikzcd}
  \end{equation*}
\end{lemma}

\begin{proof}
  By applying $\Hom_A(\blank, M)$ to the exact sequence
  $0\to I\to A\to A\div I\to 0$,
  one obtains the exact sequence
  \begin{equation*}
    \begin{tikzcd}[cramped,column sep=1em]
      0 \rar & \Hom_A(A\div I, M) \rar & M \rar & \Hom_A(I, M) \rar
      & \Ext_A^1(A\div I, M) \rar & 0.
    \end{tikzcd}
  \end{equation*}
  Taking the direct limit
  yields the desired result as $\tau(M)=0$.
\end{proof}

\begin{lemma}
  %[{{\cite[Prop.\@ 4.5.6]{popescuAbelianCategoriesApplications1973}}}]
  \label{lem:depth-inj-res}
  Let $\cF$ be a Gabriel topology on a ring $A$, and $M$ an $A$-module
  with the minimal injective resolution
  $0 \to M \to Q^0 \stackrel{d^0}{\to} Q^1 \to \dotsb$.
  % \begin{equation*}
  %   \begin{tikzcd}[cramped,column sep=1em]
  %     0 & M & Q^0 & Q^1 & \dotsb
  %     \arrow[from=1-1,]{1-2}[]{}
  %     \arrow[from=1-2,]{1-3}[]{}
  %     \arrow[from=1-3,]{1-4}[]{d^0}
  %     \arrow[from=1-4,]{1-5}[]{}
  %   \end{tikzcd}
  % \end{equation*}
  Then, for any fixed integer $n\geq 0$, $\depF M\geq n+1$ if and only if $Q^i$ is $\cF$-torsion-free for all $i\leq n$.
\end{lemma}

\begin{proof}
  The `if' part is trivial. For the `only if' part,
  when $n=0$, it is nothing but the fact that $\cF$-torsion-free modules are closed
  under taking injective envelopes.
  We proceed by assuming that the assertion holds for $n-1$.
  Let $\Omega^i=\ker(d^i)$.
  It suffices to show that $\Omega^n$ is $\cF$-torsion-free
  as $Q^n$ is its injective envelope.
  %Applying the left exact torsion functor $\tau$ to 
  By inductive assumption, the short exact sequence
  $0\to\Omega^{n-1}\to Q^{n-1}\to \Omega^{n}\to 0$
  yields the exact sequence
  \begin{equation*}
   0 = \tau(Q^{n-1})\to\tau(\Omega^{n})\to\uR^1\tau(\Omega^{n-1}).
  \end{equation*}
  Since $\depF M\geq n+1$, it follows that
  \begin{equation*}
    \uR^1\tau(\Omega^{n-1})\cong\uR^n\tau(M)
    \cong\dlim_{I\in\cF}\Ext_A^n(A\div I, M)=0.
  \end{equation*}
  Hence, $\tau(\Omega^n) =0$ and $\Omega^n$ is $\cF$-torsion-free.
\end{proof}

More generally, the $\cF$-depth of a module may be read from the 
derived torsion functors as shown by the next lemma.
\begin{lemma}%[{{\cite[Exer.\@ IX.2.10]{stenstromRingsQuotients1975}}}]
  \label{lem:depth-lcd}
  Let $\cF$ be a Gabriel topology on a ring $A$.
  % \begin{enumerate}
  %   \item For any $M\in\Mod A$, $\depF M=\inf\{i\mid\uR^i\tau(M)\neq 0\}.$
  %   \item A module $M$ is $\cF$-closed if and only if $\depF{M}\geq 2$.
  % \end{enumerate}
For any $M\in\Mod A$, $\depF M=\inf\{i\mid\uR^i\tau(M)\neq 0\}.$
\end{lemma}

\begin{proof}
  %The `$\leq$' part is trivial.
  By \ref{lem:depth-inj-res}, $\depF M \leqslant \inf\{i\mid\uR^i\tau(M)\neq 0\}.$
  If $\depF M =0$, then $\tau (M) \neq 0$. So, $\inf\{i\mid\uR^i\tau(M)\neq 0\} =0$.
  Now, suppose $0 < \depF M=n + 1 < \infty$.
  Then, there is some $I\in\cF$ such that $\Ext_A^{n+1}(A\div I, M) \neq 0$.
  On the other hand, for any right ideals $I \supmodule J \in \cF$, $\Ext_A^{n}(I\div J, M)=0$
  by \ref{lem:depth-inj-res}
  as $I\div J$ is $\cF$-torsion.
  % as a quotient of $A\div J$.
% Then by applying $\Hom_A(\blank, M)$ to 
 It follows from the exact sequence
  $ 0\to I\div J\to A\div J\to A\div I\to 0$
  that
   \[ 0=\Ext_A^{n}(I\div J, M)\to\Ext_A^{n+1}(A\div I, M)
    \to\Ext_A^{n+1}(A\div J, M)\]
  is exact,
  which implies that $\dlim_{I}\Ext_A^{n+1}(A\div I, M)$
  is a direct union. 
  Hence, $\uR^{n+1}\tau(M)\cong \dlim_{I}\Ext_A^{n+1}(A\div I, M)\neq 0$, and $\inf\{i\mid\uR^i\tau(M)\neq 0\}\leqslant n + 1 = \depF M$.
  % (2) It follows from (1) and \ref{lem:tf-4-term-exact-seq}.
\end{proof}

\begin{remark}\label{rem:depth-basis}
  Recall that a \dfn{basis} of a Gabriel topology $\cF$ is a subset $\cB$ of $\cF$
  such that every $I\in\cF$ contains some $J\in\cB$.
  Since $\cB$ is cofinal in $\cF$ with respect to reverse inclusion,
  the proof of \ref{lem:depth-lcd} shows that, for any $M\in\Mod A$,
  \[\inf\left\{i\;\middle\vert\;\Ext_A^i(A\div I, M)\neq 0
  \text{ for some }I\in\cB\right\}
  =\inf\{i\mid\uR^i\tau(M)\neq 0\}=\depF{M}.\]
\end{remark}

  % !TeX root = ../main.tex

\section{Ampleness with respect to a Gabriel topology}
\label{progenerators in Quotient Categories}

Let $\cC$ be an abelian category. The main objects considered in this section are those
$\cM$ in $\cC$ that are (fully) ample
with respect to a Gabriel topology on $B=\End_{\cC}(\cM)$. 
Under certain hypotheses, it is shown that such an $\cM$ induces an equivalence between $\cC$ and a quotient category of $\mod B$ (or $\Mod B$).
This study is motivated by Artin and Zhang's use of noncommutative ampleness to characterize noncommutative projective schemes as graded quotient categories \cite[Thm.\@ 4.5]{artinNoncommutativeProjectiveSchemes1994}.
Noncommutative ampleness is also used to establish equivalences between graded quotient categories in the context of noncommutative resolutions of noncommutative isolated singularities \cite{liNoncommutativeResolutionsASGorenstein2026,liNoncommutativeResolutionsNoncommutative2026}.

\subsection{Noncommutative analogues for ampleness}\label{subsec:ample-explanation}
We first recall the definition of ampleness given in \cite{artinNoncommutativeProjectiveSchemes1994}.

\begin{definition}[{\cite[(4.2.1)]{artinNoncommutativeProjectiveSchemes1994}}]
  \label{def:AZ-ample}
  Let $\cC$ be an abelian category, $\cM\in\cC$ and $s$ be an auto-equivalence of $\cC$.
  The pair $(\cM, s)$ is called \dfn{ample} if it satisfies the following two conditions.
  \begin{enumerate}[beginpenalty=10000,midpenalty=10000,ref={(\arabic*) in \ref{def:AZ-ample}}]
    \item\label{enum:AZ-ample-1} For any $\cX \in \cC$, there are some positive integers $l_1,l_2,\dotsc,l_p$ and an epimorphism $\bigoplus_{i=1}^ps^{-l_i}(\cM)\to\cX$ in $\cC$;
    \item\label{enum:AZ-ample-2} For any epimorphism $\cX\to\cY$ in $\cC$, there is an integer $n_0$ such that the induced map $\Hom_{\cC}(s^{-n}(\cM),\cX)\to\Hom_{\cC}(s^{-n}(\cM),\cY)$ is surjective for every $n\geqslant n_0$.
  \end{enumerate}
\end{definition}

The notion of ampleness in \ref{def:AZ-ample} is motivated by the classical notion of ampleness for line bundles in algebraic geometry. Indeed, by \cite[Theorem 1.2.6]{lazarsfeldPositivityAlgebraicGeometry2004}, a line bundle $\mathcal{L}$ on a projective scheme $X$ is ample if and only if it satisfies the global generation condition \cite[Theorem 1.2.6(3)]{lazarsfeldPositivityAlgebraicGeometry2004};
equivalently, Serre's vanishing theorem holds for $\mathcal{L}$ \cite[Theorem 1.2.6(2)]{lazarsfeldPositivityAlgebraicGeometry2004}.
In fact, if $\cC=\qgr(A)$ and $\cM=\pi(A)$ for a graded quotient algebra $A$ of a polynomial algebra, 
then condition \ref{enum:AZ-ample-1} is a variant of the global generation condition, while condition \ref{enum:AZ-ample-2} follows from Serre's vanishing theorem.

The object $\cM$ is called a \dfn{finite graded generator} of $(\cC, s)$
if condition \ref{enum:AZ-ample-1} holds without restrictions on the integers $l_1, \dotsc, l_p$.

Let $(f_i)\colon\bigoplus_{i\in I}\cX_i\to\cY$ denote the canonical morphism for morphisms $f_i\colon\cX_i\to\cY \, (i \in I)$, whenever the coproduct exists.
Let $\gHom_{\cC}(\blank,\blank)\defeq\bigoplus_{n\in\ZZ}\Hom_{\cC}(\blank, s^n(\blank))$ be the graded Hom functor
(see \cite[p.\@ 248]{artinNoncommutativeProjectiveSchemes1994}).
The following observation serves as our starting point for modifying the ample conditions and defining (full) ampleness with respect to a (graded) Gabriel topology.

\begin{lemma}\label{lem:AZ-equiv-cond} % Let $\cC$ be an abelian category, $\cM\in\cC$ and $s$ be an auto-equivalence of $\cC$.
  Keep the notation as above. Suppose that $B\defeq\gHom_{\cC}(\cM, \cM)$ is bounded-below as a $\ZZ$-graded ring (i.e.\@, $B_n=0$ for $n\ll 0$). Then the following conditions are equivalent.
  \begin{enumerate}
    \item $(\cM, s)$ satisfies condition \ref{enum:AZ-ample-1}.
    \item For any $\cX\in \cC$ and $n\in\ZZ$,
      there exist integers $l_1,l_2,\dotsc,l_p$ and an epimorphism
      $(f_i)\colon \bigoplus_{i=1}^p s^{-l_i}(\cM)\rightarrow \cX$
      such that $\{s^{l_1}(f_1), \dotsc, s^{l_p}(f_p)\} \subseteq \gHom_{\cC}(\cM,\cX)B_{\geq n}$.
    \item $\cM$ is a finite graded generator of $(\cC, s)$,
      and for any $n \in \ZZ$,
      there are some positive integers $l_1,l_2,\dotsc,l_p$
      and an epimorphism $(f_i): \bigoplus_{i=1}^ps^{-l_i}(\cM)\to s^n(\cM)$ in $\cC$ (then $\{s^{l_1}(f_1), \dotsc, s^{l_p}(f_p)\} \subseteq B_{\geq n}$).
  \end{enumerate}
\end{lemma}

\begin{proof}
  \textit{(1) $\Rightarrow $(3).} It is trivial.

  \textit{(3)$\Rightarrow$(1).}
  Fix $\cX\in\cC$. Then there exists an epimorphism $(g_j) \colon \bigoplus_j s^{-u_j}(\cM)\rightarrow \cX$.
  Now let $n=\max\{-u_j\}$; applying the second part of (3) to $\cX=s^n(\cM)$, we obtain an epimorphism $(f_k)\colon\bigoplus_{k}s^{-v_k}(\cM)\to s^n(\cM)$,
  where $v_k>0$ for all $k$.
  Then, $u_j+n+v_k > 0$, and the morphism
  \[(g_j\comp s^{-u_j-n}(f_k))_{jk}\colon\bigoplus_j\big(\bigoplus_k s^{-u_j-n-v_k} (\cM)\big)\to\cX\]
  is an epimorphism satisfying \ref{enum:AZ-ample-1}.

  \textit{(2) $\Rightarrow$(3).} Obviously, $\cM$ is a finite graded generator of $(\cC, s)$.  Since $B$ is bounded-below, in view of the reverse inclusion, the systems $\{B_{\geq n}\}$ and $\{BB_{\geq n}\}$ are cofinal. 
  Fix an integer $n$. Then there exists an $m$ such that $BB_{\geq m} \subseteq B_{\geq n+1}$. 
  By (2), there exist integers $l_1,l_2,\dotsc,l_p$ and an epimorphism
  $(f_i)\colon \bigoplus_{i=1}^p s^{-l_i}(\cM)\rightarrow s^n(\cM)$
  such that \[\{s^{l_1}(f_1), \dotsc, s^{l_p}(f_p)\} \subseteq \gHom_{\cC}(\cM, s^n(\cM))B_{\geq m} = BB_{\geq m} \subseteq B_{\geq n+1}.\]
  Since each $s^{l_i}(f_i)$ is a morphism from $\cM$ to $ s^{n+l_i}(\cM)$,
  it follows that $l_i>0$ for each $i$.

  \textit{(3)$\Rightarrow$(2).} Fix $\cX\in\cC$ and $n\in \ZZ$.
  There are epimorphisms
  \[(g_j)\colon\bigoplus_{j=1}^q s^{-u_j}(\cM)\rightarrow \cX \text{ and }
  (h_k) \colon\bigoplus_{k=1}^{r} s^{-v_{k}}(\cM)\rightarrow s^n(\cM)\]
  where each $v_{k}>0$. 
  Therefore, the composition
  \[(g_j)\comp (s^{-n-u_j}(h_k))=(g_j\comp s^{-n-u_j}(h_k))_{jk} \colon \bigoplus_{j=1}^{q}\big(\bigoplus_{k=1}^{r} s^{-v_{k}-n-u_j}(\cM)\big)\rightarrow \cX\]
  is an epimorphism,
  and $\{g_j \comp (s^{-n-u_j}(h_k))\}_{k,j}$ are the desired morphisms, as
  \[(s^{v_k+n+u_j}(g_j)) \comp (s^{v_k}(h_k))=(s^{u_j}(g_j)) (s^{v_k}(h_k)) \in \gHom_{\cC}(\cM,\cX)B_{\geq n}.
  \qedhere\]
\end{proof}

\ref{def:AZ-ample,lem:AZ-equiv-cond} motivate the following general definition in terms of Gabriel topologies,
which we present in the ungraded setting. The graded version is given in \ref{cor:AZ-thm}.

\begin{definition}\label{dfn:progen}
  Let $\cC$ be an abelian category, and let $\cM\in\cC$. Suppose that $\cG$ is a Gabriel topology on $B=\End_{\cC}(\cM)$.
  \begin{enumerate}
  [beginpenalty=10000,midpenalty=10000,ref={(\arabic*) in \ref{dfn:progen}}]
  \item  $\cM$ is called a \dfn{finite generator} of $\cC$, if for any $\cX\in\cC$, there exists an epimorphism $\cM^{\oplus r}\to\cX$ for some integer $r$.
    \item%\label{enum:finite-gen}
      $\cM$ is called \dfn{finitely $\cG$-self-generated} if, 
      for any $J\in\cG$, there exist $f_1, \dotsc, f_r\in J$
      such that $(f_i)\colon \cM^{\oplus r}\to\cM$ is an epimorphism.
    \item\label{enum:proj} $\cM$ is called \dfn{$\cG$-projective} if, for any epimorphism
      $\cX\rightarrow\cY$ in $\cC$,
      the cokernel of $\Hom_{\cC}(\cM,\cX)\to\Hom_{\cC}(\cM,\cY)$
      % $\coker\bigl(\Hom_{\cC}(\cM,\cX)\to\Hom_{\cC}(\cM,\cY)\bigr)$
      is $\cG$-torsion as a $B$-module.
   \item \label{enum:ample}
     $\cM$ is said to be \dfn{$\cG$-ample} if it is a finite generator, finitely $\cG$-self-generated and $\cG$-projective.
  \end{enumerate}
\end{definition}

Condition \ref{enum:AZ-ample-1} cannot be formulated in terms of Gabriel topologies.
Therefore,
we interpret it as condition (3) of \ref{lem:AZ-equiv-cond} and adopt this interpretation in the definition of a finite generator and finite $\cG$-self-generation; see (1) and (2) of \ref{dfn:progen}.
On the other hand, condition (2) in \ref{lem:AZ-equiv-cond} can be adapted in the language of the Gabriel topology $\cG$ as follows.

      ($\star$) \textit{For any $\cX\in \cC$ and any $J\in\cG$, there exist $f_1,\cdots,f_r\in \Hom_{\cC}(\cM,\cX)J$ such that $(f_i):\cM^{\oplus r}\to \cX$ is an epimorphism.}

However, this applies only to bounded Gabriel topologies $\cG$, namely those admitting a basis consisting of two-sided ideals. 
For example, when $\cX=\cM$ in ($\star$), the elements $f_1,\cdots,f_r$ need not belong to $J$ unless $J$ is a two-sided ideal.

Note that condition \ref{enum:AZ-ample-2} is stronger than the graded version of condition \ref{enum:proj} in the setting of \ref{cor:AZ-thm}.
In Artin--Zhang's noncommutative Serre theorem \cite[Thm.\@ 4.5]{artinNoncommutativeProjectiveSchemes1994}, this stronger condition guarantees that $B$ satisfies an additional condition, denoted by $\chi_1$.
Assuming that $\cM$ is graded $\cG$-ample, we provide in \ref{cor:AZ-thm} a variant of this theorem that does not involve condition $\chi_1$.

The notions of Artin--Zhang's ampleness and $\cG$-ampleness formulated here apply when every object of $\cC$ is noetherian.
The following notion of fully $\cG$-ample objects
is introduced for the setting in which $\cC$ is cocomplete, that is, when $\cC$ has set-indexed coproducts.

\begin{definition}\label{dfn:progen-large}
  Let $\cC$ be a cocomplete abelian category, and let $\cM\in\cC$. Suppose that $\cG$ is a Gabriel topology on $B=\End_{\cC}(\cM)$.
  \begin{enumerate}[ref=\ref{dfn:progen-large}(\arabic*)]
    \item%\label{enum:dfn-progen-large-gen}
      $\cM$ is called \dfn{$\cG$-self-generated} if
      for any $J\in\cG$ there exists a set of morphisms
      $\{f_\lambda\}_{\lambda\in\Lambda}\subseteq J$
      such that
      $(f_\lambda)\colon \cM^{\oplus \Lambda}\rightarrow \cM$
      is an epimorphism.
    \item $\cM$ is called \dfn{$\cG$-compact} if
      for any set of objects $\{\cX_\lambda\}$ in $\cC$,
      the canonical map
      $\bigoplus_\lambda\Hom_{\cC}(\cM,\cX_\lambda)
      \to\Hom_{\cC}(\cM,\bigoplus_\lambda\cX_\lambda)$
      has a $\cG$-torsion cokernel.
   \item $\cM$ is said to be \dfn{fully $\cG$-ample} if it is a generator of $\cC$ and is $\cG$-compact, $\cG$-projective, and $\cG$-self-generated.
  \end{enumerate}
\end{definition}

\begin{remark}
  Let $A$ be a ring, $\cC=\Mod A$, $\cM\in\cC$ and $\cG$ be a Gabriel topology on $B=\End_{A}(\cM)$.
  \begin{enumerate}
    \item $\cM$ admits a $B$-$A$-bimodule structure. In this case, $\cM$ is $\cG$-self-generated if and only if $J\cM=\cM$ for every $J\in\cG$. 
    If $\cM=A$, then $B\cong A$ as a right $A$-module,
    and $A$ is $\cG$-self-generated only if $\cG=\{A\}$ is trivial.
    \item When $\cG\defeq\{B\}$ is the trivial Gabriel topology, $\cM$ is $\cG$-projective if and only if $\cM$ is projective in $\cC$.
  \end{enumerate}
\end{remark}

%\subsection{Basic properties of \texorpdfstring{$\cG$}{G}-ampleness}

The following lemma justifies the terminology of $\cG$-projective and $\cG$-compact objects.
When $\cM$ is a (finite) generator, the property of $\cM$ being (finitely) $\cG$-self-generated is characterized later in \ref{prop:always-closed}, as some preliminary preparation is required.
 
\begin{lemma}\label{lem:A2-exact}
  Let $\cC$ be an abelian category, and let $\cM\in\cC$. Suppose that $\cG$ is a Gabriel topology on $B=\End_{\cC}(\cM)$.
  Then
  \begin{enumerate}
    \item $\cM$ is $\cG$-projective if and only if $\pi\comp\Hom_{\cC}(\cM,\blank)\colon \cC\to \QModG B$ is exact.
    \item Suppose that $\cC$ is cocomplete. $\cM$ is $\cG$-compact if and only if
      $\pi\comp\Hom_{\cC}(\cM,\blank)$ preserves coproducts.
  \end{enumerate}
\end{lemma}

\begin{proof}
  (1) Let $\overline{F}=\pi\comp\Hom_{\cC}(\cM,\blank)\colon \cC\to \QModG B$.
    Since $\overline{F}$ is left exact, $\overline{F}$ is exact if and only if $\pi\bigl(\coker(\Hom_{\cC}(\cM,f))\bigr)=0$ for any epimorphism $f\colon\cX\rightarrow \cY$;
    if and only if $\coker(\Hom_{\cC}(\cM,f))$ is $\cG$-torsion for any epimorphism $f\colon\cX\rightarrow \cY$; if and only if $\cM$ is $\cG$-projective.

  (2) Let $\{\cX_\lambda\}$
    be a set of objects in $\cC$.
    We first claim  that the canonical map
    \[\sigma \colon \bigoplus_\lambda\Hom_{\cC}(\cM,\cX_\lambda)
    \to\Hom_{\cC}(\cM,\bigoplus_\lambda \cX_\lambda)\]
    is always injective.
    Suppose that $\sigma(f)=0$ for some $f=(f_{\lambda})_{\lambda \in \Lambda} \in \bigoplus_\lambda\Hom_{\cC}(\cM,\cX_\lambda)$. There exists a finite subset $\Lambda_f \subseteq \Lambda$ such that $f_\lambda = 0$ for all $\lambda \in \Lambda \setminus \Lambda_f$.
    By the universal property of coproducts, one obtains the following canonical commutative diagram
    \begin{equation*}
      \begin{tikzcd}
        \bigoplus_{\lambda \in \Lambda_f} \Hom_{\cC}(\cM,\cX_\lambda) & \Hom_{\cC}(\cM,\bigoplus_{\lambda \in \Lambda_f} \cX_\lambda)\\
        \bigoplus_{\lambda \in \Lambda}\Hom_{\cC}(\cM,\cX_\lambda) & \Hom_{\cC}(\cM,\bigoplus_{\lambda \in \Lambda} \cX_\lambda)
        \arrow[from=1-1,]{1-2}[]{\cong} 
        \arrow[from=2-1,]{2-2}[]{\sigma}
        \arrow[from=1-1,rightarrowtail,]{2-1}[]{}
        \arrow[from=1-2,rightarrowtail,]{2-2}[]{}
      \end{tikzcd}
    \end{equation*}  
    It follows that $f=0$, and hence $\sigma$ is injective.  
    
    Since $\pi$ is exact and commutes with coproducts, 
    $\pi(\sigma)\colon \bigoplus_\lambda \overline{F}(\cX_\lambda)
    \to \overline{F}(\bigoplus_\lambda \cX_\lambda)$ is injective.
    Therefore, $\overline{F}$ commutes with coproducts
    if and only if $\pi(\sigma)$ is surjective for any $\{\cX_\lambda\}$;
    if and only if $\coker(\sigma)$ is $\cG$-torsion for any $\{\cX_\lambda\}$;
    if and only if $\cM$ is $\cG$-compact.
\end{proof}

Let $A$ be a ring, $\cF$ be a Gabriel topology on $A$.
Recall that the ring of quotients $A_{\cF}=\End_{\QModF A}(\pi(A))$
is endowed with the topology $\cF'$ induced by $\cF$.

\begin{proposition}\label{prop:pi-A-ample}
Let $A$ be a ring equipped with a Gabriel topology $\cF$.
Let
$\eta^A\colon A\longrightarrow A_{\cF}$
be the canonical morphism, and let $\cF'$ be the Gabriel
topology on $A_{\cF}$ induced by $\cF$. Then
$\pi_{\cF}(A)$ is fully $\cF'$-ample in
$\QModF A$.

Moreover, if $A$ is right noetherian, then
$\pi_{\cF}(A)$ is $\cF'$-ample in
$\qmodF A$.
\end{proposition}

\begin{proof}
Clearly, $\pi(A)$ is a generator of
$\QModF A$. We first show that it is
$\cF'$-self-generated. Let $J\in\cF'$ and set
$I=(\eta^A)^{-1}(J)$. By \ref{lem:top-change-of-rings}, $I\in\cF$, and hence the
inclusion $\sigma\colon I \hookrightarrow A$ induces an isomorphism
$\pi(\sigma)\colon \pi(I)\simeq\pi(A)$.

Choose an epimorphism
$f=(f_\lambda)_{\lambda\in\Lambda}\colon A^{\oplus\Lambda}\to I$.
Composing with $\pi(\sigma)$ yields
an epimorphism
$(\pi(\sigma\circ f_\lambda))_{\lambda\in\Lambda}\colon
  \pi(A)^{\oplus\Lambda}\to\pi(A).$
Recall that $\eta^A$ is defined by
\[
    A\cong \Hom_A(A,A)\to \dlim \Hom_A(I,A/\tau A)=A_{\cF}=\Hom_{A/\cF}(\pi(A),\pi(A)).
\]
Thus, the endomorphism 
\[\pi(\sigma\circ f_\lambda) =\eta^A(f_\lambda(1))\in \Hom_{A/\cF}(\pi(A),\pi(A)).\]
Since $f_\lambda(1)\in I$, 
$\eta^A(f_\lambda(1))\in J$. 
% It follows from the  diagram
%   \begin{equation*}
%     \begin{tikzcd}
%       \pi(f^\lambda) & \eta^I\comp f^\lambda & \eta^A(f^\lambda(1))\\[-2ex]
%       \Hom_{\QModF A}(\pi(A), \pi(I)) & \Hom_A(A, \omega \pi(I)) & \omega \pi(I)\\
%       \Hom_{\QModF A}(\pi(A), \pi(A)) & \Hom_A(A, \omega \pi(A)) & \omega \pi(A) &[-2.4em] A_{\cF}
%       \arrow[from=1-1,mapsto,]{1-2}[]{}
%       \arrow[from=1-2,mapsto,]{1-3}[]{}
%       \arrow[from=2-1,]{2-2}[]{\cong}
%       \arrow[from=2-2,]{2-3}[]{\cong}
%       \arrow[from=3-1,]{3-2}[]{\cong}
%       \arrow[from=3-2,]{3-3}[]{\cong}
%       % \arrow[from=2-1,equal,nfold,]{2-2}[]{}
%       % \arrow[from=2-2,equal,nfold,]{2-3}[]{}
%       % \arrow[from=3-1,equal,nfold,]{3-2}[]{}
%       % \arrow[from=3-2,equal,nfold,]{3-3}[]{}
%       \arrow[from=1-1,phantom,]{2-1}[sloped,]{\in}
%       \arrow[from=1-2,phantom,]{2-2}[sloped,]{\in}
%       \arrow[from=1-3,phantom,]{2-3}[sloped,]{\in}
%       \arrow[from=2-1,equal,nfold,]{3-1}[]{}
%       \arrow[from=2-2,equal,nfold,]{3-2}[]{}
%       \arrow[from=2-3,equal,nfold,]{3-3}[]{}
%       \arrow[from=3-3,equal,nfold,]{3-4}[]{}
%     \end{tikzcd}
%   \end{equation*}
It follows that each component
$\pi(\sigma\comp f_\lambda)$ belongs to $J$, and consequently, $\pi(A)$ is
$\cF'$-self-generated.

Recall that, by \ref{enum:eta_A-induces-equi-i}, $\overline{(\eta^A)}_!$ is an equivalence.
Therefore, for any $\pi_{\cF}(M)\in \QModF A$, there are isomorphisms
\begin{align*}
  &\phantom{{}\cong{}}\pi_{\cF'}\Hom_{\QModF A}\bigl(\pi_{\cF}(A),\pi_{\cF}(M)\bigr)\\
  &\cong \pi_{\cF'}\Hom_{\QModFF {A_{\cF}}}\bigl(\overline{(\eta^A)}_!\pi_{\cF}(A),\overline{(\eta^A)}_!\pi_{\cF}(M)\bigr)\\
  &\cong \pi_{\cF'}\Hom_{\QModFF {A_{\cF}}}\bigl(\pi_{\cF'}(A_{\cF}),\pi_{\cF'}(M\otimes_A A_{\cF})\bigr)\\
  &\cong \pi_{\cF'}\omega_{\cF'}\pi_{\cF'}(M\otimes_A A_{\cF})\\
  &\cong \pi_{\cF'}(M\otimes_A A_{\cF})\\
  &\cong \overline{(\eta^A)}_!(\pi_{\cF}(M)),
\end{align*}
where the second and last isomorphisms follow from \ref{enum:eta_A-induces-equi-ii}.
Note that $\overline{(\eta^A)}_!$ is exact and preserves coproducts.
It then follows from \ref{lem:A2-exact} that $\pi(A)$ is both $\cF'$-projective and $\cF'$-compact.
Therefore, $\pi(A)$ is fully $\cF'$-ample in $\QModF A$.

Now suppose that $A$ is right noetherian. 
Since $A\in\mod{A}$ is a finite generator,
$\pi(A)$ is a finite generator of $\qmodF A$. 
Moreover, the epimorphism
$\pi(A)^{\oplus\Lambda}\to\pi(A)$ constructed above restricts to an epimorphism on some finite
subcoproduct, since $\pi(A)$ is noetherian by \ref{lem-qmodG consists of noetherian objects}.
Thus, $\pi(A)$ is finitely
$\cF'$-self-generated, and hence $\pi(A)$ is $\cF'$-ample in $\qmodF A$.
\end{proof}

\subsection{A characterization of noetherian abelian categories}
 
We first establish several properties of finitely $\cG$-self-generated objects, adapted from corresponding results in the graded setting \cite{artinNoncommutativeProjectiveSchemes1994}.
 
\begin{lemma}\label{lem:morita-i-1}
  Let $\cC$ be an abelian category.
  Suppose that $\cM\in\cC$, $B=\End_{\cC}(\cM)$, and $\cG$ is a Gabriel topology on $B$.
  If $\cM$ is finitely $\cG$-self-generated, then for any subobject $\cY \rightarrowtail \cX$ in $\cC$, 
  the quotient module
  \begin{equation*}
    \Hom_{\cC}(\cM,\cX)/\Hom_{\cC}(\cM,\cY)
  \end{equation*}
  is $\cG$-torsion-free.
  In particular, $\Hom_{\cC}(\cM,\cX)$ is $\cG$-torsion-free for every $\cX\in \cC$.
\end{lemma}

\begin{proof}
  Suppose $g\in\Hom_{\cC}(\cM, \cX)$ 
  such that $g\, J\subseteq\Hom_{\cC}(\cM,\cY)$ for some right ideal $J\in\cG$.
  Since $\cM$ is finitely $\cG$-self-generated,
  there is an epimorphism $f=(f_i)\colon \cM^{\oplus r}\rightarrow \cM$
  for some $f_1,\dotsc, f_r\in J$.
  % where $f=(f_i)$ denotes the canonical morphism induced by $f_i: M \to M \,(i=1, \ldots, r)$.
  % Let $\iota_i\colon\cM\to\cM^{\oplus r}$ be the canonical
  % inclusion of the $i$-th summand, so that $f\iota_i=f_i$.
  % Then the image of the composition $g\comp (f_i)\colon\cM^{\oplus r}\to\cX$
  % is a subobject of $\cY$.
  Consider the morphism $gf =(gf_i)\colon\cM^{\oplus r}\to\cX$.
  % Since $gf_i \in \Hom_{\cC}(\cM,\cY)$ for all $i$,
  % let $gf_i|\colon\cM\to\cY$ denote the corresponding morphism
  % whose composite with $\cY\rightarrowtail\cX$ is $gf_i$.
  Let $gf_i|\colon\cM\to\cY$ denote the morphism
  whose composite with $\cY\rightarrowtail\cX$ is $gf_i$.
  By the universal property of coproducts, $gf$ factors through $\cY\rightarrowtail\cX$,
  as illustrated in the following commutative diagram.
  \begin{equation*}
    \begin{tikzcd}[column sep=large]
      \cM & & \cX\\
      & \cM^{\oplus r} &\\
      \cM & & \cY
      \arrow[from=1-1,]{1-3}[]{g}
      \arrow[from=2-2,]{1-1}[sloped,]{f=(f_i)}
      \arrow[from=2-2,]{1-3}[sloped,]{gf}
      \arrow[from=2-2,]{3-3}[swap,sloped,]{(gf_i|)}
      \arrow[from=3-1,]{1-1}{f_i}
      \arrow[from=3-1,]{2-2}[sloped,]{\textrm{cano.}}
      \arrow[from=3-1,]{3-3}[ ]{gf_i|}
      \arrow[from=3-3,rightarrowtail,]{1-3}[]{}
    \end{tikzcd}
  \end{equation*}
  Thus the image of $gf$ is a subobject of $\cY$. % 这里可以引用[Po] Prop. 2.6.7(a)
  Since $f$ is an epimorphism, 
  \[\im(g)=\ker\bigl(\coker(g)\bigr)\cong\ker\bigl(\coker(gf)\bigr)=\im(gf)\]
  is a subobject of $\cY$.
  Therefore, $g$ factors through $\cY \rightarrowtail \cX$, that is, $g\in\Hom_{\cC}(\cM, \cY)$. Thus, $\Hom_{\cC}(\cM,\cX)/\Hom_{\cC}(\cM,\cY)$ is $\cG$-torsion-free.
\end{proof}

% 小版本的证明
\begin{lemma}\label{lem:morita-i-2}
  Let $\cC$ be an abelian category, $\cM\in\cC$ be a finite generator of $\cC$.
  Suppose that $B=\End_{\cC}(\cM)$ is a right noetherian ring and $\cG$ is a Gabriel topology on $B$ such that $\cM$ is finitely $\cG$-self-generated. 
  If $\cX\in\cC$ such that 
  $\Hom_{\cC}(\cM,\cX)$ is a submodule of some $\cG$-torsion-free $B$-module $Y$, then $Y/\Hom_{\cC}(\cM,\cX)$ is also $\cG$-torsion-free.
\end{lemma}
\begin{proof}
  Suppose that the image of $y\in Y$ in $Y/\Hom_{\cC}(\cM,\cX)$ is $\cG$-torsion, that is,
  there is a right ideal $J\in\cG$ such that $y J\subseteq\Hom_{\cC}(\cM,\cX)$.
  It suffices to show $y\in\Hom_{\cC}(\cM,\cX)$.
  Since $\cM$ is finitely $\cG$-self-generated and a finite generator, 
  there exists an exact sequence
  \begin{equation}\label{eq:pf-lem-morita-i-2}
    \begin{tikzcd}
      \cM^{\oplus s} \arrow[r, "g=(g_{ij})"]
      & \cM^{\oplus r} \arrow[r, "f=(f_i)"] & \cM \arrow[r] & 0
    \end{tikzcd}
  \end{equation}
  where $r,s\in\mathbb{N}$, $f_1,\dotsc, f_r\in J$,
  and $g_{ij}\in B$ denotes the component of $g$ from the $j$-th summand of $\cM^{\oplus s}$ to the $i$-th summand of $\cM^{\oplus r}$.
  After adjoining generators of $J$ if necessary,
  one may assume that $f_1, \dotsc, f_r$ generate $J$
  as a right ideal.
  Let $\iota_i\colon\cM\to\cM^{\oplus r}$
  and $\kappa_j\colon\cM\to\cM^{\oplus s}$
  denote the canonical morphisms from the $i$-th and $j$-th summands, respectively, into the corresponding coproducts.
  Let $\Phi\defeq(y\, f_i)\colon \cM^{\oplus r}\to\cX$ be the canonical morphism induced by $y\, f_i\in\Hom_{\cC}(\cM,\cX)$.

  Applying $\Hom_{\cC}(\blank,\cX)$ to \ref{eq:pf-lem-morita-i-2}
  yields the exact sequence
  \begin{equation*}
    \begin{tikzcd}
      0 \arrow[r] & \Hom_{\cC}(\cM,\cX) \arrow[r, "f^*"]
      & \Hom_{\cC}(\cM^{\oplus r},\cX) \arrow[r, "g^*"]
      & \Hom_{\cC}\bigl(\cM^{\oplus s},\cX\bigr).
    \end{tikzcd}
  \end{equation*}
  Since $\Phi\iota_i=y\, f_i$ and $fg=0$, it follows that
  \[
    \Phi g\kappa_j = \sum_{i=1}^r\Phi\iota_i g_{ij}
    =y\Bigl(\sum_{i=1}^r f_ig_{ij}\Bigr)
    = y(fg\kappa_j)=0, \quad\text{for each } j=1,\dotsc, s.
  \]
  Hence, $g^*(\Phi)=\Phi g=0$.
  By exactness, there exists $h\in\Hom_{\cC}(\cM,\cX)$ such that
  $h\comp f=\Phi=(y\, f_i)_{i=1,\dotsc,r}$. Therefore,
  $h\comp f_i=y\cdot f_i$ for $i=1,\dotsc, r$.
  Consequently,
  \begin{equation*}
    (y-h)\, J
    =(y-h)\Bigl(\sum_i f_i B\Bigr)=0,
  \end{equation*}
  which implies that $y-h\in Y$ is $\cG$-torsion.
  Since $Y$ is $\cG$-torsion-free by assumption, it follows that $y=h\in\Hom_{\cC}(\cM,\cX)$,
  as desired.
\end{proof}

% 大版本
% \begin{proposition}\label{prop:always-closed}
%   Let $\cC$ be a Grothendieck category.
%   Suppose $\cM\in\cC$, $B=\End_{\cC}(\cM)$. If $\cG$ is a Gabriel topology on $B$ such that $\cM$ is a $\cG$-generator,
%   then $\Hom_{\cC}(\cM,\cX)$ is $\cG$-closed for any $\cX\in\cC$.
% \end{proposition}

The following proposition provides a characterization of the finitely $\cG$-self-generated property.

\begin{proposition}\label{prop:always-closed}
  Let $\cC$ be an abelian category, $\cM\in\cC$ be a finite generator of $\cC$.
  Suppose that $B=\End_{\cC}(\cM)$ is a right noetherian ring and  $\cG$ is a Gabriel topology on $B$. Then $\cM$ is finitely $\cG$-self-generated if and only if $\Hom_{\cC}(\cM,\cX)$ is $\cG$-closed for any $\cX\in\cC$. In this case, $B$ is $\cG$-closed.
\end{proposition}

\begin{proof}
  For the ``only if" part, note that
  $\eta\colon \Hom_{\cC}(\cM, \cX)\to
  \omega\pi\bigl(\Hom_{\cC}(\cM, \cX)\bigr)$
  has kernel and cokernel that are $\cG$-torsion.
 By
  \ref{lem:morita-i-1}, $\eta$ is injective.
  By \ref{lem:morita-i-2}, $\eta$  is surjective.

  For the “if” part, let $J\in\cG$. 
  Suppose $J$ is generated by $f_1,\dotsc,f_r$ as a right $B$-module.
  Consider the canonical morphism
  $(f_i)\colon \cM^{\oplus r}\to \cM$, and $q\colon\cM\to\cM/\cN$ where $\cN$ denotes the image of $(f_i)$. Then $q\circ f_i=0$ for every $i=1,\dotsc,r$. % 这里的原因是 \im (f_i) = \sum_i \im f_i, 对无限的时候也对, 只需要相应的余积存在
It follows that $qJ=0$ in the $B$-module $\Hom_{\cC}(\cM,\cM/\cN)$, that is,
$q$ is a $\cG$-torsion element of $\Hom_{\cC}(\cM,\cM/\cN)$,
which is however $\cG$-closed and in particular $\cG$-torsion-free by assumption.
Therefore, $q=0$, and it follows that $\cN=\cM$. 
Consequently, $(f_i)\colon\cM^{\oplus r}\to\cM$ is an epimorphism, and thus, $\cM$ is finitely $\cG$-self-generated.
\end{proof}

We are now in a position to characterize when an abelian category is a quotient category of a module category in terms of $\cG$-ampleness; this constitutes the main result of this section.

\begin{theorem}\label{thm:morita-i-finite}
  Let $\cC$ be an abelian category.
  Suppose $\cM\in\cC$ such that $B\defeq\End_{\cC}(\cM)$ is right noetherian, and there exists a Gabriel topology $\cG$ on $B$ such that $\cM$ is $\cG$-ample.
  Then $\cC$ is equivalent to $\qmodG{B}$.
\end{theorem}

\begin{proof}%[Proof of \ref{thm:morita-i}]
  By \ref{prop:always-closed}, $B$ is $\cG$-closed. 
  It follows that $\End_{B/\cG}(\pi(B))\cong B=\End_{\cC}(\cM)$ is right noetherian.
  By \ref{lem-qmodG consists of noetherian objects}, $\qmodG B$ is the noetherian full subcategory of $\QModG B$. 
  It remains to prove $\cC$ is equivalent to $\qmodG B$.

  Let $\overline{F}=\pi\comp\Hom_{\cC}(\cM,\blank)$.
  Then $\overline{F}$ is exact by \ref{lem:A2-exact} and the assumption.
  Since $\cM$ is a finite generator,
  for every $\cX\in\cC$,
  there exists an epimorphism $f\colon\cM^{\oplus r}\to\cX$,
  and thus $\overline{F}(f)\colon \pi(B)^{\oplus r}\to\overline{F}(\cX)$
  is epic.
  Recall that $\qmodG B$ consists of noetherian objects of $\QModG B$.
  It follows that $\overline{F}(\cX)\in\qmodG{B}$.
  We are going to show that
  $\overline{F}\colon \cC\to\qmodG B$ is fully faithful and essentially surjective.
  The proof is divided into three steps.

  \textit{Step 1: $\overline{F}$ is fully faithful.}\nopagebreak\par
  Recall that $\Hom_{\cC}(\cM, \cY)$ is $\cG$-closed for any $\cY\in\cC$ by
  \ref{prop:always-closed}. Then
  there are natural isomorphisms
  \begin{equation*}
    \begin{aligned}
      \Hom_{\cC}(\cM,\cY)
      &\cong \Hom_B\bigl(\Hom_{\cC}(\cM,\cM),\Hom_{\cC}(\cM,\cY)\bigr)\\
      &\cong \Hom_{\QModG B}\bigl(\pi\bigl(\Hom_{\cC}(\cM,\cM)\bigr),
        \pi\bigl(\Hom_{\cC}(\cM,\cY)\bigr)\bigr)\\
      &= \Hom_{\QModG B}\bigl(\overline{F}(\cM),\overline{F}(\cY)\bigr).
    \end{aligned}
  \end{equation*}
  The isomorphism is in fact given by $f \mapsto \overline{F}(f)$.
  Now, for any $\cX\in\cC$, there is a presentation
  $
    \cM^{\oplus s}\to\cM^{\oplus r}\to\cX\to 0
  $
  as $\cM$ is a finite generator in $\cC$.
  Since $\overline{F}$ is exact,
  one obtains the following exact commutative diagram.
  \begin{equation*}
    \begin{tikzcd}[cramped,column sep=1em]
      0 & \Hom_{\cC}(\cX,\cY)
      & \Hom_{\cC}(\cM^{\oplus r},\cY)
      & \Hom_{\cC}(\cM^{\oplus s},\cY)\\
      0 & \Hom_{\QModG B}\bigl(\overline{F}(\cX),\overline{F}(\cY)\bigr)
      & \Hom_{\QModG B}\bigl(\overline{F}(\cM^{\oplus r}),\overline{F}(\cY)\bigr)
      & \Hom_{\QModG B}\bigl(\overline{F}(\cM^{\oplus s}),\overline{F}(\cY)\bigr)
      \arrow[from=1-1,]{1-2}[]{}
      \arrow[from=1-2,]{1-3}[]{}
      \arrow[from=1-3,]{1-4}[]{}
      \arrow[from=2-1,]{2-2}[]{}
      \arrow[from=2-2,]{2-3}[]{}
      \arrow[from=2-3,]{2-4}[]{}
      \arrow[from=1-2,]{2-2}[]{\overline{F}}
      \arrow[from=1-3,]{2-3}[]{\cong}
      \arrow[from=1-4,]{2-4}[]{\cong}
    \end{tikzcd}
  \end{equation*}
  It follows that
  $\overline{F}\colon\Hom_{\cC}(\cX,\cY)\to\Hom_{\QModG B}\bigl(\overline{F}(\cX),\overline{F}(\cY)\bigr)$
  is an isomorphism.

  \textit{Step 2: $\overline{F}$ gives an equivalence: $\add(\cM)\simeq\add(\pi(B))=\pi(\proj(B))$.}%
  \nopagebreak\par
  This is analogous to the projectivization equivalence
  \cite[Prop.\@ 2.7]{auslanderRepresentationTheoryArtin1974}.
  Recall that $\add(\cM)$ denotes the full subcategory of $\cC$
  consisting of all direct summands of finite coproducts of $\cM$,
  whereas $\proj(B)$ is the full subcategory of $\mod{B}$
  consisting of all finitely generated projective modules.
  Clearly, $\overline{F}$ restricts to a functor $\add(\cM)\to \add(\pi(B))$, which indeed yields an equivalence of categories as  explained below. Since $\overline{F}$ is fully faithful, it suffices to find
  some $\cN\in\add(\cM)$ such that $\overline{F}(\cN)\cong \pi(P)$ for any $P\in\proj(B)$.
  For such a $P$, there is an idempotent $e\in \End(B^{\oplus r})$
  such that $P\cong \ker(e)$. Then
  $\pi(e)=\pi(e)^2 \in \End({\pi(B)}^{\oplus r})$.
  Since $\overline{F}$ is fully faithful and preserves finite coproducts,
  there is an idempotent $f\in\End(\cM^{\oplus r})$
  such that $\overline{F}(f)=\pi(e)$.
  Let $\cN\defeq\ker(f)$.
  It follows from the exactness of $\overline{F}$ that $\overline{F}(\cN)\cong\pi(P)$.

  \textit{Step 3: $\overline{F}\colon \cC \to \qmodG{B}$ is essentially surjective.} 
  \nopagebreak\par
  Suppose $Y\in\mod B$.
  Let $P_1 \xrightarrow{\partial} P_0 \to Y \to 0$
  be a finite $B$-projective presentation of $Y$.
  By Step 2,
  there is a morphism $f\colon\cM_1\to\cM_0$ in $\add(\cM)$
  such that $\pi(\partial)=\overline{F}(f)$.
  Let $\cX=\coker(f)$. As $\overline{F}$ is right exact,
  one obtains the following commutative diagram with exact rows,
  \begin{equation*}
    \begin{tikzcd}
      \overline{F}(\cM_1) & \overline{F}(\cM_0) & \overline{F}(\cX) & 0\\
      \pi(P_1) & \pi(P_0) & \pi(Y) & 0
      \arrow[from=1-1,]{1-2}[]{\overline{F}(f)}
      \arrow[from=1-2,]{1-3}[]{}
      \arrow[from=1-3,]{1-4}[]{}
      \arrow[from=2-1,]{2-2}[]{\pi(\partial)}
      \arrow[from=2-2,]{2-3}[]{}
      \arrow[from=2-3,]{2-4}[]{}
      \arrow[from=1-1,]{2-1}[]{\cong}
      \arrow[from=1-2,]{2-2}[]{\cong}
      \arrow[from=1-3,dashed,]{2-3}[]{}
    \end{tikzcd}
  \end{equation*}
  which implies that $\overline{F}(\cX)\cong\pi(Y)$.
\end{proof}

% 在这里提一下大版本我们也有
\begin{remark}\label{rem:pg}
  Let $\cC$ be an abelian category. Suppose that $\cM\in\cC$ is a finite generator and $B\defeq\End_{\cC}(\cM)$ is right noetherian.
  Then $\cC$ is essentially small.
  Therefore the indization of $\cC$, denoted by $\Ind(\cC)$, is a Grothendieck category, and $\cM$ is a generator of $\Ind(\cC)$ (\cite[Sect.\@ 8.6]{kashiwaraCategoriesSheaves2006}).
  We treat $\cC$ as a full abelian subcategory of $\Ind(\cC)$ via the canonical exact fully faithful functor $\cC\to\Ind(\cC)$.
  By the Popescu--Gabriel theorem \cite{popescuCaracterisationCategoriesAbeliennes1964},
  there is a strongest Gabriel topology $\cG_{\text{str}}$ on $B$ such that $\Hom_{\Ind(\cC)}(\cM, \cX)$ is $\cG_{\text{str}}$-closed for all $\cX\in\Ind(\cC)$,
  and the functor $\pi\comp\Hom_{\Ind(\cC)}(\cM,\blank)$ defines an equivalence $\Ind(\cC)\simeq\QModG[\text{str}]{B}$.
  Under this equivalence, $\cM$ corresponds to $\pi(B)$.
  Since $\pi(B)\in\QModG[\text{str}]{B}$ is noetherian, so is $\cM$ in $\cC$.
  It follows that $\cC$ is a noetherian category because $\cM$ is a finite generator. Therefore, $\cC$ is the full subcategory of $\Ind(\cC)$ consisting of noetherian objects,
  and $\pi\comp\Hom_{\cC}(\cM,\blank)$ defines an equivalence $\cC\simeq\qmodG[\text{str}]{B}$.
  Since $B$ is $\cG_{\text{str}}$-closed, \ref{prop:pi-A-ample} shows that $\cM\in\cC$ is $\cG_{\text{str}}$-ample.
  Conversely, if $\cG$ is a Gabriel topology on $B$ such that $\cM$ is $\cG$-ample, then $\cG=\cG_{\text{str}}$.

  However, determining whether a given Gabriel topology $\cG$ on $B$ coincides with $\cG_{\text{str}}$ can be highly challenging. By
  \ref{thm:morita-i-finite}, it suffices to verify that $\cM$ is finitely $\cG$-self-generated and $\cG$-projective.
  For example, in the graded setting, ampleness conditions are verified and used to establish equivalences with quotient categories in \cite[Lem.\@ 3.5 and Cor.\@ 3.6]{minamotoAmplenessTwosidedTilting2012}, \cite[Thm.\@ 2.5]{moriAmpleGroupAction2016} and \cite[Thm.\@ 7.9]{liNoncommutativeResolutionsNoncommutative2026}.
\end{remark}

\iffalse
When $B$ is right noetherian, $\QModG B$ is locally noetherian
% 引用的是locally noetherian范畴的定义: 有一族noetherian生成子的Grothendieck范畴
(\cite[Sect.\@ 5.8]{popescuAbelianCategoriesApplications1973}),
and $\qmodG B$ is its full subcategory of noetherian objects (\ref{lem-qmodG consists of noetherian objects}).
Since a locally noetherian category is determined up to equivalence by its full subcategory of noetherian objects,
\ref{thm:morita-i-finite} also has a Grothendieck version,
with $\cG$-ampleness replaced by full $\cG$-ampleness and $\qmodG B$ by $\QModG B$.
\fi

% 这边写分次 qgr情况的推论, 解释和AZ的关系与区别

% 这边给一个类似于AZ定理的版本
% 我们直接说没介绍的记号参考AZ

We now return to the graded setting that originally motivated
the definition of ampleness.
We provide a characterization of $\qgr(B)$
for a graded ring $B$,
making use of a graded version of \ref{thm:morita-i-finite}
by fixing a specific graded Gabriel topology.
This is analogous to
Artin--Zhang's theorem \cite[Thm.\@ 4.5]{artinNoncommutativeProjectiveSchemes1994}.
% The reader is referred to \cite{artinNoncommutativeProjectiveSchemes1994}
% for any undefined concept.
In particular, comparing with \cite[Thm.\@ 4.5]{artinNoncommutativeProjectiveSchemes1994}, we require an extra condition (H2) in \ref{cor:AZ-thm}, which naturally arises in the study of noncommutative resolutions of noncommutative isolated singularities.
In \cite[Thm.\@ 3.9]{liNoncommutativeResolutionsASGorenstein2026},
the authors also characterize $\qgr (B)$ under the hypothesis (H2) in \ref{cor:AZ-thm}, where the original noncommutative ampleness (\ref{def:AZ-ample}) is used, resulting in the $\chi_1$-condition as mentioned before.
 % 加了这一句话回顾一下我们的ample和AZ ample的关系
By \ref{lem:AZ-equiv-cond}, condition (H1a)
in \ref{cor:AZ-thm} is equivalent to condition \ref{enum:AZ-ample-1}.
% For graded Gabriel topologies, we refer the reader to
% \cite[Sects.\@ II.9 and II.10]{nastasescuGradedRingTheory1982}.

We briefly recall the notation in
\cite{artinNoncommutativeProjectiveSchemes1994,liNoncommutativeResolutionsASGorenstein2026}.
Let $B$ be a right noetherian $\ZZ$-graded ring, and let $\cT$ denote the category of  graded $B$-modules $N$ such that every element is annihilated by some $B_{\geq n}$.
% An element $y\in N$ is said to be \dfn{torsion} if $y\cdot B_{\geq n}=0$ for some $n\in\ZZ$.
% The module $N$ is called \dfn{torsion} if every element $x\in N$ is torsion.
As a full subcategory of the category of graded $B$-modules, $\cT$ is closed under subobjects, quotient objects, extensions, coproducts, and shifts.
Hence, $\cT$ corresponds to a graded Gabriel topology $\cG$ on $B$ such that $\cT$ is precisely the category of $\cG$-torsion modules (\cite{nastasescuGradedRingTheory1982}).
Then, $\qgr (B)$ is defined as the quotient category of finitely generated graded $B$-modules modulo the $\cG$-torsion modules.

When $B$ is additionally bounded-below,
$\cG$ has a basis consisting of the graded two-sided ideals $\{BB_{\geq n}B\mid n\in\ZZ\}$.
Let $l$ denote the lower bound of $B$.
Then, for any $n\in \ZZ$, one has $ BB_{\geq n-2l}B \subseteq B_{\geq n}$.
It follows that $\{B_{\geq n}\mid n\in\mathbb{Z}\}$ is cofinal to $\cG$ with respect to reverse inclusion.
If $B$ is moreover locally finite, then, by
\cite[Lem.\@ 2.6]{liNoncommutativeResolutionsASGorenstein2026}, $\cG$ has a basis $\{J^n\mid n\in \NN\}$,
where $J$ denotes the graded Jacobson radical of $B$.
In this case, the $\cG$-depth coincides with the classical notion of depth (\cite[p.\@ 8]{liNoncommutativeResolutionsASGorenstein2026}).

\begin{corollary}\label{cor:AZ-thm}
  Let $\cC$ be an abelian category, let
  $\cM\in\cC$, and $s$ be an autoequivalence of $\cC$.
  Suppose that $B\defeq\gHom_{\cC}(\cM, \cM)$ is a right noetherian ring, and let $\cG$ denote the graded Gabriel topology on $B$ defined above.
  Assume that the following conditions hold for the triple $(\cC,\cM,s)$.
  \begin{enumerate}[label=(H\arabic*)]
    \item $\cM$ is graded $\cG$-ample, that is
    \begin{enumerate}
      \item[(a)] $\cM$ is a finite graded generator of $(\cC, s)$,
      and for any $n \in \ZZ$,
      there are some positive integers $l_1,l_2,\dotsc,l_p$
      and an epimorphism $\bigoplus_{i=1}^ps^{-l_i}(\cM)\to s^n(\cM)$ in $\cC$.
      \item[(b)] For any epimorphism $\cX\to\cY$ in $\cC$,
        \[\coker\bigl(\gHom_{\cC}(\cM,\cX)\to
        \gHom_{\cC}(\cM,\cY)\bigr)\]
        is a $\cG$-torsion graded $B$-module.
    \end{enumerate}
    \item $\Hom_{\cC}(\cM, s^{-n}(\cM))=0$ for sufficiently large $n$.
  \end{enumerate}
  Then $\depth^{\cG}(B)\geq 2$,
  and there is an equivalence $F\colon \cC\simeq \qgr (B)$ such that $F(\cM)=\pi(B)$ and $F\comp s = (1)\comp F$, which is denoted by
   \[F\colon (\cC,\cM,s) \simeq (\qgr(B),\pi(B),(1)).\]

  Conversely, for a right noetherian bounded-below $\ZZ$-graded ring $B$ with $\depth^{\cG}(B) \geq 2$, the triple $(\qgr(B),\pi(B),(1))$ satisfies conditions {\rm(H1)} and {\rm(H2)}.
\end{corollary}

\begin{proof}
  Assume that (H1) and (H2) hold for $(\cC,\cM,s)$. 
  % H1就是分次版本的G-ample
  Graded versions of \ref{prop:always-closed} and \ref{thm:morita-i-finite}
  then imply that $B$ is $\cG$-closed
  and that the canonical functor defines an equivalence
  \[\pi\comp\gHom_{\cC}(\cM, \blank)\colon (\cC,\cM,s) \simeq (\qgr(B),\pi(B),(1)).
  \]

  Conversely, suppose that $B$ is a right noetherian bounded-below $\ZZ$-graded ring with $\depth^{\cG}(B) \geq 2$.
  It follows from $\depth^{\cG}(B) \geq 2$ that $B$ is graded $\cG$-closed and thus (H2) holds. Moreover, by a graded version of \ref{prop:pi-A-ample}, $(\qgr(B),\pi(B),(1))$ satisfies (H1)(a) and (H1)(b), that is, $\pi(B)$ is graded $\cG$-ample.
\end{proof}

\section{Quotient equivalences and Morita contexts}
\label{Quotient Equivalences and Morita Contexts}
\label{sec:morita}

In this section, we investigate the conditions under which a Morita context induces an equivalence between quotient categories. We characterize such equivalences and develop a Morita-type theory for quotient categories.

\subsection{Morita contexts and quotient categories}

We begin by recalling the definition of Morita contexts.
% \cite[Def.\@ I.3.2]{bassAlgebraicKtheory1968}
\begin{definition}[{{\cite[Def.\@ 3.11]{jacBasicAlgebraII1989}}}]
  A \emph{Morita context} is a six-tuple $(A, B, M, N, \xi, \zeta)$
  consisting of two rings $A$ and $B$,
  two bimodules ${}_BM_A$ and ${}_AN_B$,
  and two bimodule morphisms $\xi\colon N\otimes_B M\to A$ and     $\zeta\colon M\otimes_A N\to B$
  satisfying the associativity conditions
  \[[m',n]m=m'(n,m) \text{\quad and\quad} n[m,n']=(n,m)n'\]
  % $\zeta(m'\otimes n)m=m'\xi(n\otimes m)$ and $n\zeta(m\otimes n')=\xi(n\otimes m)n'$
  for all $m,m'\in M$ and $n,n'\in N$,
  where $(n, m)\defeq\xi(n\otimes m)$ and $[m,n]\defeq\zeta(m\otimes n)$.
\end{definition}

Bimodules whose tensor functors induce functors between the corresponding quotient categories are frequently used; accordingly, we introduce the following terminology.

\begin{definition}\label{def:compatible}
  Let $A$ and $B$ be rings, $\cF$ and $\cG$ be Gabriel topologies on $A$ and $B$, respectively. A bimodule ${}_BM_A$ is said to be \emph{compatible} with $\cF$ and $\cG$ if both $\blank\otimes_B M$ and $\Tor_1^B(\blank, M)$
  send $\cG$-torsion modules to $\cF$-torsion modules.
\end{definition}

By \ref{prop:induced-functors},
${}_BM_A$ is compatible with the Gabriel topologies $\cF$ and $\cG$ if and only if
$\blank\otimes_B M$ induces a functor from 
$\QModG B$ to $\QModF A$ (see \ref{def:induced-functors}), 
which is denoted by 
\[\blank\otimes_{\QModG B}\pi(M)\colon \QModG B\to\QModF A.\]

Let $(A, B, M, N, \xi, \zeta)$ be a Morita context, and let $\cF$ and $\cG$ be Gabriel topologies on $A$ and $B$, respectively, with which $M$ and $N$ are compatible. 
If $\overline F=\blank\otimes_{\QModF A} \pi(N)$ and $\overline G=\blank\otimes_{\QModG B} \pi(M)$ 
are mutually quasi-inverse 
via the natural isomorphisms
$\overline G \, \overline F\cong \id_{\QModF A}$ and $\overline F\, \overline G\cong \id_{\QModG B}$ induced by $\pi(\blank \otimes \xi)$ and $\pi(\blank \otimes \zeta)$,
then the equivalence 
is called \textit{induced by the Morita context} $(A, B, M, N, \xi, \zeta)$.

The following lemma is useful for establishing equivalences between quotient categories induced by a Morita context.

\begin{lemma}
  \label{lem:context-ker-torsion}
  Let $(A,B,M,N,\xi,\zeta)$ be a Morita context.
  Suppose that $\cF$ and $\cG$ are Gabriel topologies on $A$ and $B$, respectively.
  If $\coker(\xi)$ is $\cF$-torsion, then 
  \begin{enumerate}[ref=\ref{lem:context-ker-torsion}(\arabic*)]
    \item\label{enum:ker-torsion} {\normalfont(\cite[Lem.\@ 2.12]{liNoncommutativeResolutionsASGorenstein2026})}
   $\ker(\xi)$ is $\cF$-torsion;
    \item\label{enum:tor-torsion} $\Tor_i^B(Y, M)$ is $\cF$-torsion for any $B$-module $Y$ and any $i\geq 1$.
    \item The kernel and cokernel of the canonical ring morphism
    $\phi \colon A\to \End_B(N)$ are both $\cF$-torsion.
  \end{enumerate}
\end{lemma}

\begin{proof} Let $I=\im(\xi)$. Then $I \in \cF$, since $A/I=\coker(\xi)$ is $\cF$-torsion.

 (1) Suppose that $\sum_j y_j \otimes x_j \in \ker(\xi)$, that is, $\sum_j (y_j, x_j)=0$. Now, for any $a \in I$ with $a=\sum_i (n_i, m_i)$,
 \[\Bigl(\sum_j y_j \otimes x_j\Bigr)a = \Bigl(\sum_j y_j \otimes x_j\Bigr) \Bigl(\sum_i (n_i, m_i)\Bigr) =  \sum_{i,j} y_j \otimes [x_j, n_i]m_i =\sum_{i,j} (y_j, x_j)n_i \otimes m_i =0.\]
 It follows that $\ker(\xi)$ is $\cF$-torsion.

 (2) In fact, the stronger statement $\Tor_i^B(Y, M) I = 0$ holds, thereby completing the proof.
 Let $a=\sum_{j=1}^r (n_j, m_j) \in I$. Then $r_a\colon M \to M, x \mapsto xa$ is a morphism of left $B$-modules. Consider the morphisms of left $B$-modules
 \[\alpha\colon M \to B^{\oplus r}, x \mapsto ([x,n_1], \cdots, [x, n_r]) \textrm{ and } \beta\colon B^{\oplus r} \to M, (b_1, \cdots, b_r) \mapsto \sum_j b_j m_j.\]
 It follows from $xa=\sum_j x (n_j, m_j) = \sum_j [x, n_j]m_j$ that $r_a= \beta \alpha$. 
 
Clearly, $\Tor_i^B(Y, B^{\oplus r}) = 0$, and hence $\Tor_i^B(Y, \alpha)=0$. It follows that 
\[\Tor_i^B(Y, r_a) = \Tor_i^B(Y, \beta)\circ \Tor_i^B(Y, \alpha)=0.\]
 Therefore, $\Tor_i^B(Y, M)a= \Tor_i^B(Y, r_a)(\Tor_i^B(Y, M)) = 0.$
 Hence, $\Tor_i^B(Y, M)I=0$, and  $\Tor_i^B(Y, M)$ is $\cF$-torsion.

  (3)
  For any $a\in\ker(\phi)$ and any $\sum_i (n_i, m_i) \in \im(\xi)$,
  $a(\sum_i(n_i, m_i))=\sum_i(a n_i, m_i)=0$.
  It follows that $\ker(\phi)$ is $\cF$-torsion.
  
  On the other hand, let $f\in\End_B(N)$ and $\sum_i(n_i, m_i) \in \im(\xi)$. Then, for any $n\in N$, 
  \begin{align*}
    &\phantom{{}={}}\Bigl(f\cdot\sum_i(n_i, m_i)\Bigr)(n)=f\Bigl(\sum_i(n_i, m_i) n \Bigr)
    =\sum_i f\bigl(n_i[m_i, n]\bigr)\\
    &=\sum_i f(n_i)[m_i, n]
    =\sum_i (f(n_i), m_i) n = an = \phi(a)(n),
  \end{align*}
  where $a=\sum_i (f(n_i), m_i) \in A$.
  Hence $f\cdot\sum_i(n_i, m_i)=\phi(a)\in\im(\phi)$.
  It follows that $\coker(\phi)$ is $\cF$-torsion.
\end{proof}

In classical Morita theory, if both $\xi$ and $\zeta$ are surjective, then $\Mod A$ and $\Mod B$ are equivalent.
The following theorem can be regarded as a generalization of this result.

\begin{theorem}
  [{{\cite[Thm.\@ 2.13]{liNoncommutativeResolutionsASGorenstein2026}}}]
  \label{prop:def-induce-equiv}
  Let $(A,B,M,N,\xi,\zeta)$ be a Morita context.
  Suppose that $\cF$ and $\cG$ are Gabriel topologies on $A$ and $B$, respectively,
  and that ${}_BM_A$ and ${}_AN_B$ are compatible with these Gabriel topologies.
  Then the Morita context $(A,B,M,N,\xi,\zeta)$ induces an equivalence of categories
  \[\blank\otimes_{B/\cG}\pi(M): \QModG B \to \QModF A \, : \blank\otimes_{A/\cF}\pi(N)\] 
  if and only if $\coker(\xi)$ is $\cF$-torsion
  and $\coker(\zeta)$ is $\cG$-torsion.  
\end{theorem}

Indeed,
\ref{prop:def-induce-equiv} can be reformulated as follows by means of \ref{lem:context-ker-torsion}.

\begin{theorem}
  \label{thm:equiv-by-Morita-context}
  Let $(A,B,M,N,\xi,\zeta)$ be a Morita context, and suppose that
  $\cF$ and $\cG$ are Gabriel topologies on $A$ and $B$, respectively.
Then
  ${}_BM_A$ and ${}_AN_B$  are compatible with the Gabriel topologies $\cF$ and $\cG$,
  and the Morita context $(A,B,M,N,\xi,\zeta)$ induces an equivalence
  between $\QModF A$ and $\QModG B$ if and only if the following conditions hold.
  \begin{enumerate}
    \item[\enumilabel{(1a)}{enum:equiv-context-tensor}]
      $\blank\otimes_B M$ maps $\cG$-torsion modules to
      $\cF$-torsion modules.
    \item[(1b)] $\blank\otimes_A N$ maps $\cF$-torsion modules to
      $\cG$-torsion modules.
    \item[\enumilabel{(2a)}{enum:equiv-context-trace}]
      $\coker(\xi)$ is $\cF$-torsion.
    \item[\enumilabel{(2b)}{enum:equiv-context-trace-B}]
      $\coker(\zeta)$ is $\cG$-torsion.
  \end{enumerate}
\end{theorem}

\begin{proof} The ``only if" part follows from the definition of compatibility and \ref{prop:def-induce-equiv}.

For the ``if" part,  by \ref{enum:tor-torsion}, both
  $\Tor_1^B(\blank, M)$ and $\Tor_1^A(\blank, N)$ send torsion modules to torsion modules. Hence, ${}_BM_A$ and ${}_AN_B$  are compatible with the Gabriel topologies $\cF$ and $\cG$. The conclusion follows from \ref{prop:def-induce-equiv}.
\end{proof}

In general, it is difficult to determine whether an equivalence between quotient categories is induced by a Morita context.
The following proposition provides information about this obstruction and
the additional properties of such an equivalence; moreover, it is used to show, in certain cases, that every equivalence between quotient categories of module categories is induced by a Morita context 
(see \ref{cor:quot-equiv-conseq-closed}).

\begin{proposition}\label{prop:tensor-equiv-implies-context-equiv}
  Let $(A,B,M,N,\xi,\zeta)$ be a Morita context. Suppose that $\cF$ and $\cG$ are
  Gabriel topologies on $A$ and $B$, respectively, and $\blank\otimes_B M$ induces an equivalence \[\blank\otimes_{\QModG B}\pi(M)\colon\QModG B\to\QModF A.\]
  Then, the following holds.
  \begin{enumerate}[ref=\ref{prop:tensor-equiv-implies-context-equiv}(\arabic*)]
    \item $\pi_{\cG} \Hom_{A/\cF}(\pi_{\cF}(M),\blank)$ is a quasi-inverse of $\blank\otimes_{\QModG B}\pi(M)$.
    \item\label{enum:tensor-equiv-implies-context-equiv-ii} If $\coker(\xi)$ is $\cF$-torsion,
    then $\pi_{\cG} \Hom_{A/\cF}(\pi_{\cF}(M),\blank)$ is induced by $\blank \otimes_A N$, and therefore 
    \[\pi_{\cG} \Hom_{A/\cF}(\pi_{\cF}(M),\blank) \cong \blank\otimes_{\QModF A}\pi(N).\]
    \item {\normalfont(\cite[Prop.\@ 2.2.6]{louInvariantTheoryRegular2016})}
    If $\coker(\xi)$ is $\cF$-torsion,
    then $\coker(\zeta)$ is $\cG$-torsion, and this equivalence is induced by $(A,B,M,N,\xi,\zeta)$.
    \item\label{cor:ample-implies-context-equiv} If the canonical map $N\to \Hom_A(M,A), n \mapsto \xi(n\otimes\blank)$ is an isomorphism, and $A$ is $\cF$-closed, then $\coker(\xi)$ is $\cF$-torsion.
  \end{enumerate}
\end{proposition}

\begin{proof}
  Let $G=\blank \otimes_B M$, $F=\Hom_A(M,\blank)$,
  $\overline G=\blank\otimes_{\QModG B}\pi(M)$ and let $\overline F$ be a quasi-inverse of $\overline G$.
  %$\blank\otimes_{\QModG B}\pi(M)$.
  % Consider the following diagram of adjunctions.
  % \begin{equation*}
  %   % \label[diagram]{diagr:context-adjunction}
  %   % https://tex.stackexchange.com/a/52379
  %   \begin{tikzcd}[column sep=huge,
  %     left adj/.style={rightharpoonup,shift left=2,},
  %     right adj/.style={leftharpoondown,shift right=2,},
  %     left adj rev/.style={leftharpoonup,shift left=2,},
  %     right adj rev/.style={rightharpoondown,shift right=2,},
  %     ]
  %     \Mod A & \Mod B\\
  %     \QModF A & \QModG B
  %     \arrow[from=1-1,left adj rev,]{1-2}[]{\blank\otimes_B M}
  %     \arrow[from=1-1,right adj rev,]{1-2}[swap,]{\Hom_A(M,\blank)}
  %     \arrow[from=1-1,no line]{1-2}[rotate=90,sloped,]{\vdash}
  %     \arrow[from=2-1,left adj rev,]{1-1}[]{\pi}
  %     \arrow[from=2-1,right adj rev,]{1-1}[swap,]{\omega}
  %     \arrow[from=2-1,no line,]{1-1}[rotate=90,sloped,]{\vdash}
  %     \arrow[from=2-1,left adj,]{2-2}[]{F}
  %     \arrow[from=2-1,right adj,]{2-2}[swap,]{G}
  %     \arrow[from=2-1,no line]{2-2}[sloped,]{\simeq}
  %     \arrow[from=1-2,left adj,]{2-2}[]{\pi}
  %     \arrow[from=1-2,right adj,]{2-2}[swap,]{\omega}
  %     \arrow[from=1-2,no line,]{2-2}[rotate=90,sloped,]{\vdash}
  %   \end{tikzcd}
  % \end{equation*}
  
  (1) The outer square in \ref{diagr:context-adjunction} is commutative by assumption. It follows from the uniqueness of adjoint functors that the inner square is also commutative, that is, $F\omega_{\cF} \cong \omega_{\cG}\overline F$.
  Then
  \[\overline F\cong\pi_{\cG}\omega_{\cG} \overline F \cong \pi_{\cG} \Hom_A(M,-)\,\omega_{\cF}\cong \pi_{\cG} \Hom_{\QModF A}(\pi_{\cF}(M), \blank).\]

  (2) Suppose that $\coker(\xi)$ is $\cF$-torsion. Then, by \ref{enum:ker-torsion}, $\ker(\xi)$ is also $\cF$-torsion. 
  Hence, for any $X\in \Mod A$, the morphism
  % $X\otimes_A N\otimes_BM
  % \xrightarrow{\,\id_X\otimes_A\xi\,}
  % X$
  $\id_X\otimes\xi \colon X\otimes_A N\otimes_BM
  \to X$ is an isomorphism in $\QModF A$. 
  Therefore,
  there are natural isomorphisms of functors:
  \begin{equation*}
    \begin{aligned}
      \overline F\pi_{\cF} &\cong \overline F\pi_{\cF}(\blank\otimes_A N\otimes_B M)\\
      &\cong \overline F\bigl((\pi_{\cG}(\blank\otimes_A N))\otimes_{\QModG B}\pi(M)\bigr)\\
      &= (\overline{F} \, \overline{G})(\pi_{\cG}(\blank\otimes_A N)) \cong \pi_{\cG}(\blank\otimes_A N).
    \end{aligned}
  \end{equation*}
  It follows that $\overline F$ is induced by $\blank\otimes_A N$. 

  (3) Let $\xi_X= \id_X\otimes \xi\colon X\otimes_A N\otimes_B M \to X$. Then
  \[
      \alpha= \{\alpha^{\pi(X)}=\pi_{\cF}(\xi_X) \colon \overline{G}\,\overline{F}(\pi_{\cF}(X)) \to \pi_{\cF}(X) \mid X\in \Mod A\}\colon \overline{G}\,\overline{F}\xrightarrow[]{\cong} \id_{\QModF A}
  \]
  is a natural isomorphism.
  Similarly, 
  the other Morita context map
  $\zeta\colon M\otimes_A N\to B$
  induces a natural transformation
  \[
  \beta=\{\beta^{\pi(Y)} = \pi_{\cG}(\zeta_Y) \colon \overline{F}\,\overline{G}(\pi_{\cG}(Y)) \to \pi_{\cG}(Y) \mid Y\in\Mod B\}\colon \overline{F}\, \overline{G}\to \id_{\QModG B},
  \]
  where $\zeta_Y$ is defined as $\id_Y\otimes \zeta\colon Y\otimes_B M\otimes_A N\to Y$.
  
  For any $Y\in\Mod B$, one has $\xi_{G(Y)}=G(\zeta_Y)$ by the associativity condition of the Morita context.
  Since $\alpha$ is a natural isomorphism,
  $\pi_{\cF}G(\zeta_Y)=\pi_{\cF}(\xi_{G(Y)})=\alpha^{\pi G(Y)}$
  is an isomorphism.
  % \[
  %   \begin{tikzcd}
  %     Y\otimes_B M\otimes_A N\otimes_B M & Y\otimes_B M\otimes_A A\\
  %     Y\otimes_B B\otimes_B M & Y\otimes_B M
  %     \arrow[from=1-1,]{1-2}[]{}
  %     \arrow[from=2-1,]{2-2}[]{}
  %     \arrow[from=1-1,]{2-1}[]{}
  %     \arrow[from=1-2,]{2-2}[]{}
  %   \end{tikzcd}
  % \]
  It follows from the natural isomorphism $\pi_{\cF} G\cong \overline{G}\pi_{\cG}$ that  $\overline{G}\beta^{\pi(Y)}=\overline{G}\pi_{\cG}(\zeta_Y)$ is an isomorphism; hence, $\beta^{\pi(Y)}$ is an isomorphism. %, and therefore, $\beta$ is a natural isomorphism.
  % Therefore, $\pi_{\cG}(\zeta_Y)$ is an isomorphism.
  In particular, taking $Y=B$, one obtains that $\pi_{\cG}(\zeta)$
  % \[\pi_{\cG}(\zeta)\colon\pi_{\cG}(M\otimes_AN) \to \pi_{\cG}(B)\]
  is an isomorphism.
  Therefore, $\coker(\zeta)$ is $\cG$-torsion, and the equivalence
  \[
      % \blank\otimes_{\QModG B}\pi(M)\colon \QModF A \simeq \QModG B\reflectbox{$\colon$} \blank\otimes_{\QModF A}\pi(N)
      \overline F\colon \QModF A \simeq \QModG B\colon \overline G
  \]
  is induced by $(A,B,M,N,\xi,\zeta)$.

  (4) For any $X\in \Mod A$,
  let $\ev^X\colon \Hom_A(M,X)\otimes_B M\to X, f\otimes m\mapsto f(m)$ be the evaluation map.
 As in the proof of (3)$\Rightarrow$(1) in \ref{lem:adjunction-diagram},
  consider the following diagram of natural isomorphisms, which gives an adjoint pair $(\overline G, \overline F)$.
  \begin{equation*}
    \begin{tikzcd}
      \Hom_{\QModG B}\bigl(\cY, \overline{F}(\cX)\bigr) & \Hom_{\QModF A}\bigl(\overline{G}(\cY), \cX\bigr)\\
      \Hom_B\bigl(\omega(\cY), \omega\overline{F}(\cX)\bigr) & \Hom_{\QModF A}\bigl(\pi G\omega(\cY), \cX\bigr)\\
      \Hom_B\bigl(\omega(\cY), F\omega(\cX)\bigr) & \Hom_A\bigl(G\omega(\cY), \omega(\cX)\bigr)
      \arrow[from=3-1,]{3-2}[]{\textrm{adj}}[swap,]{\cong}
      \arrow[from=1-1,dashed,]{1-2}[]{\cong}
      \arrow[from=2-1,]{3-1}[]{\text{\ref{lem:adjunction-diagram}}}[swap,]{\cong}
      \arrow[from=1-1,]{2-1}[]{\omega}[swap,]{\cong}
      \arrow[from=3-2,]{2-2}[]{\textrm{adj}}[swap,]{\cong}
      \arrow[from=2-2,]{1-2}[]{\text{\ref{prop:induced-functors}}}[swap,]{\cong}
    \end{tikzcd}
  \end{equation*}

  Let $\cY=\overline{F}(\cX)$, and let $\alpha\colon\omega\overline{F}\cong F\omega$ denote the natural isomorphism given in \ref{lem:adjunction-diagram}.
  Along the left column of the above diagram, the identity morphism of $\overline{F}(\cX)$ corresponds to $\alpha^{\cX}$.
  Under the tensor-Hom adjunction, $\alpha^{\cX}$ corresponds to $\ev^{\omega(\cX)}G(\alpha^{\cX})$, which, under the adjoint pair $(\pi, \omega)$, corresponds to $\epsilon^{\cX}\pi(\ev^{\omega(\cX)})\pi G(\alpha^{\cX})$.
  Composing the isomorphism $\overline{G}\,\overline{F}(\cX)\cong\pi G\omega\overline{F}(\cX)$ from \ref{prop:induced-functors} with the following morphisms yields a morphism $\overline G\,\overline F(\cX)\to\cX$:
  \[
    \overline{G}\,\overline{F}(\cX)\xrightarrow[\cong]{}\pi G\omega\overline F (\cX)
    \xrightarrow[\cong]{\pi G(\alpha^{\cX})}\pi GF\omega(\cX)\xrightarrow{\pi(\ev^{\omega(\cX)})}\pi\omega(\cX)\xrightarrow[\cong]{\epsilon^{\cX}}\cX.
  \]
  This is the counit of the adjoint pair $(\overline G, \overline F)$ evaluated at $\cX$, and it is an isomorphism because $\overline G$ and $\overline F$ are mutually quasi-inverse. Therefore, $\pi(\ev^{\omega(\cX)})$ is an isomorphism; hence, the kernel and cokernel of $\ev^{\omega(\cX)}$ are $\cF$-torsion.

  Since $A$ is assumed to be $\cF$-closed, the kernel and cokernel of $\ev^A$ are respectively isomorphic to those of $\ev^{\omega\pi(A)}$ and are therefore $\cF$-torsion.
  Denote by $\theta$ the isomorphism $N\to \Hom_A(M,A), n \mapsto \xi(n\otimes\blank)$. 
  Then the composition of the following morphisms
  \[
      N\otimes_B M\xrightarrow[]{\theta\otimes \id_M} \Hom_A(M,A)\otimes_B M \xrightarrow[]{\ev^A} A
  \]
  is precisely $\xi$. 
  Since $\theta$ is an isomorphism and $\coker(\ev^A)$ is $\cF$-torsion, it follows that $\coker(\xi)$ is $\cF$-torsion.
\end{proof}

\begin{remark}
  The condition that $\coker(\xi)$ is $\cF$-torsion is necessary
  in \ref{enum:tensor-equiv-implies-context-equiv-ii}.
  A simple counterexample is given by
  the Morita context $(A, A, A, 0, 0, 0)$
  with trivial topologies.
\end{remark}

\subsection{A Morita-type theorem for quotient categories}

% Now we are ready to give a Morita-type theorem for quotient categories.

We now turn to the converse construction: starting with an equivalence between quotient categories,
we recover bimodules that induce it.

\begin{theorem}\label{thm:morita-type}
Let $A$ and $B$ be rings equipped with Gabriel topologies
$\cF$ and $\cG$, respectively. Suppose that there is an
equivalence of categories
  \begin{equation*}
    % \overline F\colon \QModF A \simeq \QModG B\reflectbox{$\colon$} \overline G.
    \overline F\colon \QModF A \simeq \QModG B\colon \overline G.
  \end{equation*}
%where $G$ is a quasi-inverse of $F$.
Let $\cF'$ and $\cG'$ be the Gabriel topologies on
$A_{\cF}$ and $B_{\cG}$ induced by the canonical ring morphisms, respectively. 
% Transporting $F$ and $G$ to the
% categories of closed modules, set
% \[
% \widetilde F
%  \defeq
% \bar{a}_{\cG}F\bar{a}_{\cF}^{-1}
% \colon
% \Mod (A,\cF)
%  \xrightarrow{\sim}
% \Mod (B,\cG)
% \]
% and
% \[
% \widetilde G
%  \defeq
% \bar{a}_{\cF}G\bar{a}_{\cG}^{-1}.
% \]
% Define
% \[
% M
%  \defeq
% \bar{a}_{\cF}G\pi_{\cG}(B)
%  =\widetilde G(B_{\cG}),
% \qquad
% N
%  \defeq
% \bar{a}_{\cG}F\pi_{\cF}(A)
%  =\widetilde F(A_{\cF}).
% \]
Let $M \defeq \bar{a}_{\cF}\overline G\pi_{\cG}(B), \textrm{ and }
N \defeq \bar{a}_{\cG}\overline F\pi_{\cF}(A)$, where
$M$ is regarded naturally as an $A$-module and an $A_{\cF}$-module, and similarly $N$ is regarded naturally
as a $B$-module and a $B_{\cG}$-module. Then the following
statements hold.
\begin{enumerate}[ref=\ref{thm:morita-type}(\arabic*)]
  \item%\label{enum:quot-equiv-conseq-ring}
    There are natural ring isomorphisms
    \[
      \begin{aligned}
        A_{\cF} &\cong
        \End_{B/\cG} \bigl(\pi_{\cG}(N)\bigr)
        \cong \End_{B_{\cG}}(N) = \End_{B}(N); \\
        B_{\cG} &\cong
        \End_{A/\cF} \bigl(\pi_{\cF}(M)\bigr)
        \cong \End_{A_{\cF}}(M) = \End_{A}(M).
      \end{aligned}
    \]
    Via these isomorphisms, $M$ is a $B_{\cG}$-$A_{\cF}$-bimodule
    and $N$ is an $A_{\cF}$-$B_{\cG}$-bimodule.
  \item There are natural bimodule isomorphisms
    \[
      M \cong \Hom_{B_{\cG}}(N,B_{\cG}),
      \quad
      N \cong \Hom_{A_{\cF}}(M,A_{\cF}).
    \]
  \item\label{enum:quot-equiv-conseq-hom} There is an equivalence of categories
    \[ \Hom_A(M,\blank) \colon \Mod (A,\cF) \simeq \Mod (B,\cG) \colon \Hom_B(N,\blank).\]
  \item\label{enum:quot-equiv-conseq-functor} There are natural isomorphisms
    \begin{align*}
      \overline F &\cong \pi_{\cG} \Hom_{A/\cF}\bigl(\pi_{\cF}(M),\blank\bigr)\cong \blank\otimes_{A/\cF} \pi(N);\\
      \overline G &\cong \pi_{\cF}\Hom_{B/\cG}\bigl(\pi_{\cG}(N),\blank\bigr)\cong\blank\otimes_{B/\cG} \pi(M).
    \end{align*}
    In particular, $\overline F$ is induced by $\blank\otimes_A N$ and $\overline G$ is induced by $\blank\otimes_B M$.
  \item The object $\pi_{\cF}(M)$ is fully $\cG'$-ample in $\Mod A/\cF$,
    and $\pi_{\cG}(N)$ is fully $\cF'$-ample in $\Mod B/\cG$.
  \item The centres $Z(A_{\cF})$ and $Z(B_{\cG})$ are isomorphic.
\end{enumerate}
\end{theorem}

% 我应该在下一个推论的(4)的证明中提一下Morita context的结构映射是怎么构造的
\begin{proof}
By symmetry, it suffices to prove the assertions involving $M$; the
corresponding assertions involving $N$ follow in the same way.
The commutative diagram \ref{eq:localization-diagram} will be used freely here.
% \begin{equation*}
%   \begin{tikzcd}[
%     row sep=large,
%     left diagonal/.style={
%       start anchor={[sloped,allow upside down,xshift=-0.5em]north west},
%       end anchor={[sloped,allow upside down,xshift=-1em]south}
%     },
%     right diagonal/.style={
%       start anchor={[sloped,allow upside down,xshift=1em]south},
%       end anchor={[sloped,allow upside down,xshift=0.5em]north east}
%     }
%   ]
%     \Mod A & & \Mod A/\cF\\
%     & {\Mod(A,\cF)} &
%     \arrow[from=1-1,left adj]{1-3}[]{\pi}
%     \arrow[from=1-1,right adj]{1-3}[swap]{\omega=i\bar a}
%     \arrow[from=1-1,no line]{1-3}[rotate=90,sloped]{\vdash}
%     \arrow[from=2-2,left diagonal,left adj rev]{1-1}[swap,sloped]{a}
%     \arrow[from=2-2,left diagonal,right adj rev]{1-1}[sloped]{i}
%     \arrow[from=2-2,left diagonal,no line]{1-1}
%       [rotate=90,sloped,allow upside down]{\vdash}
%     \arrow[from=1-3,right diagonal,left adj rev]{2-2}[swap,sloped]{\pi i}
%     \arrow[from=1-3,right diagonal,right adj rev]{2-2}[sloped]{\bar a}
%     \arrow[from=1-3,right diagonal,no line]{2-2}[sloped]{\simeq}
%   \end{tikzcd}
% \end{equation*}
  \begin{equation*}
\begin{tikzcd}
\Mod A \arrow[rr, "\pi", shift left] \arrow[rd, "a"', shift right=2] &                                                   & \Mod A/\cF \arrow[ll, "\omega = i \bar{a}", shift left] \arrow[ld, "\bar{a}"' ] \\
 & {\Mod (A, \cF)} \arrow[lu, "i"'] \arrow[ru, "\pi i"', shift right=2] &
\end{tikzcd}
  \end{equation*}

% 在(2.1)前的段落里提到过一句 Mod-(A,\cF)->Mod-(A_{\cF}) 是 fully faithful 的
(1) Since $\bar{a}_{\cF}$ and $\overline G$ are fully
faithful, it follows that
\[\End_{A_{\cF}}(M)\cong\End_{A/\cF} \bigl(\overline G\pi_{\cG}(B)\bigr) 
\cong\End_{B/\cG} \bigl(\pi_{\cG}(B)\bigr)
 = B_{\cG} .\]
Moreover, since
$\pi_{\cF}(M) \cong\overline G\pi_{\cG}(B)$,
it follows that
$\End_{A/\cF}
 \bigl(\pi_{\cF}(M)\bigr)
 \cong B_{\cG}$.
The conclusion then follows from 
$\End_{A}(M) \cong \End_{A_{\cF}}(M)$ since $M \in \Mod (A,\cF)$
is an $\cF$-closed module.

(2) Since $\bar{a}_{\cF}$ is an equivalence, and $(\overline G, \overline F)$ is a pair of adjoint functors, we have
% the adjunction between $\overline F$ and $\overline G$ give natural isomorphisms
\[
\begin{aligned}
\Hom_{A_{\cF}}(M,A_{\cF}) &
\cong
\Hom_{A/\cF}
 \bigl(\overline G\pi_{\cG}(B),\pi_{\cF}(A)\bigr) \\
&\cong
\Hom_{B/\cG}
 \bigl(\pi_{\cG}(B),\overline F\pi_{\cF}(A)\bigr) \\
&\cong
\bar{a}_{\cG}\overline F\pi_{\cF}(A)
=N.
\end{aligned}
\]
These isomorphisms are compatible with the left
$A_{\cF}$-action and the right $B_{\cG}$-action, and
hence are bimodule isomorphisms.

(3) Let $X\in\Mod (A,\cF)$. Since both $M$ and $X$ are $\cF$-closed, again using the
full faithfulness of $\bar{a}_{\cF}$ and the adjunction between
$\overline F$ and $\overline G$, we obtain
\begingroup\allowdisplaybreaks
\begin{align*}
\Hom_A(M, X) &\cong
\Hom_{A/\cF}\bigl(\overline G\pi_{\cG}(B),\pi_{\cF}i_{\cF}(X)\bigr) \\
&\cong \Hom_{B/\cG}\bigl(\pi_{\cG}(B), \overline F\pi_{\cF}i_{\cF}(X)\bigr) \\
&\cong \bar{a}_{\cG} \overline F\pi_{\cF}i_{\cF}(X) \\
&=\widetilde F(X),
\end{align*}
\endgroup
where $\widetilde F \defeq \bar{a}_{\cG}\overline F\pi_{\cF}i_{\cF}$.
Thus,
$\widetilde F\cong \Hom_A(M,\blank) \colon \Mod (A, \cF) \to \Mod (B, \cG)$. 
Similarly, $\widetilde G\defeq \bar{a}_{\cF}\overline G \pi_{\cG}i_{\cG}\cong \Hom_B(N,\blank)\colon \Mod (B,\cG)\to \Mod (A,\cF)$.
Since $(\overline a,\pi i)$ is a pair of equivalence, it follows that $\widetilde F$ and $\widetilde G$ are quasi-inverse to each other, and thus so are $\Hom_A(M,\blank)$ and $\Hom_B(N,\blank)$.

(4)
It follows from $\bar{a}_{\cG}\overline F\pi_{\cF}i_{\cF}\cong \Hom_A(M,\blank)$
that
$\overline F \cong \pi_{\cG}\Hom_A\bigl(M,\bar{a}_{\cF}(\blank)\bigr)$, 
since $\pi i$ and $\bar{a}$ are quasi-inverse to one another. 
Furthermore,
since $\bar{a}_{\cF} \pi_{\cF}(M) \cong M$,
\[
\Hom_A\bigl(M,\bar{a}_{\cF}(\blank)\bigr)
\cong
\Hom_{A/\cF}\bigl(\pi_{\cF}(M),\blank\bigr)
\cong
\Hom_A\bigl(M,\omega_{\cF}(\blank)\bigr).
\]
By (3), $\Hom_A(M,\blank)$ sends $\cF$-closed modules to $\cG$-closed modules.
It follows from \ref{lem:adjunction-diagram} that
$\overline G$ is induced by $\blank\otimes_B M$, that is, $\overline G\cong \blank \otimes_{B/\cG} \pi (M)$.
Similarly, $\overline F\cong \blank \otimes_{A/\cF} \pi (N)$. 

(5) By \ref{prop:pi-A-ample},
$\pi_{\cG}(B)$ is fully $\cG'$-ample in
$\Mod B/\cG$. Since $\overline{G}$ is an equivalence
and
$\overline{G}\pi_{\cG}(B)\cong\pi_{\cF}(M)$,
the object $\pi_{\cF}(M)$ is fully $\cG'$-ample in
$\Mod A/\cF$.

(6) It follows from (1) and (2) that there is a right normalized Morita context $(A_{\cF},B_{\cG},M,N,\xi,\zeta)$ (see \cite[Def.\@ 2.3]{mullerQuotientCategoryMorita1974} for definition). 
Hence, \cite[Cor.\@ 9]{mullerQuotientCategoryMorita1974} shows that $Z(A_{\cF})\cong Z(B_{\cG})$.
\end{proof}

By \ref{cor:eta_A-induces-equi}, the equivalence $\QModF A\simeq \QModG B$ leads to the equivalence $\Mod A_{\cF}/\cF'\simeq \Mod B_{\cG}/\cG'$.
By applying \ref{thm:morita-type} to this equivalence, or more straightforwardly,
when $A$ and $B$ are closed with respect to their Gabriel topologies, \ref{thm:morita-type} admits the following simplified formulation; items that remain unchanged are omitted.

\begin{corollary}\label{cor:quot-equiv-conseq-closed}
  Let $A$ and $B$ be rings equipped with Gabriel topologies $\cF$ and $\cG$,
  respectively. Suppose that $A$ is $\cF$-closed and  $B$ is $\cG$-closed.
  Assume there is a pair of mutually quasi-inverse functors
  \begin{equation*}
    % \overline F\colon \QModF A \simeq \QModG B\reflectbox{$\colon$} \overline G.
    \overline F\colon \QModF A \simeq \QModG B\colon \overline G.
  \end{equation*}
  Let $M\defeq \omega_{\cF} \overline{G} \pi_{\cG}(B)$ and $N\defeq\omega_{\cG} \overline{F} \pi_{\cF}(A)$.
  Then the following hold.
  \begin{enumerate}[ref=\ref{cor:quot-equiv-conseq-closed}(\arabic*)]
    \item As rings, $B\cong\End_A(M)$ and $A\cong\End_B(N)$.
      Via these isomorphisms, $M$ is a $B$-$A$-bimodule
      and $N$ is an $A$-$B$-bimodule.
    \item\label{enum:quot-equiv-conseq-module} As bimodules, $N\cong\Hom_A(M, A)$
      and $M\cong\Hom_B(N, B)$.
    \item%\label{enum:quot-equiv-conseq-ample-closed}
      $\pi_{\cF}(M)$ is fully $\cG$-ample,
      and $\pi_{\cG}(N)$ is fully $\cF$-ample.
    \item $\overline F$ and $\overline G$ are induced by a Morita context
      $(A, B, M, N, \xi, \zeta)$.
    \item The centres $Z(A)$ and $Z(B)$ are isomorphic.
  \end{enumerate}
\end{corollary}

\begin{proof}
  It remains to prove (4) and the others follow directly from \ref{thm:morita-type}.

  (4) By (1), we identify $B$ with $\End_A(M)$.
  Choose a bimodule isomorphism $\theta\colon N\cong\Hom_A(M,A)$ as in (2).
  Define $\xi\colon N \otimes_B M \to A, n\otimes m \mapsto \theta(n)(m)$, and
  %for $m\in M$ and $n\in N$.
  define $\zeta\colon M \otimes_A N \to B, m\otimes n \mapsto \zeta(m\otimes n)\in B$ such that
  \[
    \zeta(m\otimes n)(m')=m\,\theta(n)(m'), \text{ for all }m'\in M.
  \]
  It can be verified that $(A,B,M,N,\xi,\zeta)$ forms a Morita context.
  %(\cite[pp.\@ 164--166]{jacBasicAlgebraII1989}).
  The induced map $N\to \Hom_A(M,A)$, given by $n\mapsto \xi(n\otimes\blank)$, is $\theta$ and is therefore an isomorphism.

  Since $A$ and $M$ are $\cF$-closed, it follows from \ref{enum:quot-equiv-conseq-hom} that
  \begin{align*}
    M
    &\cong \Hom_A(A, M)\\
    &\cong \Hom_B(\Hom_A(M,A), \Hom_A(M, M))\\
    &\cong \Hom_B(N, B),
  \end{align*}
  %One can verify that the composite 
  which induces precisely the isomorphism $M\to\Hom_B(N, B), m \mapsto \zeta(m\otimes\blank)$.
  % Therefore, the composite $m\in M\mapsto\zeta(m\otimes\blank)\in\Hom_B(N, B)$ is an isomorphism.

  It follows from \ref{enum:quot-equiv-conseq-functor} that $\overline F\cong \blank\otimes_{A/\cF}\pi (N)$ and $\overline G\cong \blank \otimes_{B/\cG}\pi (M)$.
  Then, \ref{cor:ample-implies-context-equiv} yields that $\coker (\xi)$ is $\cF$-torsion and $\coker (\zeta)$ is $\cG$-torsion.
  Therefore, $\overline F$ and $\overline G$ are induced by a Morita context
  $(A, B, M, N, \xi, \zeta)$ by \ref{prop:def-induce-equiv}.
\end{proof}

\ref{cor:quot-equiv-conseq-closed} motivates the following notion of a modulo-torsion-invertible bimodule,
which is first defined in the graded case in \cite[Def.\@ 2.17]{liNoncommutativeResolutionsASGorenstein2026}.

\begin{definition}\label{dfn:invert-bimod}
  Let $A$ and $B$ be rings equipped with Gabriel topologies $\cF$ and $\cG$, respectively.
  A bimodule ${}_BM_A$ is said to be $\cG$-$\cF$-\dfn{modulo-torsion-invertible}
  if there exists a Morita context $(A, B, M, N, \xi, \zeta)$
  such that the conditions in \ref{thm:equiv-by-Morita-context} hold.
\end{definition}

\begin{proposition}
  Let $A_1,A_2,A_3$ be rings equipped with Gabriel topologies
  $\mathcal F_1,\mathcal F_2,\mathcal F_3$, respectively.
  Suppose that $A_i$ is
  $\mathcal F_i$-closed, for $i=1, 2, 3$.
  \begin{enumerate}
    \item The map $_{A_1}(N_1)_{A_2}\mapsto \blank\otimes_{\QModF[1]{A_1}}\pi(N_1)$
      defines a bijection between the class of isomorphism classes of $\cF_1$-$\cF_2$-modulo-torsion-invertible bimodules 
      % with $\depth^{\cF_2}M_1\geqslant 2$
      that are $\cF_2$-closed
      and the class of natural isomorphism classes of functors giving equivalences between $\QModF[1]{A_1}$ and $\QModF[2]{A_2}$.
      Its inverse is given by the map $\overline{F_1}\mapsto \omega_{\cF_2}\overline{F_1}\pi_{\cF_1}(A_1)$.
    \item 
      Let
      \[
      \QModF[1]{A_1}\xrightarrow{\overline{F_1}}
      \QModF[2]{A_2}\xrightarrow{\overline{F_2}}
      \QModF[3]{A_3}
      \]
      be equivalences corresponding to $_{A_1}(N_1)_{A_2}$ and $_{A_2}(N_2)_{A_3}$ under the bijection given in (1).
      Then $\overline{F_2}\,\overline{F_1}$ corresponds to $\omega_{\cF_3}\pi_{\cF_3}(N_1\otimes_{A_2}N_2)$.
  \end{enumerate}
\end{proposition}

\begin{proof}
  (1) It has already been proved in \ref{cor:quot-equiv-conseq-closed}.

  (2) Since $\overline{F_i}$ is induced by $\blank\otimes_{A_i}N_i$,
  for every $X_1\in\Mod{A_1}$ there are natural isomorphisms
  \begin{align*}
    \overline{F_2}\,\overline{F_1}\pi_{\cF_1}(X_1)
    &\cong\overline{F_2}\pi_{\cF_2}(X_1\otimes_{A_1}N_1)\\
    &\cong\pi_{\cF_3}\bigl((X_1\otimes_{A_1}N_1)\otimes_{A_2}N_2\bigr)\\
    &\cong\pi_{\cF_3}\bigl(X_1\otimes_{A_1}(N_1\otimes_{A_2}N_2)\bigr).
  \end{align*}
  Taking $X_1=A_1$ and applying $\omega_{\cF_3}$ yields the asserted bimodule by (1).
  % Naturality with respect to left multiplication on $A_1$ preserves the left action.
\end{proof}

Suppose that $A$ is $\cF$-closed.
Let $\Pic_{\cF}(A)$ denote the set of isomorphism classes of $\cF$-$\cF$-modulo-torsion-invertible
$A$-$A$-bimodules that are $\cF$-closed as right $A$-modules.
The preceding proposition equips this set with a group structure, with multiplication
\[[M_1]\cdot[M_2]=[\omega\pi(M_1\otimes_A M_2)],\]
identity element $[A]$, and inverse, by \ref{enum:quot-equiv-conseq-functor} and \ref{enum:quot-equiv-conseq-module},
\[[M]^{-1}=[\Hom_A(M, A)].\]
When $\cF$ is the trivial topology, this definition recovers the usual Picard group of invertible $A$-$A$-bimodules.

  % !TeX root = ../main.tex

\section{Application to the noncommutative Auslander theorem}
\label{sec:application}

Let $k$ be a field, $S$ a $k$-algebra, and $H$ a finite-dimensional Hopf algebra over $k$.
Suppose that $S$ is a left $H$-module algebra.
% The Auslander map is
% \[S\#H\to\End_{S^H}(S),\quad s\#h\mapsto (x\mapsto s(h\rightharpoonup x)).\]
% The noncommutative Auslander theorem is said to hold for the action
% if the Auslander map is bijective.

Choose $0\neq t\in\int^l_H$.
There is a Morita context
\begin{equation}\label{eq:Auslander-context}
  \bigl(A\defeq S^H, B\defeq S\#H, M\defeq{}_{S\#H}S_{S^H}, N\defeq{}_{S^H}S_{S\#H}, \xi, \zeta\bigr),
\end{equation}
where $S$ is an $S^H$-bimodule via left and right multiplication, respectively,
and a left $S\#H$-module via the standard action
\[
  (s\#h)\cdot s' = s(h\rightharpoonup s').
\]
% and the right $S\#H$-module structure of $S$ is `twisted'.
For further details, we refer to \cite[Sect.\@ 4.5]{monHopfAlgebrasTheir1993}.
In this context, the canonical map $B\to\End_A(M)$ induced by the left $B$-action on $M$ is the Auslander map \ref{eq:auslander-map}.
Recall that the noncommutative Auslander theorem is said to hold for this action if the Auslander map is bijective.

We apply our results to the canonical map $B\to\End_A(M)$
associated with a general Morita context.
We first show that the Morita context induces an equivalence
between the quotient categories determined by the Gabriel topologies
generated by the trace ideals.
We then establish criteria for the injectivity and bijectivity of the canonical map.
Finally, we specialize these results to Hopf actions on AS-regular algebras.

% 这部分是Morita I的推论, 证明迹理想决定的商范畴是等价的
Let $A$ be a ring and let $T$ be a two-sided ideal of $A$.
Let $\cF(T)$ denote the weakest Gabriel topology on $A$ containing $T$.
Explicitly, $\cF(T)$ consists of the right ideals $I$ such that, for any $X\in\Mod A$, $\Ann_X(T)=0$ implies $\Ann_X(I)=0$, 
and an $A$-module $X$ is $\cF(T)$-torsion-free if and only if $\Ann_X(T)=0$
(see \cite[Prop.\@ 1]{mullerQuotientCategoryMorita1974} and \cite[Chap.\@ IX, Exer.\@ 14]{stenstromRingsQuotients1975}).
If $T$ is finitely generated as a right ideal,
then $\{T^n\mid n\in\NN\}$ is a basis of $\cF(T)$ (\cite[Chap.\@ VI, Prop.\@ 6.10]{stenstromRingsQuotients1975}).
If $T$ is idempotent, then $\cF(T)$ consists of all right ideals containing $T$
(\cite[Chap.\@ VI, Prop.\@ 6.11]{stenstromRingsQuotients1975}).

M\"uller proved that the Hom functors associated with a Morita context
induce mutually quasi-inverse equivalences between the categories of closed modules
determined by its trace ideals (\cite[Thm.\@ 3]{mullerQuotientCategoryMorita1974}).
The following proposition identifies this equivalence
with the one induced by the original Morita context.
% Let $(A,B,M,N,\xi,\zeta)$ be a Morita context.
% Let $S\leq A$ and $T\leq B$ denote the trace ideals,
% and $\cF_S$ and $\cG_T$ denote the corresponding Gabriel topologies on $A$ and $B$ respectively.
% M\"uller proves that $\Hom_A(M,\blank)$ and $\Hom_B(N,\blank)$ are mutually quasi-inverse functors
% between $\Mod (A,\cF_S)$ and $\Mod (B, \cG_T)$ (\cite[Thm.\@ 3]{mullerQuotientCategoryMorita1974}).
% We are going to show further that the equivalence is induced by the Morita context.

\begin{proposition}\label{prop:trace-ideal-equiv}
  Let $(A,B,M,N,\xi,\zeta)$ be a Morita context,
  and set $T_A=\im(\xi)$ and $T_B=\im(\zeta)$.
  Then this Morita context induces an equivalence 
  \[\overline{F} \colon \QModFT{A}\simeq\QModFT{B} \, \colon \overline{G}.\]
  Moreover, the following diagram commutes.
  \[
    \begin{tikzcd}[column sep=huge]
      \QModFT{A} & \QModFT{B} & \QModFT{A}\\
      \Mod(A,\cF(T_A)) & \Mod(B, \cF(T_B)) & \Mod(A, \cF(T_A))
      \arrow[from=1-1,]{1-2}[]{\overline{F}}
      \arrow[from=1-2,]{1-3}[]{\overline{G}}
      \arrow[from=2-1,]{2-2}[]{\Hom_A(M, \blank)}
      \arrow[from=2-2,]{2-3}[]{\Hom_B(N, \blank)}
      \arrow[from=1-1,]{2-1}[]{\overline{a}}
      \arrow[from=1-2,]{2-2}[]{\overline{a}}
      \arrow[from=1-3,]{2-3}[]{\overline{a}}
    \end{tikzcd}
  \]
\end{proposition}

\begin{proof} By the associativity conditions of a Morita context, $T_A N=N T_B$. Suppose that $Y_B$ is $\cF(T_B)$-torsion-free and that $f\in\Hom_B(N, Y)$ satisfies $f \, T_A=0$.
Then, for any $n\in N$, $f(n)T_B\subseteq f N T_B=f \, T_A N=0$. Hence $f=0$. It follows that $\Hom_B(N, Y)$ is $\cF(T_A)$-torsion-free (see also the proof of \cite[Thm.\@ 3]{mullerQuotientCategoryMorita1974}).

  Suppose that $X_A$ is $\cF(T_A)$-torsion.
  For any $\cF(T_B)$-torsion-free $B$-module $Y$,
  \[\Hom_B(X\otimes_A N, Y)\cong\Hom_A(X, \Hom_B(N, Y))=0,\]
  as $\Hom_B(N, Y)$ is $\cF(T_A)$-torsion-free.
  Hence, $X\otimes_A N$ is $\cF(T_B)$-torsion. Therefore, $\blank\otimes_A N$ sends $\cF(T_A)$-torsion modules to $\cF(T_B)$-torsion modules.
  
  By symmetry, $\blank\otimes_B M$ sends $\cF(T_B)$-torsion modules to $\cF(T_A)$-torsion modules.

  Moreover, it is clear that $\coker(\xi)=A\div T_A$ and $\coker(\zeta)=B\div T_B$
  are $\cF(T_A)$-torsion and $\cF(T_B)$-torsion, respectively.
  Hence, by \ref{thm:equiv-by-Morita-context}, the Morita context induces a mutually quasi-inverse equivalence
  \[
    \overline F=\blank\otimes_{\QModFT{A}}\pi(N)
    \colon\QModFT{A}\simeq\QModFT{B} \colon
    \blank\otimes_{\QModFT{B}}\pi(M)=\overline G.
  \]

  By \ref{enum:tensor-equiv-implies-context-equiv-ii} and adjunction,
  there are natural isomorphisms
  \[\overline{F}\cong\pi_{\cF(T_B)}\Hom_{\QModFT{A}}(\pi(M),\blank)
  \cong\pi_{\cF(T_B)}\Hom_{A}(M,\overline{a}(\blank)).\]
  By \ref{lem:adjunction-diagram}, $\Hom_A(M, \blank)$ sends $\cF(T_A)$-closed modules
  to $\cF(T_B)$-closed modules.
  It follows that $\overline{a}\overline{F}\cong\Hom_A(M,\overline{a}(\blank))$.
  The other natural isomorphism follows by symmetry.
\end{proof}

% 这部分是Morita II的推论, 给出Auslander映射是双射的充要条件
The following proposition provides criteria for the canonical map $B\to\End_A(M)$
to be injective or bijective,
assuming that the Morita context induces an equivalence between quotient categories.
\begin{proposition}\label{cor:auslander-context}
  Let $(A,B,M,N,\xi,\zeta)$ be a Morita context,
  and let $\cF$ and $\cG$ be Gabriel topologies on $A$ and $B$, respectively.
  Denote by $\psi$ the canonical map $B\to\End_A(M)$.
  Assume that the Morita context induces an equivalence between
  $\QModF A$ and $\QModG B$,
  and that $M$ is $\cF$-torsion-free. Then the following statements hold.
  \begin{enumerate}
    \item $\psi$ is injective if and only if $B$ is $\cG$-torsion-free.
    \item $\psi$ is bijective if $B$ is $\cG$-closed.
  \end{enumerate}
  Moreover, if $M$ is $\cF$-closed, then
  \begin{enumerate}[resume*]
    \item $\psi$ is bijective if and only if $B$ is $\cG$-closed.
  \end{enumerate}
\end{proposition}

\begin{proof}
  Let $G=\blank\otimes_B M$, and let $\overline{G}$ denote the  equivalence $\QModG B\simeq\QModF A$ induced by $G$.
  Without loss of generality, assume that $\pi_{\cF} G=\overline{G}\pi_{\cG}$.

  Recall that $\eta_{\cG}^B$ is injective if and only if $B$ is $\cG$-torsion-free (\ref{lem:kernel-eta}),
  and is bijective if and only if $B$ is $\cG$-closed.

  Since $\overline{G}(\pi_{\cG}(B))\cong\pi_{\cF}(M)$, the functor $\overline{G}$ induces a natural ring isomorphism
  \[\theta\colon B_{\cG}=\End_{\QModG B}(\pi_{\cG}(B))\cong \End_{\QModF A}(\pi_{\cF}(M)).\]
   Denote left multiplication by $b \in B$ by $l_b\colon B\to B$.
  Under the identification $B_{\cG}=\End_{\QModG B}(\pi_{\cG}(B))$, the element $\eta_{\cG}^B(b)$ corresponds to $\pi_{\cG}(l_b)$.
  It follows that $\theta\eta_{\cG}^B(b)=\overline{G}\pi_{\cG}(l_b)=\pi_{\cF}G(l_b)=\pi_{\cF}(\psi(b))$,
  and hence
  \begin{equation}\label{eq:cor-auslander-context}
    \eta_{\cG}^B=\theta^{-1}\pi_{\cF}\psi.
  \end{equation}

  Note that the following diagram commutes by the adjunction.
  \[
    \begin{tikzcd}
      \Hom_A(M, M) & \Hom_A(M, M_{\cF})\\
      \End_A(M) & \End_{\QModF A}(\pi_{\cF}(M))
      \arrow[from=1-1,]{1-2}[]{(\eta_{\cF}^M)_*}
      \arrow[from=2-1,]{2-2}[]{\pi_{\cF}}
      \arrow[from=1-1,]{2-1}[]{}[swap,]{=}
      \arrow[from=1-2,]{2-2}[]{\cong}[swap,]{\textup{adj}}
    \end{tikzcd}
  \]
  Recall that $\eta_{\cF}^M$ is injective if and only if $M$ is $\cF$-torsion-free,
  and is bijective if and only if $M$ is $\cF$-closed. 
  Therefore, $\pi_{\cF}\colon\End_A(M)\to\End_{\QModF A}(\pi_{\cF}(M))$ is injective if $M$ is $\cF$-torsion-free,
  and it is bijective if $M$ is $\cF$-closed.
  Assertions (1)--(3) follow from \ref{eq:cor-auslander-context}.
\end{proof}

% 应用到迹理想拓扑下我们能给出对于任意Morita context, Auslander映射是双射的条件
Taking $\cF=\cF(T_A)$ and $\cG=\cF(T_B)$ in the preceding proposition yields the following corollary.
\begin{corollary}\label{cor:auslander-trace-context}
  Let $(A,B,M,N,\xi,\zeta)$ be a Morita context,
  and set $T_A=\im(\xi)$ and $T_B=\im(\zeta)$.
  Denote by $\psi$ the canonical map $B\to\End_A(M)$.
  Suppose that $M$ is $\cF(T_A)$-torsion-free. Then the following statements hold.
  \begin{enumerate}
    \item $\psi$ is injective if and only if $B$ is $\cF(T_B)$-torsion-free.
    \item $\psi$ is bijective if $B$ is $\cF(T_B)$-closed.
  \end{enumerate}
  Moreover, if $M$ is $\cF(T_A)$-closed, then
  \begin{enumerate}[resume*]
    \item $\psi$ is bijective if and only if $B$ is $\cF(T_B)$-closed.
  \end{enumerate}
\end{corollary}

\begin{proof}
  The result follows by combining \ref{prop:trace-ideal-equiv,cor:auslander-context}.
\end{proof}

When $T_A=A$ and $T_B$ is idempotent,
the preceding criteria may be formulated in terms of grade.
This recovers \cite[Prop.\@ 2.9]{buchweitzMoritaContextsIdempotents2003}
and yields an alternative proof.
Recall that, if $A$ is a ring and $M$ is an $A$-module,
the grade of $M$ is defined by
\[
  j_A(M)\defeq\inf\{i\in\NN\mid\Ext_A^i(M, A)\neq 0\}.
\]

\begin{corollary}\label{cor:auslander-context-projective}
  Let $(A,B,M,N,\xi,\zeta)$ be a Morita context.
  Suppose that $T_A\defeq\im(\xi)=A$ and that $T_B\defeq\im(\zeta)$ is idempotent.
  Then the canonical ring morphism $B\to\End_A(M)$ is injective if and only if $j_B(B\div T_B)\geq 1$,
  and it is bijective if and only if $j_B(B\div T_B)\geq 2$.

  In particular, suppose that $B$ is a ring and that $e\in B$ is an idempotent.
  Then the canonical ring morphism $B\to\End_{eBe}(Be)$ is injective if and only if $j_B(B\div (e))\geq 1$,
  and it is bijective if and only if $j_B(B\div (e))\geq 2$.
\end{corollary}

\begin{proof}
  Since $T_A=A$, $\cF(T_A)=\{A\}$ is the trivial topology;
  hence, $M$ is automatically $\cF(T_A)$-closed.
  Since $T_B$ is idempotent, $\uR^i\tau_{\cF(T_B)}(B)=\Ext^i_B(B\div T_B, B)$.
  The conclusion then follows from \ref{lem:tf-4-term-exact-seq,cor:auslander-trace-context}.
\end{proof}

% 思路1 我们可以考虑非半单的Hopf作用
% 于是我们需要介绍Hopf作用的Morita context
We now return to Hopf actions on algebras and apply these criteria to the Morita context \ref{eq:Auslander-context}.
By \cite[Sect.\@ 4.5]{monHopfAlgebrasTheir1993},
the trace ideals of this context are
$T_A=\im(\hat{t})$ and $T_B=(1\#t)$,
where $\hat{t}\colon S\to S^H, s\mapsto t\cdot s$
is an $S^H$-bimodule map.

If $H$ is semisimple, one may choose $t$ such that $\epsilon(t)=1$.
Then $\hat{t}$ is surjective, and $T_B=(1\#t)$ is idempotent
(\cite[p.\@ 48]{monHopfAlgebrasTheir1993}).
Consequently, $T_A=A$, and \ref{cor:auslander-context-projective} applies.
For a general finite-dimensional Hopf algebra, however,
these properties need not hold.

Recall that a ring $S$ is GK-Cohen--Macaulay if $\GKdim(S)<\infty$ and 
\[j_S(M)+\GKdim(M)=\GKdim(S)\]
for every finitely generated left or right $S$-module $M$
(\cite[p.\@ 989]{staffordHomologicalPropertiesGraded1994}).
Moreover, GK dimension is exact on modules over a GK-Cohen--Macaulay ring.

The following corollary applies \ref{cor:auslander-context} to provide a sufficient condition, formulated
in terms of GK dimension, for the Auslander map to be bijective.
% However, applying \ref{cor:auslander-context} still shows that the GK dimension condition
% ensures that the Auslander map is bijective.

\begin{corollary}\label{cor:nc-auslander-thm}
  Let $S$ be a noetherian $k$-algebra, and let $H$ be a finite-dimensional Hopf algebra over $k$ acting on $S$
  such that $\hat{t}\neq 0$.
  Suppose that $S$ is GK-Cohen--Macaulay of GK dimension $d\geq 2$,
  and has no zero-divisors.
  Then the Auslander map $S\#H\to\End_{S^H}(S)$ is an isomorphism
  if $\GKdim(S\#H\div (1\#t))\leq d-2$.
\end{corollary}

\begin{proof}
  Consider the Morita context \ref{eq:Auslander-context}.
  Since $\hat{t}\neq 0$, the trace ideal $T_A=\im(\hat{t})$ is nonzero.
  Since $S$ has no zero-divisors,
  it follows that $\Ann_{M}(T_A)=0$; that is, $M$ is $\cF(T_A)$-torsion-free.

  It is standard that $B=S\# H$ is noetherian and a Frobenius extension of $S$.
  Therefore, for any finitely generated $B$-module $Y$, one has $\GKdim_B(Y)=\GKdim_S(Y)$ and $j_B(Y)=j_S(Y)$ (see \cite[Lem.\@ 1.26]{zhuAuslanderTheoremPI2023}).
  It follows that $B$ is GK-Cohen--Macaulay of GK dimension $d$.
  Since $B$ is right noetherian, $\{T_B^n\mid n\in\NN\}$ is a basis of $\cF(T_B)$.
  Consequently, by \ref{rem:depth-basis},
  $B$ is $\cF(T_B)$-closed if and only if $\Hom_{B}(B\div T_B^n, B)=0=\Ext_B^1(B\div T_B^n, B)$ for any $n\in\NN$;  equivalently, 
  if and only if $j_B(B\div T_B^n)\geq 2$ for any $n\in\NN$;  equivalently, 
  if and only if $\GKdim(B\div T_B)\leq d-2$.
  Hence, $B$ is $\cF(T_B)$-closed, and the result follows from \ref{cor:auslander-trace-context}.
\end{proof}

\ref{main-thm-auslander} follows immediately,
since $\hat{t}\neq 0$ for every nonzero left integral $t$,
provided that $S$ is a faithful left $H$-module.

% Example 1 in postive char
The Auslander map for permutation group actions was studied in \cite{gaddisAuslandersTheoremPermutation2019}
in characteristic zero. The following provides an example in the modular case.

\begin{example}\label{ex:auslander-1}
  Suppose that $\charOp(k)=p>2$.
  Let $S$ be the standard graded $(-1)$-skew polynomial algebra of dimension $n$.
  Then $S$ is a GK-Cohen--Macaulay Artin--Schelter regular domain of dimension $n$.
  Let $G$ be a subgroup of the symmetric group on $n$ letters acting on $S$ by permuting the indeterminates,
  and suppose that $p\mid|G|$.
  Note that the computation in \cite[Lem.\@ 2.1 and Prop.\@ 2.3]{gaddisAuslandersTheoremPermutation2019} remains valid whenever $\charOp(k)\neq 2$.
  % Suppose $H=kC_3=k\langle\sigma\rangle$ acts on $S$ by $\sigma(x)=y, \sigma(y)=z, \sigma(z)=x$.
  % By computation, \[(x^2+y^2+z^2)^2, ((y^2-x^2)(z^2-y^2)(x^2-z^2))^2\] is a central regular sequence
  % in $(1\# t)\subseteq S\# H$.
  Therefore, $\GKdim(S\#kG\div (1\#t))\leq n-2$,
  and by the preceding corollary, the Auslander map is bijective. 
\end{example}

% Example 2 in char 0
In the following example, we consider a Taft algebra action on a skew polynomial algebra in characteristic zero.
This serves as a special case of a broader class of actions first studied 
in \cite{clineActionsQuantumLinear2020}.
\begin{example}\label{ex:auslander-2}
  Suppose that $k$ has characteristic $0$ and that $\lambda \in k$ is a primitive third root of unity.
  Let $H$ be the Taft algebra generated by $g$ and $p$, subject to the relations
  \[g^3=1, \quad p^3=0, \quad gp=\lambda pg,\]
  where $g$ is a group-like element and $p$ is a skew-primitive element satisfying $\Delta(p)=p\otimes1+g\otimes p$.
  Let $S$ be the skew polynomial algebra
  \[S=k\langle x,y,z\rangle\div(yx-\lambda xy, zx-\lambda^2xz, zy-\lambda^2yz).\]
  Then $S$ is a three-dimensional Artin--Schelter regular domain that is GK-Cohen--Macaulay.
 By \cite[Exam.\@ 3.8]{clineActionsQuantumLinear2020}, $H$ acts homogeneously on $S$ via
  \[g\rightharpoonup (x,y,z)=(\lambda^2x, \lambda y, z),\quad p\rightharpoonup (x,y,z)=(0,x,y).\]
  For this action, the Auslander map $S\#H\to\End_{S^H}(S)$ is bijective.
\end{example}

\begin{proof}
  Let $t=3^{-1}(1+g+g^2)p^2\in\int^l_H$.
  A direct computation yields $\hat{t}(x^2z)=\lambda^2x^3\neq 0$; hence, $\hat{t}\neq 0$.
  By \ref{cor:nc-auslander-thm}, it suffices to show that $\GKdim(S\#H\div (1\#t))\leq 1$.

  Clearly, $k[x^3, y^3, z^3]\subseteq Z(S)\cap S^H$,
  and is therefore a central subalgebra of $S\#H$.
  Since $\{x^6, y^{18}\}$ forms a central regular sequence in $S$, and $S\#H$ is finite free over $S$,
  it follows that $\{x^6\# 1, y^{18}\# 1\}$ forms a central regular sequence in $S\#H$.

  It remains to prove that $x^6\# 1, y^{18}\# 1\in T_B=(1\# t)$.
  Throughout the following computations, $s\in S$ and $h\in H$ are identified with $s\# 1$ and $1\# h$ in $S\# H$, respectively.
  All products are taken in $S\# H$;
  for example, $hs$ denotes $(1\#h)(s\# 1)$ and is distinct from the action $h\rightharpoonup s$.
  
  Let $e_j=3^{-1}\sum_{i=0}^2\lambda^{-ij}g^i$
  for $j\in\ZZ\div 3\ZZ$.
  Note that $t=e_0p^2$, $e_0+e_1+e_2=1$, and $e_0x^j=x^je_j$.
  By the construction of the action, $px=\lambda^2xp$ and $py=x+\lambda yp$.
  It follows that
  \[T_B\ni x^{2-i}tx^i=x^{2-i}e_0p^2x^i=\lambda^ix^{2-i}e_0x^ip^2=\lambda^ix^2e_ip^2, \text{ for }i=0, 1, 2,\]
  and therefore $x^2p^2=\sum_{i=0}^2\lambda^{-i}x^{2-i}tx^i\in T_B$.
  Note that $p^2y=\lambda^2yp^2-xp$.
  Thus, 
  \[T_B\ni x^2p^2y-yx^2p^2=-x^3p \quad\text{and}\quad T_B\ni x^3py-\lambda yx^3p=x^4.\]
  Hence, $x^4\in T_B$ and, consequently, $x^6\in T_B$.

  In what follows, we work in $\overline{B}=B\div (x)$.
  \newcommand{\op}{\overline{p}}\newcommand{\oy}{\overline{y}}\newcommand{\oz}{\overline{z}}%
  \newcommand{\ot}{\overline{t}}\newcommand{\ooe}{\overline{e}}%
  We have $\op\oy=\lambda \oy\op$, $\op\oz=\oy+\oz\op$,
  and $\ooe_0\oy^i=\oy^i\ooe_{-i}$.
  Therefore, as in the preceding paragraph, \[\overline{T_B}\ni \oy^{2-i}\ot\oy^i=\oy^{2-i}\ooe_0\op^2\oy^i=\lambda^{2i}\oy^2\ooe_{-i}\op^2,\text{ for }i=0,1,2,\]
  and hence $\oy^2\op^2\in\overline{T_B}$.
  Note that $\op^2\oz=\oz\op^2-\lambda^2\oy\op$.
  It follows that
  \[\overline{T_B}\ni \oy^2\op^2\oz-\lambda^2\oz\oy^2\op^2=-\lambda^2\oy^3\op
  \quad\text{and}\quad
  \overline{T_B}\ni \oy^3\op\oz-\oz\oy^3\op=\oy^4,\]
  and hence $y^4\in T_B+(x)$.
  Note that $x\in B$ is normal; hence $(x)^4\subseteq (x^4)\subseteq T_B$.
  It follows that $y^{16}\in T_B$, and consequently $y^{18}\in T_B$.
  Therefore, $\GKdim(S\#H\div (1\#t))\leq 1$, from which it follows that the Auslander map is bijective.
\end{proof}

  % !TeX root = ../main.tex

\subsection*{AI Statement}
The authors used AI tools to refine the language of this manuscript.
AI also assisted with the computations in \ref{ex:auslander-2},
which were verified independently by the authors.

% \subsection*{Acknowledgment}

  \appendix

  \myprintbibliography

@article{artinNoncommutativeProjectiveSchemes1994,
  title = {Noncommutative {{Projective Schemes}}},
  author = {Artin, Michael and Zhang, James Jian},
  date = {1994-12-01},
  journaltitle = {Advances in Mathematics},
  shortjournal = {Adv. Math.},
  volume = {109},
  number = {2},
  pages = {228--287},
  doi = {10.1006/aima.1994.1087},
  url = {https://www.sciencedirect.com/science/article/pii/S0001870884710875},
  urldate = {2023-03-17},
  langid = {english},
  mrnumber = {1304753}
}

@article{auslanderRepresentationTheoryArtin1974,
  title = {Representation {{Theory}} of {{Artin Algebras I}}},
  author = {Auslander, Maurice},
  date = {1974-01-01},
  journaltitle = {Communications in Algebra},
  shortjournal = {Comm. Algebra},
  volume = {1},
  number = {3},
  pages = {177--268},
  publisher = {Taylor \& Francis},
  doi = {10.1080/00927877408548230},
  url = {https://doi.org/10.1080/00927877408548230},
  urldate = {2024-12-17},
  langid = {english},
  mrnumber = {0349747}
}

@article{baoNoncommutativeAuslanderTheorem2018,
  title = {Noncommutative {{Auslander}} Theorem},
  author = {Bao, Yan-Hong and He, Ji-Wei and Zhang, James Jian},
  date = {2018-12},
  journaltitle = {Transactions of the American Mathematical Society},
  shortjournal = {Trans. Amer. Math. Soc.},
  volume = {370},
  number = {12},
  pages = {8613--8638},
  doi = {10.1090/tran/7332},
  url = {https://www.ams.org/tran/2018-370-12/S0002-9947-2018-07332-9/},
  urldate = {2023-04-04},
  langid = {english},
  mrnumber = {3864389}
}

@article{baoPertinencyHopfActions2019,
  title = {Pertinency of {{Hopf}} Actions and Quotient Categories of {{Cohen-Macaulay}} Algebras},
  author = {Bao, Yan-Hong and He, Ji-Wei and Zhang, James Jian},
  date = {2019},
  journaltitle = {Journal of Noncommutative Geometry},
  shortjournal = {J. Noncommut. Geom.},
  volume = {13},
  number = {2},
  pages = {667--710},
  doi = {10.4171/JNCG/336},
  url = {https://mathscinet.ams.org/mathscinet-getitem?mr=3988759},
  urldate = {2023-04-28},
  mrnumber = {3988759}
}

@inproceedings{buchweitzMoritaContextsIdempotents2003,
  title = {Morita Contexts, Idempotents, and {{Hochschild}} Cohomology---with Applications to Invariant Rings},
  booktitle = {Commutative Algebra},
  author = {Buchweitz, Ragnar-Olaf},
  date = {2003},
  series = {Contemporary {{Mathematics}}},
  number = {331},
  pages = {25--53},
  publisher = {American Mathematical Society},
  location = {Providence, Rhode Island},
  doi = {10.1090/conm/331/05901},
  eventdate = {2001-07},
  eventtitle = {International {{Conference}}},
  mrnumber = {2011764},
  shortseries = {Comtemp. Math.},
  venue = {Grenoble}
}

@article{chanMcKayCorrespondenceSemisimple2018,
  title = {{{McKay}} Correspondence for Semisimple {{Hopf}} Actions on Regular Graded Algebras, {{I}}},
  author = {Chan, Kenneth and Kirkman, Ellen Elizabeth and Walton, Chelsea M. and Zhang, James Jian},
  date = {2018-08-15},
  journaltitle = {Journal of Algebra},
  shortjournal = {J. Algebra},
  volume = {508},
  pages = {512--538},
  doi = {10.1016/j.jalgebra.2018.05.008},
  url = {https://www.sciencedirect.com/science/article/pii/S0021869318303016},
  urldate = {2023-05-26},
  langid = {english},
  mrnumber = {3810305}
}

@article{chanMcKayCorrespondenceSemisimple2019,
  title = {{{McKay}} Correspondence for Semisimple {{Hopf}} Actions on Regular Graded Algebras. {{II}}},
  author = {Chan, Kenneth and Kirkman, Ellen Elizabeth and Walton, Chelsea M. and Zhang, James Jian},
  date = {2019-03-11},
  journaltitle = {Journal of Noncommutative Geometry},
  shortjournal = {J. Noncommut. Geom.},
  volume = {13},
  number = {1},
  pages = {87--114},
  doi = {10.4171/jncg/305},
  url = {https://ems.press/journals/jncg/articles/16076},
  urldate = {2025-07-17},
  langid = {english},
  mrnumber = {3941474}
}

@article{clineActionsQuantumLinear2020,
  title = {Actions of Quantum Linear Spaces on Quantum Algebras},
  author = {Cline, Zachary and Gaddis, Jason},
  date = {2020-08-15},
  journaltitle = {Journal of Algebra},
  shortjournal = {Journal of Algebra},
  volume = {556},
  pages = {246--286},
  doi = {10.1016/j.jalgebra.2020.03.011},
  url = {https://www.sciencedirect.com/science/article/pii/S0021869320301411},
  urldate = {2026-09-11},
  mrnumber = {4082335}
}

@book{faithAlgebraRingsModules1973,
  title = {Algebra: {{Rings}}, {{Modules}} and {{Categories I}}},
  author = {Faith, Carl},
  date = {1973},
  series = {Die {{Grundlehren}} Der Mathematischen {{Wissenschaften}}},
  number = {190},
  publisher = {Springer-Verlag},
  location = {Berlin, Germany},
  doi = {10.1007/978-3-642-80634-6},
  url = {http://link.springer.com/10.1007/978-3-642-80634-6},
  urldate = {2024-12-14},
  langid = {english},
  mrnumber = {0366960},
  pagetotal = {xxiii+565 pp.},
  shortseries = {Grundlehren Math. Wiss.}
}

@article{gaddisAuslandersTheoremPermutation2019,
  title = {Auslander's {{Theorem}} for Permutation Actions on Noncommutative Algebras},
  author = {Gaddis, Jason and Kirkman, Ellen and Moore, W. and Won, Robert},
  date = {2019-05},
  journaltitle = {Proceedings of the American Mathematical Society},
  shortjournal = {Proc. Amer. Math. Soc.},
  volume = {147},
  number = {5},
  pages = {1881--1896},
  doi = {10.1090/proc/14363},
  url = {https://www.ams.org/proc/2019-147-05/S0002-9939-2019-14363-X/},
  urldate = {2023-06-02},
  langid = {english},
  mrnumber = {3937667}
}

@article{hePreresolutionsNoncommutativeIsolated2022,
  title = {Preresolutions of Noncommutative Isolated Singularities},
  author = {He, Ji-Wei and Ye, Yu},
  date = {2022-04-06},
  journaltitle = {Pacific Journal of Mathematics},
  shortjournal = {Pacific J. Math.},
  volume = {316},
  number = {2},
  pages = {367--394},
  publisher = {Mathematical Sciences Publishers},
  doi = {10.2140/pjm.2022.316.367},
  url = {https://msp.org/pjm/2022/316-2/p06.xhtml},
  urldate = {2024-09-06},
  langid = {english},
  mrnumber = {4404004}
}

@book{jacBasicAlgebraII1989,
  title = {Basic {{Algebra}}},
  author = {Jacobson, Nathan},
  date = {1989},
  edition = {2},
  volume = {2},
  publisher = {{W. H. Freeman and Company}},
  location = {New York, New York},
  langid = {english},
  mrnumber = {1009787},
  pagetotal = {xviii+686}
}

@book{kashiwaraCategoriesSheaves2006,
  title = {Categories and {{Sheaves}}},
  author = {Kashiwara, Masaki and Schapira, Pierre},
  editor = {Berger, M. and Eckmann, B. and De La Harpe, P. and Hirzebruch, F. and Hitchin, N. and H\"ormander, L. and Knus, M.-A. and Kupiainen, A. and Lebeau, G. and Ratner, M. and Serre, D. and family=Sinai, given=Ya. G., given-i={{Ya}}G and Sloane, N.J.A. and Totaro, B. and Vershik, A. and Waldschmidt, M.},
  editortype = {redactor},
  date = {2006},
  series = {Grundlehren Der Mathematischen {{Wissenschaften}}},
  number = {332},
  publisher = {Springer},
  location = {Berlin, Heidelberg},
  doi = {10.1007/3-540-27950-4},
  url = {http://link.springer.com/10.1007/3-540-27950-4},
  urldate = {2024-12-16},
  mrnumber = {2182076},
  shortseries = {Grundlehren Math. Wiss.}
}

@book{lazarsfeldPositivityAlgebraicGeometry2004,
  title = {Positivity in {{Algebraic Geometry I}}},
  author = {Lazarsfeld, Robert Kendall},
  date = {2004},
  publisher = {Springer},
  location = {Berlin, Heidelberg},
  doi = {10.1007/978-3-642-18808-4},
  url = {https://link.springer.com/10.1007/978-3-642-18808-4},
  urldate = {2026-08-20},
  mrnumber = {2095471}
}

@book{leuschkeCohenMacaulayRepresentations2012,
  title = {Cohen--{{Macaulay Representations}}},
  author = {Leuschke, Graham J. and Wiegand, Roger},
  date = {2012-05-02},
  series = {Mathematical {{Surveys}} and {{Monographs}}},
  number = {181},
  publisher = {American Mathematical Society},
  location = {Providence, Rhode Island},
  doi = {10.1090/surv/181},
  url = {https://www.ams.org/surv/181},
  urldate = {2024-09-27},
  langid = {english},
  mrnumber = {2919145},
  pagetotal = {xviii+367},
  shortseries = {Math. Surveys Monogr.}
}

@article{liNoncommutativeResolutionsASGorenstein2026,
  title = {Noncommutative Resolutions of {{AS-Gorenstein}} Isolated Singularities},
  author = {Li, Hao-Nan and Shen, Men-Da and Wu, Quan-Shui},
  date = {2026-06-15},
  journaltitle = {Journal of Noncommutative Geometry},
  shortjournal = {J. Noncommut. Geom.},
  doi = {10.4171/jncg/670},
  url = {https://ems.press/journals/jncg/articles/14299817},
  urldate = {2026-09-30},
  langid = {english},
  publication-number = {670},
  note = {online first}
}

@online{liNoncommutativeResolutionsNoncommutative2026,
  title = {Noncommutative Resolutions of Noncommutative Isolated Singularities},
  author = {Li, Hao-Nan and Wu, Quan-Shui},
  date = {2026-09-18},
  eprint = {2609.21409},
  eprinttype = {arXiv},
  eprintclass = {math.RA},
  doi = {10.48550/arXiv.2609.21409},
  url = {http://arxiv.org/abs/2609.21409},
  urldate = {2026-09-24},
  pubstate = {prepublished}
}

@thesis{louInvariantTheoryRegular2016,
  type = {phdthesis},
  title = {Invariant Theory of AS Regular Algebras and Co-Poisson Structures},
  author = {Lou, Qi},
  date = {2016-04-01},
  institution = {Fudan University},
  location = {Shanghai, China},
  langid = {chinese},
  pagetotal = {vi+104}
}

@article{minamotoAmplenessTwosidedTilting2012,
  title = {Ampleness of Two-Sided Tilting Complexes},
  author = {Minamoto, Hiroyuki},
  date = {2012-01-01},
  journaltitle = {International Mathematics Research Notices},
  shortjournal = {Int. Math. Res. Not. IMRN},
  volume = {2012},
  number = {1},
  pages = {67--101},
  doi = {10.1093/imrn/rnr001},
  url = {https://doi.org/10.1093/imrn/rnr001},
  urldate = {2024-03-28},
  langid = {english},
  mrnumber = {2874928}
}

@book{monHopfAlgebrasTheir1993,
  title = {Hopf {{Algebras}} and {{Their Actions}} on {{Rings}}},
  author = {Montgomery, Susan},
  date = {1993-10-28},
  series = {{{CBMS Regional Conference Series}} in {{Mathematics}}},
  volume = {82},
  publisher = {American Mathematical Society},
  doi = {10.1090/cbms/082},
  langid = {english},
  mrnumber = {1243637}
}

@article{moriAmpleGroupAction2016,
  title = {Ample Group Action on {{AS-regular}} Algebras and Noncommutative Graded Isolated Singularities},
  author = {Mori, Izuru and Ueyama, Kenta},
  date = {2016-10},
  journaltitle = {Transactions of the American Mathematical Society},
  shortjournal = {Trans. Amer. Math. Soc.},
  volume = {368},
  number = {10},
  pages = {7359--7383},
  doi = {10.1090/tran/6580},
  url = {https://www.ams.org/tran/2016-368-10/S0002-9947-2015-06580-5/},
  urldate = {2023-05-24},
  langid = {english},
  mrnumber = {3471094}
}

@article{mullerQuotientCategoryMorita1974,
  title = {The Quotient Category of a {{Morita}} Context},
  author = {M\"uller, Bruno J},
  date = {1974-03-01},
  journaltitle = {Journal of Algebra},
  shortjournal = {J. Algebra},
  volume = {28},
  number = {3},
  pages = {389--407},
  doi = {10.1016/0021-8693(74)90048-9},
  url = {https://www.sciencedirect.com/science/article/pii/0021869374900489},
  urldate = {2024-12-04},
  langid = {english},
  mrnumber = {0447336}
}

@book{nastasescuGradedRingTheory1982,
  title = {Graded Ring Theory},
  author = {N\u ast\u asescu, Constantin and family=Oystaeyen, given=Freddy M. J., prefix=van, useprefix=true},
  date = {1982},
  series = {North-{{Holland Mathematical Library}}},
  number = {28},
  publisher = {North-Holland Publishing Co.},
  location = {Amsterdam, Holland},
  mrnumber = {0676974},
  pagetotal = {ix+340},
  shortseries = {North-Holland Math. Library}
}

@book{popescuAbelianCategoriesApplications1973,
  title = {Abelian Categories with Applications to Rings and Modules},
  author = {Popescu, Nicolae},
  date = {1973},
  series = {London {{Mathematical Society Monographs}}},
  number = {3},
  publisher = {Academic Press},
  location = {London, England},
  langid = {english},
  mrnumber = {0340375},
  pagetotal = {xii+467},
  shortseries = {London Math. Soc. Monogr.}
}

@article{popescuCaracterisationCategoriesAbeliennes1964,
  title = {Caract\'erisation des cat\'egories ab\'eliennes avec g\'en\'erateurs et limites inductives exactes},
  author = {Popescu, Nicolae and Gabriel, Pierre},
  date = {1964},
  journaltitle = {Comptes Rendus Hebdomadaires des S\'eances de l'Acad\'emie des Sciences},
  shortjournal = {C. R. Acad. Sci. Paris},
  volume = {258},
  number = {1},
  pages = {4188--4190},
  langid = {french},
  mrnumber = {0166241}
}

@article{qinNoncommutativeQuasiresolutions2019,
  title = {Noncommutative Quasi-Resolutions},
  author = {Qin, Xiao-Shan and Wang, Yan-Hua and Zhang, James Jian},
  date = {2019-10-15},
  journaltitle = {Journal of Algebra},
  shortjournal = {J. Algebra},
  volume = {536},
  pages = {102--148},
  doi = {10.1016/j.jalgebra.2019.07.015},
  url = {https://www.sciencedirect.com/science/article/pii/S0021869319303904},
  urldate = {2024-09-09},
  mrnumber = {3989612}
}

@article{serreFaisceauxAlgebriquesCoherents1955,
  title = {Faisceaux Alg\'ebriques Coh\'erents},
  author = {Serre, Jean-Pierre},
  date = {1955},
  journaltitle = {Annals of Mathematics. Second Series},
  shortjournal = {Ann. of Math. (2)},
  volume = {61},
  number = {2},
  pages = {197--278},
  publisher = {[Annals of Mathematics, Trustees of Princeton University on Behalf of the Annals of Mathematics, Mathematics Department, Princeton University]},
  doi = {10.2307/1969915},
  url = {https://www.jstor.org/stable/1969915},
  urldate = {2025-07-03},
  langid = {french},
  mrnumber = {0068874},
  shorthand = {FAC},
  sortkey = {FAC}
}

@article{smithCorrigendumMapsNonCommutative2016,
  title = {Corrigendum to ``{{Maps Between Non-Commutative Spaces}}''},
  author = {Smith, S. Paul},
  date = {2016-11},
  journaltitle = {Transactions of the American Mathematical Society},
  shortjournal = {Trans. Amer. Math. Soc.},
  volume = {368},
  number = {11},
  pages = {8295--8302},
  doi = {10.1090/tran/6908},
  url = {https://www.ams.org/tran/2016-368-11/S0002-9947-2016-06908-1/},
  urldate = {2024-10-22},
  langid = {english},
  mrnumber = {3546801}
}

@article{staffordHomologicalPropertiesGraded1994,
  title = {Homological {{Properties}} of ({{Graded}}) {{Noetherian PI Rings}}},
  author = {Stafford, J. Toby and Zhang, James Jian},
  date = {1994-09-15},
  journaltitle = {Journal of Algebra},
  shortjournal = {J. Algebra},
  volume = {168},
  number = {3},
  pages = {988--1026},
  doi = {10.1006/jabr.1994.1267},
  url = {https://www.sciencedirect.com/science/article/pii/S0021869384712671},
  urldate = {2023-04-09},
  langid = {english},
  mrnumber = {1293638}
}

@book{stenstromRingsQuotients1975,
  title = {Rings of {{Quotients}}},
  author = {Stenstr\"om, Bo},
  date = {1975},
  series = {Die {{Grundlehren}} Der Mathematischen {{Wissenschaften}}},
  number = {217},
  publisher = {Springer-Verlag},
  location = {Berlin, Germany},
  doi = {10.1007/978-3-642-66066-5},
  langid = {english},
  mrnumber = {0389953},
  pagetotal = {viii+309 pp.},
  shortseries = {Grundlehren Math. Wiss.}
}

@article{ueyamaNoncommutativeGradedAlgebras2015,
  title = {Noncommutative Graded Algebras of Finite {{Cohen}}--{{Macaulay}} Representation Type},
  author = {Ueyama, Kenta},
  date = {2015-09},
  journaltitle = {Proceedings of the American Mathematical Society},
  shortjournal = {Proc. Amer. Math. Soc.},
  volume = {143},
  number = {9},
  pages = {3703--3715},
  doi = {10.1090/proc/12527},
  url = {https://www.ams.org/proc/2015-143-09/S0002-9939-2015-12527-0/},
  urldate = {2023-05-31},
  langid = {english},
  mrnumber = {3359563}
}

@inproceedings{vandenberghNoncommutativeCrepantResolutions2004,
  title = {Non-Commutative {{Crepant Resolutions}}},
  booktitle = {The {{Legacy}} of {{Niels Henrik Abel}}},
  author = {family=Bergh, given=Michel, prefix=Van den, useprefix=true},
  editor = {Laudal, Olav Arnfinn and Piene, Ragni},
  date = {2004},
  pages = {749--770},
  publisher = {Springer},
  location = {Berlin, Heidelberg},
  doi = {10.1007/978-3-642-18908-1_26},
  url = {https://doi.org/10.1007/978-3-642-18908-1_26},
  urldate = {2024-01-02},
  eventdate = {2002},
  eventtitle = {Abel {{Bicentennial Conference}}},
  langid = {english},
  mrnumber = {2077594},
  shortauthor = {VdB},
  venue = {Oslo}
}

@article{vandenberghThreedimensionalFlopsNoncommutative2004,
  title = {Three-Dimensional Flops and Noncommutative Rings},
  author = {family=Bergh, given=Michel, prefix=Van den, useprefix=true},
  date = {2004-04},
  journaltitle = {Duke Mathematical Journal},
  shortjournal = {Duke Math. J.},
  volume = {122},
  number = {3},
  pages = {423--455},
  publisher = {Duke University Press},
  doi = {10.1215/S0012-7094-04-12231-6},
  url = {https://projecteuclid.org/journals/duke-mathematical-journal/volume-122/issue-3/Three-dimensional-flops-and-noncommutative-rings/10.1215/S0012-7094-04-12231-6.full},
  urldate = {2023-06-07},
  langid = {english},
  mrnumber = {2057015},
  shortauthor = {VdB}
}

@article{zhuAuslanderTheoremPI2023,
  title = {Auslander Theorem for {{PI Artin--Schelter}} Regular Algebras},
  author = {Zhu, Rui-Peng},
  date = {2023-09},
  journaltitle = {Proceedings of the American Mathematical Society},
  shortjournal = {Proc. Amer. Math. Soc.},
  volume = {151},
  number = {09},
  pages = {3705--3719},
  doi = {10.1090/proc/16424},
  url = {https://www.ams.org/proc/2023-151-09/S0002-9939-2023-16424-2/},
  urldate = {2025-05-17},
  langid = {english},
  mrnumber = {4607617}
}
\end{document}